\documentclass[12pt]{article}
\usepackage[a4paper,margin=0.92in]{geometry}
\usepackage[utf8]{inputenc}
\usepackage[T1]{fontenc}
\usepackage{amsmath,amssymb,amsfonts,amsthm,mathtools,bm,mathrsfs}
\usepackage{graphicx}
\usepackage{booktabs,longtable,array}
\usepackage{enumitem}
\usepackage{hyperref}
\usepackage{xcolor}
\usepackage{setspace}
\usepackage{placeins}
\usepackage{algorithm}
\usepackage{algpseudocode}
\usepackage{subfig}
\usepackage{float}
\usepackage{tikz}
\usetikzlibrary{arrows.meta,positioning,shapes.geometric}
\usepackage{caption}
\hypersetup{colorlinks=true,linkcolor=blue!55!black,citecolor=blue!55!black,urlcolor=blue!55!black}
\setlist[itemize]{leftmargin=2em}
\setlist[enumerate]{leftmargin=2.2em}
\graphicspath{{fig/}}

\newtheorem{theorem}{Theorem}[section]
\newtheorem{proposition}[theorem]{Proposition}
\newtheorem{lemma}[theorem]{Lemma}
\newtheorem{definition}[theorem]{Definition}
\newtheorem{remark}[theorem]{Remark}

\newtheorem{problem}{Problem}

\newcommand{\cH}{\mathcal H}

\newcommand{\norm}[1]{\lVert #1\rVert}
\newcommand{\HH}{\mathcal H}
\newcommand{\BH}{\mathcal B}
\newcommand{\GG}{\mathsf G}
\newcommand{\CC}{\mathsf C}
\newcommand{\VV}{\mathcal V}
\newcommand{\Id}{\operatorname{Id}}
\newcommand{\eps}{\varepsilon}

\title{A Statistical Imaging Framework for Inverse Obstacle Scattering and Its Resolution Analysis}
\author{Jinrui Zhang\thanks{\footnotesize Department of Mathematics, HKUST, Clear Water Bay, Kowloon, Hong Kong (jinruizhang@ust.hk).},\; Hai Zhang\thanks{\footnotesize Department of Mathematics, HKUST, Clear Water Bay, Kowloon, Hong Kong (haizhang@ust.hk).}
}
\date{}

\begin{document}
	\maketitle
	\begin{abstract}
We develop a general statistical framework for inverse obstacle imaging and resolution analysis from far-field second moments under random illumination. The regularized $Q$-statistic is introduced and utilized to measure the minimal statistically weighted cost of synthesizing the test far field
associated with a sampling point. Its reciprocal provides an imaging indicator. Under uniform full-aperture plane-wave illumination,
we relate this indicator to several classical sampling indicators and derive their statistical counterparts for general random
illumination. For the resolution analysis, separating the effects of radiation and noise from the statistics of induced boundary sources, we measure a correlation-free resolution limit by the one-sided half-height width of the boundary transition. In the regime of fixed wavenumber $k$ with the noise-calibrated regularization level $\rho_\tau\to 0$, we establish a sharp
asymptotic law for this width for simply connected domains with analytic boundaries. The width scales as 
\(h_\Gamma(x)\log\log(1/\rho_\tau)/\log(1/\rho_\tau)\), with $h_\Gamma(x)$ a geometric prefactor along the boundary determining both the boundary image contrast and the local transition width. For general boundary-source statistics, we examine how the source mean, variance, and correlation affect imaging quality. Finally, in representative scattering configurations, we provide qualitative interpretations of some typical imaging phenomena through the statistical structure of their induced boundary sources.  Numerical examples illustrate these mechanisms and indicate that practical resolution is governed jointly by source statistics, radiation, the noise level, rather than by wavelength alone.
	\end{abstract}
	
	
\section{Introduction}
Inverse wave scattering aims to recover inaccessible objects or material parameters from measurements of scattered waves. It underpins many important fields such as radar, sonar, geophysical exploration, nondestructive testing and medical imaging
\cite{coltonInverseAcousticElectromagnetic2019}. For inverse obstacle scattering, although classical uniqueness theory shows that ideal far-field data determine impenetrable obstacles under appropriate data and a priori assumptions \cite{ColtonSleeman1983,KirschGrinberg2008, StefanovUhlmann2005}, it does not ensure stable recovery of fine boundary features from noisy data \cite{Isakov1992,DiCristoRondi2003}. Moreover, many practical applications involve passive imaging, where the individual incident fields are unknown and only statistical information about them is available. These challenges call for a unified framework that produces physically
interpretable images and quantifies how resolution depends on the noise
level. In particular, such a framework should show when imaging below
the classical half-wavelength diffraction limit is possible, namely
when the noise level is sufficiently low. In this paper, we develop such
a statistical framework, using two-dimensional sound-soft obstacle
scattering as a concrete setting.

Classical inverse scattering methods with known incident fields include
regularized nonlinear reconstruction
\cite{coltonInverseAcousticElectromagnetic2019,ColtonKress2018Review} and
qualitative imaging methods based on sampling, subspace and migration
techniques \cite{houDirectImagingMethod2007,ItoJinZou2012,ChenChenHuang2013}.
In the passive setting considered here, second-order (correlation)
measurements have been used to extend the linear sampling method to random
point-source illumination \cite{GarnierHaddarMontanelli2023} and to develop a
direct imaging method based on double cross-correlations and the
Helmholtz--Kirchhoff identity \cite{YinYan2026}. Under suitable Gaussian
assumptions on the random sources, near-field correlation data uniquely
determine both the source strength and the obstacle geometry
\cite{HohageLiu2025}.

Despite much progress in reconstruction methods and uniqueness results in the passive setting, resolution analysis that accounts for measurement noise is
much less developed, and the resolution limit of inverse obstacle scattering
is still poorly understood. Existing resolution analyses in inverse
scattering follow two main approaches.

The first approach is based on estimation and sensitivity analysis.
Cram\'er--Rao bounds measure how accurately parameters can be estimated
using derivatives of the forward map. They typically assume a prescribed
finite-dimensional parametrization of the target and a specified
statistical noise model \cite{YeBreslerMoulin2001,DiongRoueffLasayguesLitman2015}.
Sensitivity analyses have also been carried out for multistatic imaging of
extended targets in the Born and high-frequency regimes
\cite{AmmariEtAl2012Multistatic}. Related analyses based on shape
derivatives describe how scattering data respond to boundary perturbations
\cite{Kirsch1993,Hettlich1995}. These works mainly quantify local,
linearized parameter uncertainty. 

The second approach measures the widths of point-spread functions, or
localization kernels, in correlation-based direct imaging. For direct
sampling and reverse-time migration, these kernels often have Bessel- or
sinc-type profiles \cite{Liu2017NovelSampling,ChenChenHuang2013}. Their
widths depend on the wavelength and the aperture. In passive
correlation-based migration, they also depend on the coherence of the
sources \cite{GarnierPapanicolaou2010}.

The recent theory of the computational resolution limit studies the
reconstruction of point sources from band-limited data
\cite{LiuZhang2021,liuMathematicalTheoryComputational2021,LiuZhang2022}.
It quantifies the minimum separation between sources needed to detect
the number of sources and to recover their locations, in terms of the
noise level, the bandwidth and the number of sources. In particular, it
shows that point sources can be resolved below the classical diffraction
limit when the signal-to-noise ratio is sufficiently high.
A main goal of this paper is to develop a theory in the same spirit for
inverse obstacle scattering with unknown incident fields. We show that
super-resolution is achievable under appropriate statistical assumptions
on the incident fields and a sufficiently low noise level.

Our starting point is the far-field second-moment operator $\Sigma=\mathbb E[y\otimes y^*]$, where \(y=y(\omega,\hat{x})\) is a random far-field pattern with $\omega$ being the random parameter, $y \otimes y^{*}$ denotes the rank-one operator on $L^{2}(\mathbb{S}^{1})$
defined by
$(y \otimes y^{*})\, g := \langle g, y \rangle_{L^{2}(\mathbb{S}^{1})}\, y,
g \in L^{2}(\mathbb{S}^{1})$, 
and the expectation is
taken over the random illumination ensemble. 
For the test pattern $\phi_z(\hat x)=e^{-ik\hat x\cdot z}$, define
\[Q_\tau(z):=\big\langle \phi_z,(\Sigma+\tau I)^{-1}\phi_z\big\rangle_{L^{2}(\mathbb{S}^{1})},\]
with the regularization level $\tau$ calibrated to the noise in the second moment. 
	
The main contributions of this paper can be summarized as follows.
	\begin{enumerate}[label=(\roman*)]
	\item We establish a minimal regularized synthesis-cost interpretation of the $Q$-statistic and adopt its reciprocal as an imaging indicator. Under general illumination, a common second-moment reformulation of classical sampling indicators provides their statistical counterparts, including the factorization method (FM), multiple signal classification (MUSIC), and direct sampling method (DSM).

\item We show that, in the noise-free case with random plane-wave incidence,
the second-moment operator $\Sigma$ uniquely determines the obstacle,
provided that the distribution of the incident directions and the wavenumber $k$ satisfy appropriate conditions. The proof is based on a range
test.

\item In the noisy case, we introduce the correlation-free resolution limit, separating the effect of noise and radiation from the resolution loss caused by source correlation. Measured by the one-sided half-height width of the boundary transition, the noise-dependent soft spectral cutoff yields a width scaling as
\(h_\Gamma(x)\log\log(1/\rho_\tau)/\log(1/\rho_\tau)\) for simply connected domains with analytic boundaries, with $h_\Gamma(x)$ a geometric prefactor along the boundary, as the noise-calibrated parameter $\rho_\tau\to 0$.
		\item We relate general statistics of the boundary source to far-field modal accessibility and characterize the resolution mechanism due to the source mean, variance, and correlation, which supports qualitative explanations of imaging  behavior in several representative scattering settings such as shadow-side degradation, loss of detail in concave regions, gap blurring between nearby obstacles and resolution gain with point-source illumination, illustrated by numerical experiments.
	\end{enumerate}
	
The remainder of this paper is organized as follows. 	Section~\ref{sec:mathematical-statistical-formulation} formulates the scattering and second-moment models. Section~\ref{sec:q-distance-resolution} develops the $Q$-statistic, its sampling-method connections, and the correlation-free resolution analysis. Section~\ref{sec:source-second-moment-imaging} studies the effects of boundary-source statistics, and Section~\ref{sec:physical-statistics-consequences} interprets them in physical scattering configurations. Section~\ref{sec:numerical-experiments} presents numerical experiments with prescribed and scattering-induced sources. Section~\ref{connnn} concludes this paper. Appendices~\ref{appendix-point source}--\ref{app:heat-kernel-model} provide the random point-source illumination model, the disk estimates needed for Theorem \ref{thm:disk-halfheight-resolution}, the complete sharp analytic-boundary asymptotic proof of Theorem \ref{thm:analytic-boundary-resolution} and the diffusion-based correlation model, respectively.
	
	\section{Mathematical and Statistical Formulation}
	\label{sec:mathematical-statistical-formulation}
	
	This section first recalls the two-dimensional sound-soft scattering problem and
	its equivalent induced-source representation, then introduces random
	plane-wave illumination as an example, derives the second-moment factorization, and
	formulates the general statistical inverse problems.
	Figure~\ref{fig:scattering-setup} summarizes the measurement setting for
	a general random incident field.

	\begin{figure}[htbp]
		\centering
		\includegraphics[width=0.8\linewidth]{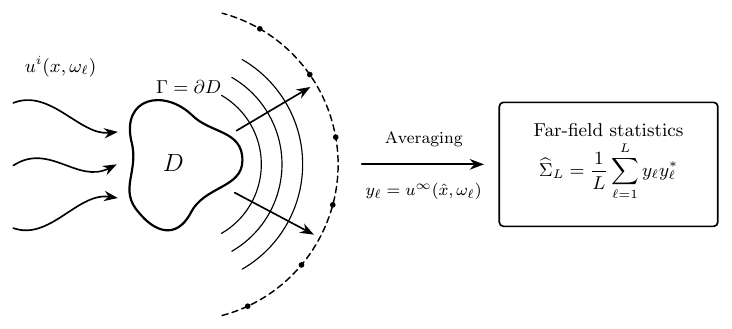}
		\caption{Statistical obstacle imaging. Each realization of the incident
			field excites the fixed obstacle \(D\), producing a measured far-field
			sample \(y_\ell=u^\infty(\hat x,\omega_\ell)\). Averaging sample outer
			products gives the far-field second moment.}
		\label{fig:scattering-setup}
	\end{figure}
	
	\subsection{Deterministic obstacle scattering}
	\label{subsec:forward-scattering}
	
	Let $D=\bigcup_{j=1}^{M}D_j\subset\mathbb R^2$ be a collection of bounded sound-soft obstacles with connected exterior and sufficiently smooth boundary
	$\Gamma=\partial D=\bigcup_{j=1}^{M}\Gamma_j$.  Given the incident field $u^i$, the scattered field is
	the unique solution of
	\begin{align}
		\begin{cases}
			\Delta u^s+k^2u^s=0 &\text{in }\mathbb R^2\setminus\overline D,\\
			u^s=-u^i &\text{on }\Gamma,
		\end{cases}
	\end{align}
	with the Sommerfeld radiation condition \(\sqrt r(\partial_r u^s-iku^s)\to0\) as $r=|x|\to\infty$, uniformly in
	$\hat x=x/|x|$.  The total field is $u=u^i+u^s$.  Its scattered part has the
	far-field asymptotics
	\begin{equation}
		u^s(x)=\frac{e^{ikr}}{\sqrt r}
		\left(u^\infty(\hat x)+\mathcal O(r^{-1})\right),
		\qquad r\to\infty,
		\label{eq:farfield-asymptotics}
	\end{equation}
	which defines $u^\infty\in C^\infty(\mathbb S^1)$.  Standard well-posedness
	and far-field properties may be found in
	\cite{coltonInverseAcousticElectromagnetic2019,CakoniColtonHaddar2016}.  We consider the two-dimensional sound-soft problem in this paper while the statistical mechanism below can be naturally extended
	to three dimensions and to other boundary conditions.
	
	\subsection{Induced boundary-source representation}
	\label{subsec:induced-boundary-source}
	
	This subsection introduces the induced boundary-source representation
	that links the obstacle scattering problem to the source-based statistical formulation developed below. Let $\Phi_k(x,y)=\frac{i}{4}H_0^{(1)}(k|x-y|)$ be the outgoing
	fundamental solution in two dimensions.  Representing the scattered field as
	\begin{equation}
		u^s(x)=\int_\Gamma \Phi_k(x,y)\varphi(y)\,ds(y),
		\qquad x\in\mathbb R^2\setminus\Gamma,
		\label{eq:single-layer-representation}
	\end{equation}
	introduces an equivalent boundary density $\varphi$, which we call the
	\emph{induced boundary source}.  With
	\[(S_\Gamma q)(x):=\int_\Gamma\Phi_k(x,y)q(y)\,ds(y),
		\qquad x\in\Gamma,\]
	the boundary condition gives $S_\Gamma\varphi=-u^i|_\Gamma$.  The operator
	$S_\Gamma:L^2(\Gamma)\to H^{1}(\Gamma)$ is invertible when $k^2$ is
	not an interior Dirichlet eigenvalue. We call this the \textbf{nonresonance condition} and assume it from now on. Then
	\[\varphi=-S_\Gamma^{-1}(u^i|_\Gamma).\]
	At a resonance, a combined-field representation  retains the second-moment factorization structure with the corresponding boundary operator and radiation map.
	
	The boundary-to-far-field map is
	\begin{equation}
		(A_\Gamma q)(\hat x):=\gamma_k\int_\Gamma
		e^{-ik\hat x\cdot y}q(y)\,ds(y),
		\qquad \hat x\in\mathbb S^1,
		\label{eq:boundary-to-farfield-operator}
	\end{equation}
	where $\gamma_k=e^{i\pi/4}/\sqrt{8\pi k}$.  Hence
	\begin{equation}
		u^\infty=A_\Gamma\varphi
		=-A_\Gamma S_\Gamma^{-1}(u^i|_\Gamma).
		\label{eq:far-field-Aphi}
	\end{equation}
	
	\begin{remark}[Inverse source versus inverse scattering]
		Equation~\eqref{eq:far-field-Aphi} also gives
		$u^\infty(\hat x)=\gamma_k\widehat{\varphi\,ds}(k\hat x)$, up to the Fourier
		transform convention.  It therefore connects inverse obstacle  scattering with inverse source problems.  The distinction is
		essential.  In an inverse source problem the source density is an active source
		prescribed independently of the propagation medium, whereas in the inverse scattering problem considered here, the source density $\varphi$ is induced by
		the incident field through the obstacle boundary condition and changes with
		$\Gamma$.
	\end{remark}
	
	\subsection{Random plane-wave incidence and the statistical formulation}
	\label{subsec:statistical-second-moment-problem}
	
	This subsection introduces random plane-wave illumination and the resulting statistical formulation, which will serve as a typical example for the statistical imaging framework developed subsequently. Specializing the illumination to plane waves
	$u_d^i(x)=e^{ikd\cdot x}$, $d\in\mathbb S^1$, let the direction $d$ be random
	with probability density $\mu\in L^\infty(\mathbb S^1)$ relative to arc-length
	measure $d\sigma$, so that $\mu\ge0$ and
	$\int_{\mathbb S^1}\mu\,d\sigma=1$.  Denote the induced boundary source
	and measured far field by $\varphi_d=-S_\Gamma^{-1}[u_d^i|_\Gamma]$ and
	$y_d=A_\Gamma\varphi_d$, respectively.

	The incident trace has the second-moment operator
	$\mathcal{K}_{\mu,\Gamma}$ whose kernel is
	\begin{equation}
		K_{\mu,\Gamma}(x,y)
		=\int_{\mathbb S^1}e^{ikd\cdot(x-y)}\mu(d)\,d\sigma(d),
		\qquad x,y\in\Gamma.
		\label{eq:incident-trace-second-moment}
	\end{equation}
	For uniform full-aperture illumination, $\mu=(2\pi)^{-1}$ and
	$K_{\mu,\Gamma}(x,y)=J_0(k|x-y|)$.  
	
	Linearity propagates the second moment through the scattering system.  On the
	boundary, the induced boundary sources have the second moment
	\begin{equation}
		\mathcal{K}_{s,\Gamma}
		:=\mathbb E_\mu[\varphi_d\otimes\varphi_d^*]
		=S_\Gamma^{-1}\mathcal{K}_{\mu,\Gamma}S_\Gamma^{-*},
		\label{eq:source-second-moment-sandwich}
	\end{equation}
	where $S_\Gamma^{-*}:L^2(\Gamma)\to H^{-1}(\Gamma)$ denotes
the adjoint of $S_\Gamma^{-1}:H^1(\Gamma)\to L^2(\Gamma)$
with respect to the $L^2$-based Sobolev duality.
Here $\mathcal K_{\mu,\Gamma}$ extends boundedly from $H^{-1}(\Gamma)$ to $H^1(\Gamma)$.  This yields the far-field second-moment operator
	\begin{equation}
		\Sigma_{D,\mu}
		:=\mathbb E_\mu[y_d\otimes y_d^*]
		=A_\Gamma \mathcal{K}_{s,\Gamma}A_\Gamma^*
		=A_\Gamma S_\Gamma^{-1}\mathcal{K}_{\mu,\Gamma}
		S_\Gamma^{-*}A_\Gamma^*.
		\label{eq:farfield-second-moment-sandwich}
	\end{equation}
	Equations~\eqref{eq:source-second-moment-sandwich} and
	\eqref{eq:farfield-second-moment-sandwich} are the source-side and data-side
	factorizations of the statistical forward problem. Their two-stage
	structure is illustrated in Figure~\ref{fig:covariance-flow}.

	\begin{figure}[htbp]
		\centering
		\includegraphics[width=0.96\linewidth]{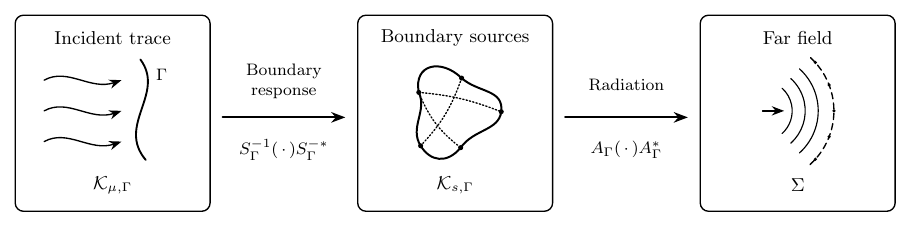}
		\caption{Statistical forward chain. The incident-trace second moment
			\(\mathcal K_{\mu,\Gamma}\) is mapped to the boundary-source second
			moment \(\mathcal K_{s,\Gamma}\) by
			\(\mathcal K\mapsto S_\Gamma^{-1}\mathcal KS_\Gamma^{-*}\), and then to the far-field
			second moment \(\Sigma\) by \(\mathcal K\mapsto A_\Gamma \mathcal K A_\Gamma^*\).
			Dotted links on the
			boundary schematically represent source correlations.}
		\label{fig:covariance-flow}
	\end{figure}

	The factorization also separates the physics.  The incident density $\mu$ fixes
	the directional coherence, $S_\Gamma^{-1}$ encodes the boundary condition,
	geometry, and multiple scattering, and $A_\Gamma$ determines which induced
	source modes reach the far field.  The smoothing kernel of $A_\Gamma$ strongly attenuates highly oscillatory boundary modes  that may carry fine-scale information.

We note that the far-field second-moment operator $\Sigma_{D,\mu}$ is closely related to the
far-field operator used in active imaging.
Recall that the far-field operator $F\colon L^2(\mathbb S^1)\to L^2(\mathbb S^1)$ is defined by
\begin{equation}
  (Fg)(\hat x) := \int_{\mathbb S^1} u_D^\infty(\hat x,d)\,g(d)\,d\sigma(d),
  \qquad g\in L^2(\mathbb S^1),
  \label{eq:farfield-operator}
\end{equation}
see \cite{coltonInverseAcousticElectromagnetic2019}. Its adjoint is
$(F^*h)(d)=\int_{\mathbb S^1}\overline{u_D^\infty(\hat y,d)}\,h(\hat y)\,d\sigma(\hat y)$.

Let the incident direction $d$ be random with probability density
$\mu\in L^\infty(\mathbb S^1)$, as introduced above, and let
$M_\mu\colon L^2(\mathbb S^1)\to L^2(\mathbb S^1)$, $(M_\mu g)(d):=\mu(d)\,g(d)$,
be the corresponding multiplication operator. Then the second-moment kernel is
\begin{equation}
  \Sigma_{D,\mu}(\hat x,\hat y)
  := \mathbb E\!\left[u_D^\infty(\hat x,d)\,\overline{u_D^\infty(\hat y,d)}\right]
  = \int_{\mathbb S^1} u_D^\infty(\hat x,d)\,\mu(d)\,
    \overline{u_D^\infty(\hat y,d)}\,d\sigma(d).
  \label{eq:second-moment-kernel}
\end{equation}
It follows that
\begin{equation}
  \Sigma_{D,\mu} = F M_\mu F^* .
  \label{eq:second-moment-kernel-FmuFstar}
\end{equation}
For uniform illumination, $\mu\equiv 1/(2\pi)$, this reduces to
\begin{equation}
  2\pi\,\Sigma_{D,\mathrm{unif}} = FF^*,
  \label{eq:second-moment-ffstar}
\end{equation}
which is the time-reversal operator \cite{HazardRamdani2004} that underlies
MUSIC- and factorization-type imaging methods.

The following proposition shows that second-moment data retain
the classical ideal range information under full-aperture random plane-wave illumination.
	
	\begin{proposition}
		\label{thm:nonuniform-range-equivalence} 
        Assume the directional density $\mu$ satisfies
		\begin{equation}
			0<c\le \mu(d)\le C<\infty
			\quad\text{for a.e. }d\in\mathbb S^1.
			\label{eq:illumination-bounds}
		\end{equation}
		Then $\Sigma_{D,\mu}$ defined in \eqref{eq:second-moment-kernel-FmuFstar} satisfies
		\begin{equation}
			cFF^*\le\Sigma_{D,\mu}\le CFF^*,
			\qquad
			\operatorname{Ran}(\Sigma_{D,\mu}^{1/4})
			=\operatorname{Ran}((FF^*)^{1/4}).
			\label{eq:range-invariance-second-moment}
		\end{equation}
	\end{proposition}
	\begin{proof}
Since \(cI\le M_\mu\le CI\) and
\(\Sigma_{D,\mu}=FM_\mu F^*\), we obtain
\(cFF^*\le\Sigma_{D,\mu}\le CFF^*\).
The L\"owner--Heinz inequality \cite{Furuta2001} yields
\[
\sqrt c\,(FF^*)^{1/2}
\le \Sigma_{D,\mu}^{1/2}
\le \sqrt C\,(FF^*)^{1/2}.
\]
Applying Douglas' range-inclusion lemma \cite{Douglas1966}
to \(S=\Sigma_{D,\mu}^{1/4}\) and \(T=(FF^*)^{1/4}\),
whose products satisfy
\(\sqrt c\,TT^*\le SS^*\le\sqrt C\,TT^*\),
gives both range inclusions and hence
\(\operatorname{Ran}S=\operatorname{Ran}T\).
\end{proof}
	
    The lower bound in \eqref{eq:illumination-bounds} ensures uniformly positive illumination over the full aperture while the upper bound
	excludes singular concentration of illumination power. Under this assumption, the obstacle is determined from far-field second moments through the classical range test in the factorization method \cite{coltonInverseAcousticElectromagnetic2019}.
	
	\begin{theorem}
	    	\label{cor:second-moment-identifiability}
		Under the assumptions of Proposition \ref{thm:nonuniform-range-equivalence} and the nonresonance condition, with
		$\phi_z(\hat x)=e^{-ik\hat x\cdot z}$,
		\[
		z\in D
		\quad\Longleftrightarrow\quad
		\phi_z\in\operatorname{Ran}(\Sigma_{D,\mu}^{1/4}).
		\]
		Consequently, $\Sigma_{D,\mu}$ uniquely determines the sound-soft obstacle $D$.
	\end{theorem}

This uniqueness result concerns exact data and does not imply the resolution attainable in the presence of noise. Even when \(\mu\) remains uniformly positive, more strongly illuminated boundary regions tend to be easier to image, whereas weakly illuminated regions may be difficult to distinguish from noise.  Thus illumination distributions with the same ideal identifiability can yield different image contrast and spatial resolution, as examined in the following sections.
    
	\subsection{General statistical inverse obstacle problem}
	\label{subsec:statistical-inverse-problem}
The plane-wave model above is one particular choice of illumination ensemble.
More generally, let $(\Omega,\mathcal F,\mathbb P)$ be a probability space, with
$\omega\in\Omega$ the random parameter that drives the illumination, and let
\begin{equation}
  g(\omega) := u^i(\omega)\big|_\Gamma \in H^{1}(\Gamma)
  \label{eq:random-trace}
\end{equation}
be the random incident trace on $\Gamma=\partial D$. We assume
$\mathbb E\|g\|_{H^{1}(\Gamma)}^2<\infty$, so that the second-moment operator for the trace of incident fields on $\Gamma$
\begin{equation}
  \mathcal K_{i,\Gamma} := \mathbb E\,[\,g\otimes g^*\,]
  \label{eq:incident-second-moment}
\end{equation}
is well defined. We then define
the induced boundary source density and the corresponding far-field pattern by
\begin{equation}
  \varphi(\omega) := -S_\Gamma^{-1} g(\omega) \in L^2(\Gamma),
  \qquad
  u^\infty(\omega) := A_\Gamma \varphi(\omega) \in L^2(\mathbb S^1).
  \label{eq:random-density-farfield}
\end{equation}
Plane-wave illumination is the special case
$g(\omega)=e^{ik\,x\cdot d(\omega)}\big|_{x\in\Gamma}$, where the incident
direction $d(\omega)\in\mathbb S^{1}$ is random with density $\mu$. The
second-moment operators of the induced boundary sources and of the
far-field data,
\[
  \mathcal K_{s,\Gamma}:=\mathbb E[\varphi\otimes\varphi^*],
  \qquad
  \Sigma:=\mathbb E[u^\infty\otimes(u^\infty)^*],
\]
factorize as
\begin{equation}
  \mathcal K_{s,\Gamma} = S_\Gamma^{-1}\,\mathcal K_{i,\Gamma}\,S_\Gamma^{-*},
  \qquad
  \Sigma = A_\Gamma\,\mathcal K_{s,\Gamma}\,A_\Gamma^{*}.
  \label{eq:second-moment-factorization}
\end{equation}
For plane-wave illumination,
\eqref{eq:second-moment-factorization} reduces to
$\Sigma=\Sigma_{D,\mu}=F M_\mu F^*$, where $M_\mu$ denotes multiplication
by $\mu$. See also Appendix~\ref{appendix-point source} for near-field point-source illumination.

In practice, $\Sigma$ is not available. Instead, we observe $L$ noisy far-field
measurements
\begin{equation}
  y_\ell^\delta = u^\infty(\omega_\ell) + \eta_\ell,
  \qquad \ell=1,\dots,L,
  \label{eq:noisy-data}
\end{equation}
where $\omega_1,\dots,\omega_L$ are independent realizations of the illumination and
$\eta_\ell$ denotes measurement noise of level $\delta$. From these data we form the
empirical second moment
\begin{equation}
  \widehat\Sigma_L := \frac1L\sum_{\ell=1}^{L} y_\ell^\delta\otimes (y_\ell^\delta)^* .
  \label{eq:empirical-second-moment}
\end{equation}
The estimate $\widehat\Sigma_L$ differs from $\Sigma$ because of the
sampling error due to finite realizations, and the measurement noise. 

Within a prescribed class of admissible illuminations, we address two complementary
problems.
	
	\begin{problem}[Data-driven statistical imaging]
		\label{prob:statistical-inverse-obstacle}
		Given the wavenumber $k>0$, observation aperture and measured second moment \(\widehat\Sigma\), localize the obstacle boundary $\Gamma$. 
	\end{problem}

    \begin{problem}[Resolution analysis]
    \label{prob:statistical-resolution}
    Given a reference obstacle, an illumination ensemble,
    the wavenumber, and a noise level, develop a framework
    for assessing the indicator-based resolution limit and the statistical imaging mechanism.
\end{problem}

To address both problems, we introduce a regularized $Q$-statistic that can be
evaluated directly from the far-field second moment. It can equivalently be
interpreted as the minimal statistical cost of a  source whose radiated far field matches a prescribed
test pattern. This interpretation links data-driven imaging to the
boundary-source representation \eqref{eq:second-moment-factorization}, and it serves
as a common starting point for both reconstruction and resolution analysis.

 The resolution analysis is complicated by the fact that the observed second moment
combines several effects at once: measurement noise, illumination statistics,
boundary geometry, scattering (including multiple scattering), and radiation to the
far field. Noise enters only through the empirical estimate $\widehat\Sigma_L$. The
factorization \eqref{eq:second-moment-factorization} splits the noise-free second
moment $\Sigma$ into a radiation part $A_\Gamma$ and a source part 
$\mathcal K_{s,\Gamma}$ that combines the
scattering effect and the
illumination statistics.
We therefore first consider a correlation-free boundary-source limit, in which $\mathcal K_{s,\Gamma} \propto I$. Then
$\Sigma \propto A_\Gamma A_\Gamma^*$, so the data depend on the scatterer only
through the support $\Gamma$ of the radiating source. In this setting, we characterize resolution by how sharply the imaging
functional changes across $\Gamma$. We then reintroduce the structure of
$\mathcal K_{s,\Gamma}$ and examine how the coherent mean, a spatially varying
fluctuation variance and the spatial correlation of the boundary source each change
the imaging behavior. Finally, we use the induced boundary-source statistics to give
qualitative interpretations of the images in representative scattering
configurations.

	\section{Resolution Analysis via the \texorpdfstring{\(Q\)}{Q}-Statistic}
	\label{sec:q-distance-resolution}
	
	In this section, we define the \(Q\)-statistic and its regularized form \(Q_\tau\) from the far-field second moment $\Sigma$, interpret them variationally as minimal statistically weighted costs of synthesizing the test far-field signal, and use \(I_{Q,\tau}=1/Q_\tau\) for imaging. The relation \(\Sigma\propto FF^*\) under uniform plane-wave illumination reinterprets several classical sampling methods as distinct spectral weighting rules within a common statistical signal space, thus yielding their statistical counterparts for general random illumination.

A comparison of plane-wave illumination with illumination by near-field point sources
shows that the latter can resolve the boundary below the classical
half-wavelength limit, highlighting the influence of the boundary-source
statistics on resolution. We introduce a correlation-free boundary-source limit to characterize the noise-dependent resolution by the one-sided
half-height width of the boundary transition of the imaging functional, and derive its
asymptotics for disks and simply connected domains with analytic boundaries.

	\subsection{Statistical synthesis cost and the \texorpdfstring{\(Q\)}{Q}-statistic}
	\label{subsec:exact-q-distance}
	
	This subsection introduces the statistical synthesis cost and the associated $Q$-statistic, which quantify the cost of reproducing a test far-field pattern from the statistical signal space. Let $\{(\lambda_j,u_j)\}$ denote the positive eigenpairs of $\Sigma$, $\Sigma u_j=\lambda_j u_j$ with orthonormal $u_j$ and eigenvalues $\lambda_j$ arranged in nonincreasing order. Throughout this section, let $X_\Gamma=L^2(\Gamma)$ and $H=L^2(\mathbb S^1)$. The test far field corresponding to a point source at $z$ is $\phi_z(\hat x)=e^{-ik\hat x\cdot z}\in H$.  For brevity, write $\mathcal K:=\mathcal K_{s,\Gamma}$ the second-moment operator of the induced boundary sources defined on $X_\Gamma$, so that $\Sigma=A_\Gamma \mathcal K A_\Gamma^*$.

	\begin{definition}
    For $f\in H$, define the $\Sigma$-weighted synthesis cost by
    \[
    \mathcal E_0(f,\Sigma):=
    \begin{cases}
        \displaystyle
        \sum_{\lambda_j>0}
        \frac{|\langle f,u_j\rangle_H|^2}{\lambda_j},
        & f\in\operatorname{Ran}(\Sigma^{1/2}),\\[1ex]
        +\infty,
        & f\notin\operatorname{Ran}(\Sigma^{1/2}).
    \end{cases}
    \]
    The exact $Q$-statistic is
    \[
    Q_0(z,\Sigma):=\mathcal E_0(\phi_z,\Sigma),
    \qquad z\in\mathbb R^2.
    \]
\end{definition}
Here, synthesis means reproducing a prescribed test far field from the
statistical signal space, with the cost measured in the metric induced
by $\Sigma$. The inverse weights $\lambda_j^{-1}$ assign a larger cost
to components along weakly excited signal directions.
For $f\in\operatorname{Ran}(\Sigma^{1/2})$, this cost can
equivalently be written as
$\mathcal E_0(f,\Sigma)=\|\Sigma^{\dagger/2}f\|_H^2$,
where $\Sigma^\dagger$ is the Moore--Penrose inverse.

To give this data-space cost a source-side interpretation,
define the associated source-energy space and extended norm by
\begin{equation}
    \mathcal H_{\mathcal K}
    :=\operatorname{Ran}(\mathcal K^{1/2}),
    \qquad
    \|q\|_{\mathcal K}^2:=
    \begin{cases}
        \|\mathcal K^{\dagger/2}q\|_{X_\Gamma}^2,
        & q\in\mathcal H_{\mathcal K},\\
        +\infty,
        & q\notin\mathcal H_{\mathcal K}.
    \end{cases}
    \label{eq:source-covariance-norm}
\end{equation}
The following proposition shows that the data-space synthesis cost
equals the minimal source energy, weighted by the source statistics,
needed to reproduce the prescribed far field.

	\begin{proposition}
		\label{prop:minimum-second-moment-energy}
		Under the factorization $\Sigma=A_\Gamma \mathcal K A_\Gamma^*$, for every $f\in H$,
		\begin{equation}
			\mathcal E_0(f,\Sigma)=\min_{A_\Gamma q=f}\|q\|_{\mathcal K}^2.
			\label{eq:exact-mimic-minimum}
		\end{equation}
		Both quantities equal $+\infty$ when
		$f\notin\operatorname{Ran}(\Sigma^{1/2})$.  
	\end{proposition}
\begin{proof}
Set $B=A_\Gamma \mathcal K^{1/2}$, so that $\Sigma=BB^*$.
Since the source energy is infinite outside
$\mathcal H_{\mathcal K}$, restricting the infimum to this
space does not change its value.
The minimum-preimage characterization
\[\|q\|_{\mathcal K}^2=\min_{\mathcal K^{1/2}w=q}\|w\|_{X_\Gamma}^2,\qquad q\in\mathcal H_{\mathcal K},\]
therefore gives
\[\min_{\substack{q\in X_\Gamma\\A_\Gamma q=f}}\|q\|_{\mathcal K}^2=\min_{Bw=f}\|w\|_{X_\Gamma}^2.\]
Since $BB^*=\Sigma$, the positive singular values of $B$
are $\sqrt{\lambda_j}$, with corresponding left singular
vectors $u_j$. Note that \(\operatorname{Ran}(B)=\operatorname{Ran}(\Sigma^{1/2}).\) For $f$ in this range, the singular-value expansion of $B$
yields a minimum-norm solution $w^\dagger$ satisfying
\[
    \min_{Bw=f}\|w\|_{X_\Gamma}^2
    =
    \|w^\dagger\|_{X_\Gamma}^2
    =
    \sum_{\lambda_j>0}
    \frac{|\langle f,u_j\rangle_H|^2}{\lambda_j}
    =
    \mathcal E_0(f,\Sigma).
\]
\end{proof}
	
    \begin{remark}[Statistically weighted source energy]
Under the nonresonance condition, $A_\Gamma$ is injective
(see Section~\ref{subsec:correlation-free-resolution}).
Thus, whenever $A_\Gamma q=\phi_z$ is solvable, its
solution $q_z$ is unique and
$Q_0(z,\Sigma)=\|q_z\|_{\mathcal K}^2$. 
The source statistics therefore
set the metric for the source cost.
Modes that carry more statistical power (larger eigenvalues) cost
less. In the correlation-free case \(\mathcal K=\kappa I\), no mode is
preferred, and the cost becomes \(\kappa^{-1}\|q\|_{L^2(\Gamma)}^2\). This lets
us study radiation and noise separately from source-statistical weighting.
\end{remark}
	
	For practical imaging, the exact cost requires regularization since finite-energy synthesis may be infeasible and the inverse weights $\lambda_j^{-1}$ are sensitive to errors in the measured second moment. We therefore replace exact synthesis by a penalized approximation.

    For $\tau>0$, define the regularized synthesis cost and the associated $Q$-statistic by
	\begin{equation}
		\mathcal E_\tau(f,\Sigma)
		:=\left\langle f,(\Sigma+\tau I)^{-1}f\right\rangle_H,
		\qquad
		Q_\tau(z,\Sigma):=\mathcal E_\tau(\phi_z,\Sigma).
		\label{eq:regularized-mimic-energy}
	\end{equation}
    The following variational identity makes this relaxation precise.
	
	\begin{proposition}[Source--signal variational formula]
		\label{thm:regularized-mimic-variational}
		Under the factorization $\Sigma=A_\Gamma \mathcal K A_\Gamma^*$, for every $f\in H$ and
		$\tau>0$,
		\begin{equation}
			\mathcal E_\tau(f,\Sigma)
			=
			\min_{q\in\mathcal H_{\mathcal K}}
			\left\{
			\|q\|_{\mathcal K}^2+\frac1\tau\|A_\Gamma q-f\|_H^2
			\right\}.
			\label{eq:regularized-mimic-variational}
		\end{equation}
		The unique minimizer is
		\[
		q_\tau=\mathcal K A_\Gamma^*(\Sigma+\tau I)^{-1}f.
		\]
	\end{proposition}
	
	\begin{proof}
		Set $B=A_\Gamma \mathcal K^{1/2}$.  By the minimum-preimage property
		in~\eqref{eq:source-covariance-norm}, the source-side minimum equals
		\[
		\min_{w\in X_\Gamma}
		\left\{\|w\|_{X_\Gamma}^2+\tau^{-1}\|Bw-f\|_H^2\right\}.
		\]
		The normal equation is
		$(B^*B+\tau I)w=B^*f$, whose solution is
		$w_\tau=B^*(BB^*+\tau I)^{-1}f$.  Substitution in the singular
		system of $B$ gives the value
		$\left\langle f,(BB^*+\tau I)^{-1}f\right\rangle_H$.  Applying
		$\mathcal K^{1/2}$ to $w_\tau$ gives the stated $q_\tau$.
	\end{proof}
        In particular, taking \(f=\phi_z\) in formula \eqref{eq:regularized-mimic-variational} gives the variational characterization of the regularized \(Q\)-statistic used
		throughout the paper, which balances statistical source energy against synthesis error. By \eqref{eq:regularized-mimic-energy}, 
 the regularized cost is finite and bounded by $\|f\|^2_H/\tau$ for every $f\in H$. Moreover, with the value $+\infty$ allowed, monotone convergence gives
\(
    \mathcal E_\tau(f,\Sigma)
    \nearrow
    \mathcal E_0(f,\Sigma)
    \text{ as }\tau\downarrow0.
\)

 Since $Q_\tau(z,\Sigma)$ is small at sampling points $z$ whose test far
field can be synthesized at low cost, we use its reciprocal as the
imaging indicator: 
\begin{equation}
    I_{Q,\tau}(z,\Sigma):=\frac{1}{Q_\tau(z,\Sigma)}.
    \label{eq:q-imaging-indicator}
\end{equation}
The indicator $I_{Q,\tau}$ is therefore large where the synthesis cost is low, and shares the Capon-type form
\cite{Capon1969}. While
Capon's method has been analyzed theoretically
for uncorrelated random-source identification \cite{GriesmaierRaumer2022}, 
the random scattering-induced sources on the boundary of the unknown obstacle considered here are generally correlated, with their statistics determined jointly by random illumination and nonlocal scattering interactions.
In practice, $\Sigma$ is replaced by the measured second moment $\widehat\Sigma$.
	
	\subsection{A spectral view of sampling methods}
	\label{subsec:comparison-classical-indicators}
	This subsection compares $Q$ with classical sampling methods. Under uniform full-aperture plane-wave illumination,
    \[T:=2\pi\Sigma=FF^*.\]
    Under the nonresonance condition, \(Q\)-imaging, factorization method (FM), multiple signal classification (MUSIC) and direct sampling method (DSM) can all be characterized in terms of the same eigensystem, differing primarily in their spectral weighting strategies, specifically, inverse weighting for \(Q\) and FM, hard-cutoff projection for MUSIC, and direct correlation for DSM. Consequently, inverse weighting amplifies the weakly excited modes \cite{RezacHarris2022},
whereas DSM involves no inversion and
localizes by correlating the data directly with the test far field.
Because the far-field second moment $\Sigma$ can be estimated from
measurements under general illumination, its eigensystem gives a
statistical counterpart of each sampling indicator directly. These
counterparts are summarized in Table~\ref{tab:sampling-spectral-weights}.

   \begin{table}[!htbp]
    \centering
    \caption{Statistical indicators under general illumination.
    Here \(\lambda\ge0\) denotes an eigenvalue of \(\Sigma\),
    and \(\Pi_r\) is the orthogonal projection onto the retained
    signal eigenspace.}
    \label{tab:sampling-spectral-weights}
    \begin{tabular}{@{}lll@{}}
        \toprule
        Method & Imaging indicator & Spectral treatment \\
        \midrule
        \(Q\)-imaging
        & \(\bigl\langle
        \phi_z,(\Sigma+\tau I)^{-1}\phi_z
        \bigr\rangle_H^{-1}\)
        & inverse weighting with \((\lambda+\tau)^{-1}\) \\[1ex]
        FM
        & \(\bigl\langle
        \phi_z,(\Sigma^{1/2}+\alpha I)^{-1}\phi_z
        \bigr\rangle_H^{-1}\)
        & inverse weighting with \((\sqrt{\lambda}+\alpha)^{-1}\) \\[1ex]
        MUSIC
        & \(\|(I-\Pi_r)\phi_z\|_H^{-2}\)
        & binary weighting \\[1ex]
        DSM
        & \(\langle\phi_z,\Sigma\phi_z\rangle_H\)
        & direct weighting with \(\lambda\) \\
        \bottomrule
    \end{tabular}
\end{table}
    \begin{remark}[Connection between \(Q\)-imaging and the linear sampling method]
In the uniform full-aperture plane-wave setting, it can be proved that
\[
Q_\tau(z,T)
=
\min_{g\in H}
\left\{
\|g\|_H^2+\tau^{-1}\|Fg-\phi_z\|_H^2
\right\}.
\]
The minimizer is
\[
g_{\tau,z}
=
(F^*F+\tau I)^{-1}F^*\phi_z,
\]
which is the Tikhonov-regularized solution of the far-field
equation $Fg=\phi_z$ used in the linear sampling method (LSM).
Thus
$Q_\tau(z,T)=\|g_{\tau,z}\|_H^2+\tau^{-1}\|Fg_{\tau,z}-\phi_z\|_H^2$
is the minimal value of a combined cost: the squared norm of the
solution plus the data misfit, weighted by $\tau^{-1}$. The standard
linear sampling method (LSM), by contrast, typically uses only the
norm $\|g_{\tau,z}\|_H$ of the regularized solution and ignores the
misfit term. 
\end{remark}
    \begin{remark}[Direct-correlation limit]
Under random illumination, statistical DSM is the expected squared correlation between the measured and test fields, averaged over the
illumination ensemble,
\[
I_{\mathrm{DSM}}(z)
= \langle \phi_z,\Sigma\phi_z\rangle_H
= \mathbb E_\omega
  \bigl|\langle y(\omega),\phi_z\rangle_H\bigr|^2.
\]
 With $J_z(y)=J_0(k|z-y|)$, the factorization \(\Sigma=A_\Gamma \mathcal K A_\Gamma^*\) gives
\(I_{\rm DSM}(z)\propto\langle J_z,\mathcal KJ_z\rangle_{L^2(\Gamma)}\),
which reduces to a \(J_0^2\)-profile $\int_\Gamma J_0(k|z-y|)^2\,ds(y)$ when \(\mathcal K=I\). This Bessel kernel describes the usual diffraction-scale localization mechanism for such correlation-based methods. Moreover, with $\tau\gg\|\Sigma\|$,
\[
Q_\tau(z,\Sigma)
=\frac{2\pi}{\tau}-\frac{I_{\rm DSM}(z)}{\tau^2}
+O\!\left(\frac{\|\Sigma\|^2}{\tau^3}\right),
\quad I_{Q,\tau}(z)
=
\frac{\tau}{2\pi}
+\frac{I_{\rm DSM}(z)}{4\pi^2}
+O\!\left(\frac{\|\Sigma\|^2}{\tau}\right).
\]
Thus statistical DSM gives the leading spatial contrast of the strongly
regularized   $Q$-indicator after subtraction of the constant background.
\end{remark}
	We illustrate the comparison on two geometries: a smooth pear-shaped obstacle and two radius-\(0.7\)
	disks centered at \((\pm0.8,0)\). The wavelength is \(\lambda=0.5\), so
	the disk gap is \(0.2=0.4\lambda\). A perturbed nonnegative second-moment matrix $\widehat\Sigma$ with $\|\widehat\Sigma-\Sigma\|_2/\|\Sigma\|_2=10^{-2}$ is used. We take
	\(\tau=10^{-2}\|\Sigma\|_2\) for \(Q\) and \(\alpha=\sqrt\tau\) for FM in Table~\ref{tab:sampling-spectral-weights}. MUSIC retains 26 leading modes.

	\begin{figure}[htbp]
		\centering
		\captionsetup{font={small,stretch=1.05}}
		\subfloat[\(Q\)]{\includegraphics[width=0.24\linewidth]{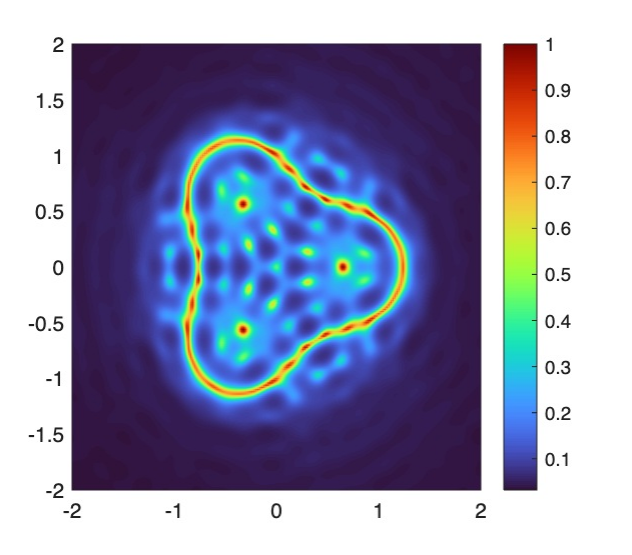}}\hfill
		\subfloat[FM]{\includegraphics[width=0.24\linewidth]{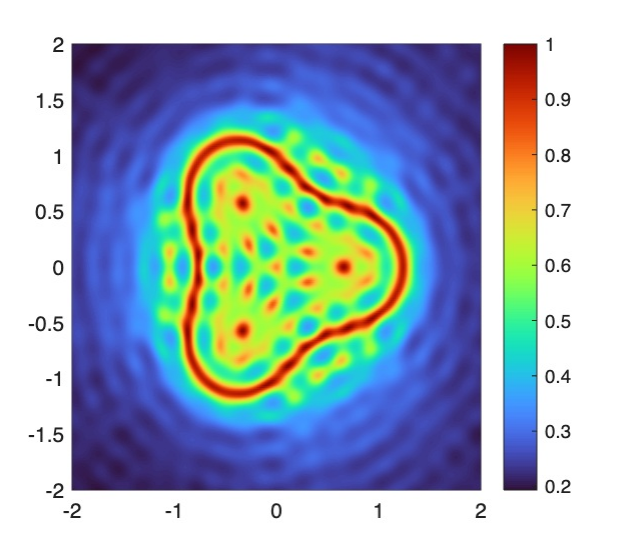}}\hfill
		\subfloat[MUSIC]{\includegraphics[width=0.24\linewidth]{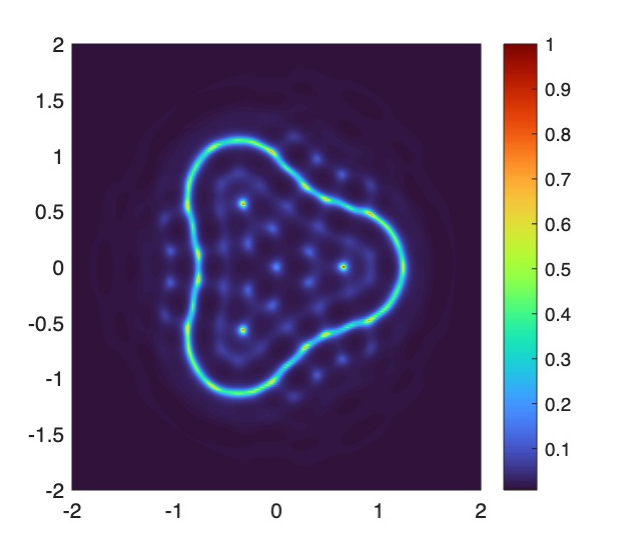}}\hfill
		\subfloat[ DSM]{\includegraphics[width=0.24\linewidth]{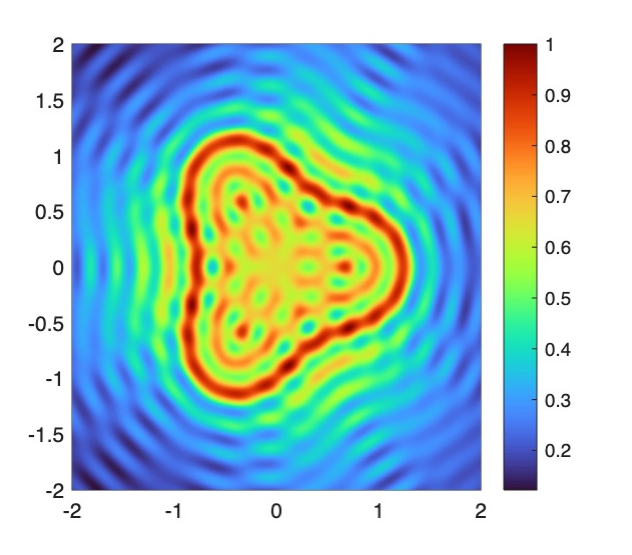}}\\[0.4ex]
		\subfloat[\(Q\)]{\includegraphics[width=0.24\linewidth]{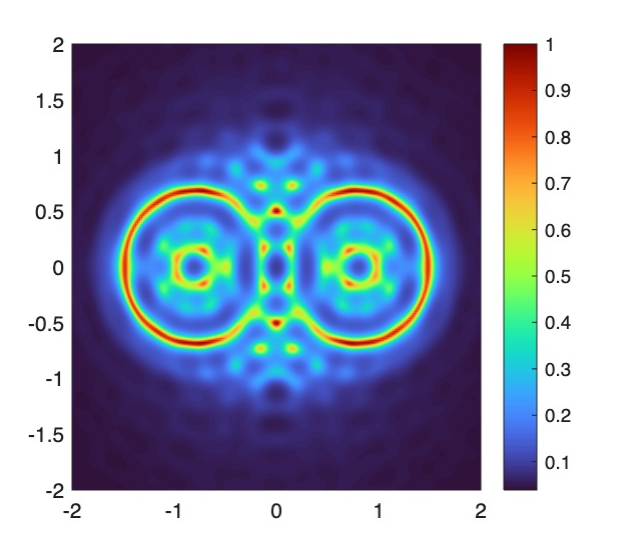}}\hfill
		\subfloat[FM]{\includegraphics[width=0.24\linewidth]{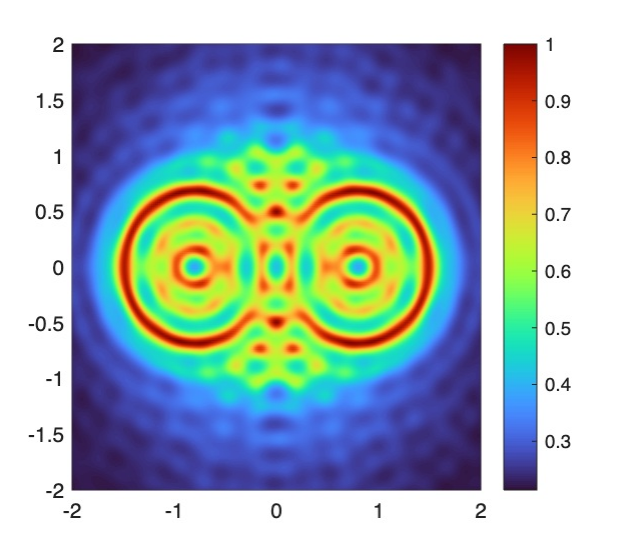}}\hfill
		\subfloat[MUSIC]{\includegraphics[width=0.24\linewidth]{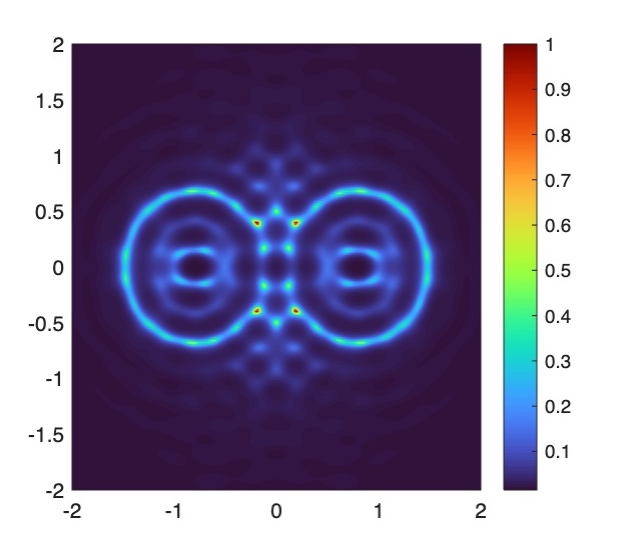}}\hfill
		\subfloat[ DSM]{\includegraphics[width=0.24\linewidth]{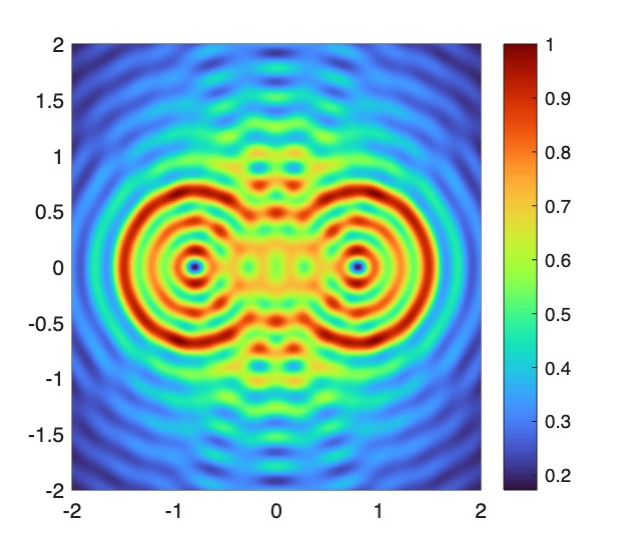}}\\[0.5ex]
		\caption{Uniform full-aperture plane-wave illumination: pear (top) and
			two nearby disks (bottom), at \(\lambda=0.5\) and 1\% relative
			second-moment perturbation. Each indicator is normalized by its own
			maximum.}
		\label{fig:classical-comparison}
	\end{figure}
	
	Figure~\ref{fig:classical-comparison} shows how the choice of spectral
weighting changes the images reconstructed from the same scattering
data. Because $Q_\tau$ applies the strongest inverse weighting, it
resolves fine boundary structure but also amplifies oscillations inside
the scatterer. FM uses a weaker inverse
weighting and accordingly gives a broader response with brighter
interiors. MUSIC, which projects onto a selected signal subspace, can
produce a thin and clean contour, but its quality depends sensitively
on the chosen rank. DSM, which involves no inversion, gives a more
diffuse image with oscillatory sidelobes.
	
	\subsection{Random point-source illumination and super-resolution}
	\label{subsec:general-incidence-Q}

	The preceding comparison fixes the illumination, yet all four methods struggle to recover the facing arcs of the two closely spaced disks. This shared difficulty suggests that limited boundary visibility may reflect constraints in the available data, not merely the choice of imaging functional. We therefore keep the \(Q\)-statistic fixed and vary the illumination to examine how the resulting data affect boundary visibility.

Figure~\ref{fig:point-source-improvement} shows this dependence for the two
disks of Figure~\ref{fig:classical-comparison}. Independent and uniformly distributed point sources placed near the boundaries
reveal the facing arcs more clearly than uniform plane-wave illumination. In particular, the two disks remain well resolved even when
their separation is below the conventional half-wavelength diffraction limit, attaining super-resolution. 
As the source circles approach the boundaries, the images also come closer to
the correlation-free case; see Appendix \ref{appendix-point source} for a justification. These observations point to two separate limits on
resolution: correlations among the excited boundary sources, and radiation
and noise. The next subsection quantifies the second limit, and
Section~\ref{sec:source-second-moment-imaging} then brings back the source statistics.

	\begin{figure}[htbp]
		\centering
		\captionsetup{font={small,stretch=1.05}}
		\subfloat[\(d=0.05\)]{\includegraphics[width=0.24\linewidth]{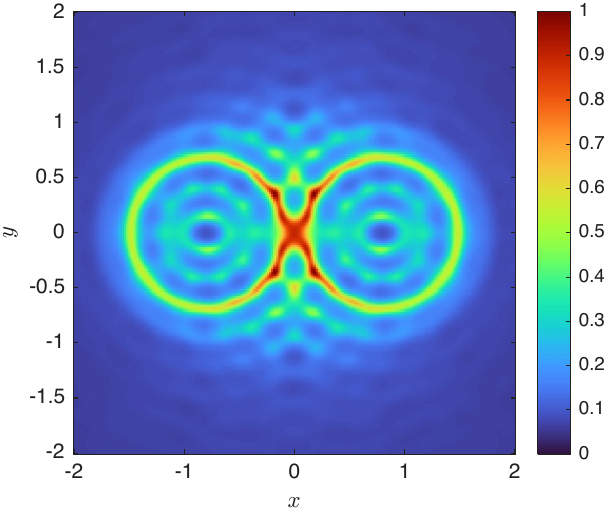}}\hfill
		\subfloat[\(d=0.025\)]{\includegraphics[width=0.24\linewidth]{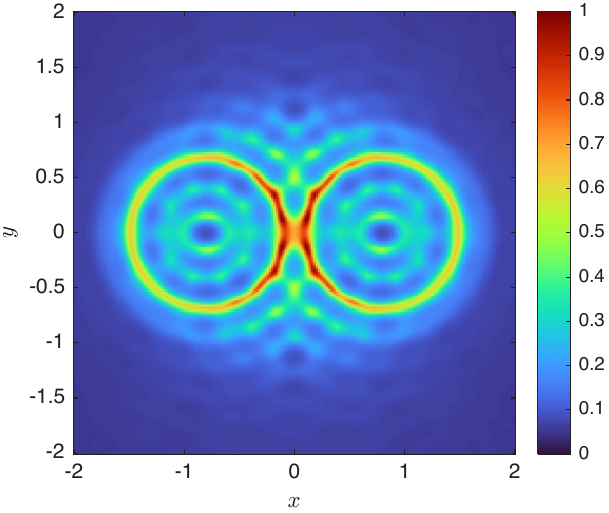}}\hfill
		\subfloat[\(d=0.01\)]{\includegraphics[width=0.24\linewidth]{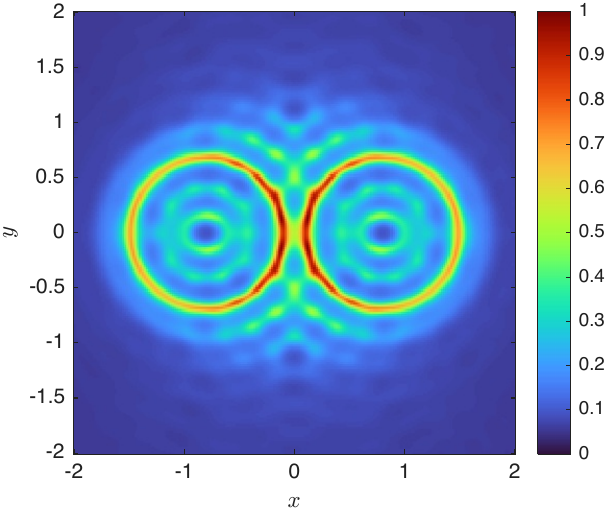}}\hfill
		\subfloat[Correlation-free]{\includegraphics[width=0.24\linewidth]{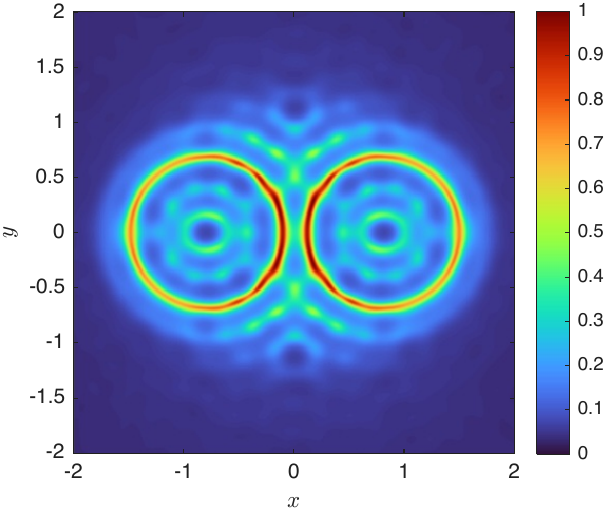}}
		\caption{Illumination-dependent \(Q\)-imaging of two radius-\(0.7\)
			disks centered at \((-0.8,0)\) and \((0.8,0)\).
			(a)--(c) Point sources are uniformly distributed on two exterior
			concentric circles, one around each disk, of radius \(0.7+d\).
			The far field is observed in 100 uniform directions, with
			\(\tau=\|\widehat\Sigma-\Sigma\|_2\) in (a)--(c).
			(d) Correlation-free boundary emission, with covariance proportional
			to the identity in \(L^2(\Gamma)\). The wavelength is
			\(\lambda=0.5\), the relative far-field second-moment perturbation
			is \(10^{-2}\), and each   image \(1/Q_\tau\) is
			normalized by its own maximum.}
		\label{fig:point-source-improvement}
	\end{figure}

	\subsection{Correlation-free resolution limit}
	\label{subsec:correlation-free-resolution}
To isolate radiation and noise, we first consider correlation-free boundary
sources with \(\mathcal K \propto  I\).
This can be understood through rescaled short-correlation limits
(Remark~\ref{rem:short-correlation-justification}), rather than the second-moment operator of an \(L^2(\Gamma)\)-valued random source with finite expected energy. Its radiated second moment is nevertheless well defined and trace class. 

		The same far-field reference is also obtained as a limit of physical monopole
		illumination as in Appendix \ref{appendix-point source}. In each realization, a monopole of fixed amplitude \(a\ne0\)
		is placed uniformly with respect to arclength on the exterior parallel
		boundary \(\Gamma_\delta=\{y+\delta\nu(y):y\in\Gamma\}\).
		Under the nonresonance condition,
		Theorem~\ref{thm:monopole-farfield} proves that its scattered far-field
		second moment satisfies
		\[
		\left\|\Sigma_\delta-\kappa_0 A_\Gamma A_\Gamma^*\right\|
		\longrightarrow0,\qquad
		\kappa_0={|a|^2}/{|\Gamma|},\qquad \delta\downarrow0. 
		\]
		The limit \(\kappa_0A_\Gamma A_\Gamma^*\) is the far-field second moment
produced by correlation-free boundary sources with \(\mathcal K=\kappa_0I\).

		After normalization, we analyze this reference through the one-sided
		half-height width of the boundary transition and establish sharp
		small-noise asymptotics at fixed frequency for disks and simply connected
		analytic boundaries.

\subsubsection{Correlation-free model and ideal identification}
\label{subsubsec:cf-model-identification}

We consider the idealized correlation-free
boundary source second moment $\mathcal K=I$.  We normalize
the common source strength and far-field factor and set
\begin{equation}
 A:=\gamma_k^{-1}A_\Gamma,\qquad
 (Aq)(\hat x)=\int_\Gamma e^{-ik\hat x\cdot y}q(y)\,ds(y),
 \qquad
 \Sigma:=AA^*.
 \label{eq:normalized-ideal-radiation}
\end{equation}
Assume that $D$ has a $C^2$ boundary and connected exterior,
and under the nonresonance condition, $A:L^2(\Gamma)\to L^2(\mathbb S^1)$ is compact and
 injective with dense range \cite{coltonInverseAcousticElectromagnetic2019}. 

Let
\[Av_j=\sigma_j u_j,\qquad
 A^*u_j=\sigma_j v_j,\qquad
 \sigma_1\ge\sigma_2\ge\cdots>0,\]
where $\{v_j\}$ and $\{u_j\}$ are orthonormal bases of
$L^2(\Gamma)$ and $L^2(\mathbb S^1)$, respectively.
Since $\Sigma u_j=\sigma_j^2u_j$, the spectral
representation of the ideal $Q$-statistic becomes
\begin{equation}
 Q_0(z,\Sigma)
 =
 \sum_{j=1}^{\infty}
 \frac{|\langle\phi_z,u_j\rangle_H|^2}{\sigma_j^2}
 =
 \sum_{j=1}^{\infty}|c_j(z)|^2,\quad c_j(z):=\frac{\langle\phi_z,u_j\rangle_H}{\sigma_j}.
 \label{eq:cf-picard-energy}
\end{equation}
By Picard's criterion, $Q_0(z)$ is finite precisely when
the radiation equation $Aq_z=\phi_z$ admits a solution
in $L^2(\Gamma)$. In that case, the solution is unique and
$Q_0(z)=\|q_z\|_{L^2(\Gamma)}^2$. The standard single-layer range argument
\cite{coltonInverseAcousticElectromagnetic2019,KirschGrinberg2008} then yields an exact geometric characterization of the obstacle
\begin{equation}
 Q_0(z)<\infty
 \quad\Longleftrightarrow\quad
 z\in D.
 \label{eq:ideal-identification}
\end{equation}
Here, interior solvability follows from Green's representation,
while exterior and boundary points are excluded by the
point-source singularity and the regularity of single-layer
potentials with $L^2(\Gamma)$ densities.

To move from exact identification in the noise-free case to resolution
at finite noise, we next choose the regularization parameter in terms of
the noise level, and then measure resolution by the outward
half-height width of the indicator. 

\subsubsection{Regularization and boundary resolution}

The ideal range criterion identifies the obstacle but does not quantify
the boundary transition at finite noise. We therefore use the ideal regularized
$Q$-statistic
\begin{equation}
    Q_\tau(z)
    =\sum_{j=1}^{\infty}
      \frac{|\langle\phi_z,u_j\rangle_H|^2}{\sigma_j^2+\tau}
    =\sum_{j=1}^{\infty}
      \frac{\sigma_j^2}{\sigma_j^2+\tau}|c_j(z)|^2,
    \qquad \tau>0,
    \label{eq:Q-soft-versus-hard}
\end{equation}
which is finite and strictly positive for every sampling point.
The filter factor is close to one when $\sigma_j^2\gg\tau$
and suppresses contributions relative to the exact source cost
when $\sigma_j^2\ll\tau$. We therefore describe the number of modes above the
regularization level by
\[
    J_\tau:=\#\{j:\sigma_j^2\ge\tau\}.
\]
This index characterizes the effective spectral band and the statistic
$Q_\tau$ itself retains the full spectrum.

Noise enters this reference analysis through the choice of $\tau$.
Given the second-moment error bound
$\|\widehat\Sigma-\Sigma\|\le \epsilon_\Sigma$, the calibration $\tau=C\epsilon_\Sigma$ with a fixed $C>1$ is adopted to suppress the weak, noise-dominated signal modes.
 In the following, we first establish the boundary-transition laws for the regularized statistic $Q_\tau$ formed with the exact second moment $\Sigma=AA^*$ while a detailed stability analysis for perturbed second moments is given in Section \ref{nnoise}.

\begin{definition}[One-sided half-height resolution]
    \label{def:analytic-transition-length}
    Let \(x\in\Gamma\) be a regular boundary point with exterior unit normal
	\(\nu(x)\). Fix an exterior normal collar of width $r_0>0$ and define the boundary-normalized   profile by
    \[
        \mathcal I_{\tau,x}(r)
        :=\frac{Q_\tau(x)}
        {Q_\tau(x+r\nu(x))},
        \qquad  0\le r<r_0.
    \]
    Its one-sided half-height width is
    \begin{equation}
        w_{1/2}(x;\tau)
        :=\inf\left\{
        0\le r<r_0:
        \mathcal I_{\tau,x}(r)\le\frac12
        \right\},
        \label{eq:analytic-halfheight-definition}
    \end{equation}
    with $\inf\varnothing=+\infty$, indicating that no half-height
    crossing occurs within the prescribed neighborhood.
\end{definition}

Thus $w_{1/2}(x;\tau)$ measures the first outward distance at which
the   image falls to half its boundary value, equivalently
where $Q_\tau$ reaches twice its boundary value. 
The following analysis examines how this width depends on the
noise level. It is worth mentioning that the half-height level is motivated by the familiar full width
at half maximum (FWHM) criterion. For the asymptotic limits proved below, any other fixed level in $(0,1)$
yields the same asymptotic order up to a prefactor.
Throughout, $f\sim g$ when $f/g\to1$
, and $f\asymp g$ when
$cg\le f\le Cg$ for fixed constants $0<c\le C<\infty$
in the stated limit.

	\subsubsection{Asymptotic resolution for a disk}
    We determine the asymptotic half-height width of the boundary transition for a disk. Let $D=B_R(0)$, with $R$ and $k$ fixed as $\tau\downarrow0$, and set
\[
 \rho_\tau:=\frac{\tau}{\|A\|^2},
 \qquad L_\tau:=\log(1/\rho_\tau).
\]
Since the angular Fourier modes diagonalize $AA^*$, the corresponding eigenvalues and regularization weights are
\[ \lambda_n=4\pi^2R\,J_n(kR)^2,\qquad
 p_{n,\tau}:=\frac{\lambda_n}{\lambda_n+\tau},
 \qquad n\in\mathbb Z.\]
For $r=|z|$, the testing field satisfies
$|\langle\phi_z,u_n\rangle_H|^2=2\pi J_n(kr)^2$.
Substitution into the spectral formula for $Q_\tau$ gives
\begin{equation}
 Q_\tau^{\rm disk}(r)
 =2\pi\sum_{n\in\mathbb Z}\frac{J_n(kr)^2}{\lambda_n+\tau}
 =\frac{1}{2\pi R}\sum_{n\in\mathbb Z}p_{n,\tau}
       \left|\frac{J_n(kr)}{J_n(kR)}\right|^2.
 \label{eq:disk-exact-Q}
\end{equation}
Here $J_n(kR)\ne0$ by the nonresonance condition.
	For each $\tau>0$, this radial function is finite and smooth. As $\tau\downarrow0$,
it converges to the exact source energy
for $r<R$ but diverges for $r\ge R$. Regularization therefore replaces the ideal inside--outside
distinction by a boundary transition of finite thickness. For sufficiently small $\tau$, let
\begin{equation}
 N_\tau:=\max\{n\ge0:\lambda_n\ge\tau\}
 \label{eq:disk-effective-order}
\end{equation}
denote the largest angular order above the regularization level.
Since nonzero angular orders occur in pairs,
$J_\tau=2N_\tau+1$ for sufficiently small $\tau$. The following theorem establishes the boundary-transition scale $R/N_\tau$
and its dependence on regularization.

\begin{theorem}[Disk boundary profile and resolution]
	\label{thm:disk-halfheight-resolution}
	Under the nonresonance condition, as $\tau\downarrow0$,
	\begin{equation}
		\frac{Q_\tau^{\rm disk}(R(1+t/N_\tau))}{Q_\tau^{\rm disk}(R)}
		\longrightarrow
		G(t):=\int_0^1e^{2ts}\,ds=\frac{e^{2t}-1}{2t},
		\label{eq:disk-boundary-layer-profile}
	\end{equation}
	locally uniformly for $t\ge0$, with $G(0)=1$.
	Consequently, the one-sided half-height width satisfies
	\begin{equation}
		w_{1/2}^{\rm disk}(\tau)
		\sim t_{1/2}\frac{R}{N_\tau}
		\sim 2t_{1/2}R\,
		\frac{\log\log(1/\rho_\tau)}{\log(1/\rho_\tau)},
		\label{eq:disk-halfheight-noise-law}
	\end{equation}
	where $t_{1/2}=0.6282156\ldots$ is the unique positive solution of $G(t)=2$.
\end{theorem}
\begin{proof}
	Write $N=N_\tau$. Lemma~\ref{lem:disk-regularized-estimates} in
	Appendix~\ref{app:disk-regularized-estimates} gives
	$Q_\tau^{\rm disk}(R)\sim N/(\pi R)$ and, uniformly on every fixed interval
	$0\le t\le T$,
	\[
	\frac{Q_\tau^{\rm disk}(R(1+t/N))}{Q_\tau^{\rm disk}(R)}
	=\frac1N\sum_{n=1}^{N}\left(1+\frac{t}{N}\right)^{2n}+o(1)
	\longrightarrow\int_0^1e^{2ts}\,ds.
	\]
	The last step is a uniform Riemann-sum limit, proving
	\eqref{eq:disk-boundary-layer-profile}.
	Since $G$ is strictly increasing from $1$ to infinity, local uniform convergence implies that the first level-$2$
crossing converges to $t_{1/2}$.
	Returning to the physical distance $tR/N$ yields
	$w_{1/2}^{\rm disk}(\tau)\sim t_{1/2}R/N$.
	Finally, the same lemma gives $N\sim L_\tau/(2\log L_\tau)$,
	which proves \eqref{eq:disk-halfheight-noise-law}.
\end{proof}
\begin{remark}[Role of wavenumber and noise]
For a disk, angular orders near $kR$ mark the transition
from radiating to evanescent modes.
At a fixed, moderate noise level that significantly affects the evanescent modes while only weakly perturbing the radiating ones, this suggests an effective bandwidth  $N_\tau\approx kR$ and hence the
conventional wavelength-scale width $R/N_\tau\approx k^{-1}$.
Our asymptotics instead fix $kR$ and let the noise level tend to zero,
and hence the calibrated relative regularization $\rho_\tau\to 0$. Eventually, $N_\tau\gg kR$, yielding a
transition width of order
\(
R\,\frac{\log\log(1/\rho_\tau)}{\log(1/\rho_\tau)}.
\)
Thus, reducing noise permits arbitrarily fine boundary transitions
in principle, although only at the slow logarithmic rate above.
\end{remark}
    For ease of visualization, we use the angularly truncated statistic
\[
Q_N^{\rm disk}(r)
:=\frac{1}{2\pi R}
\sum_{|n|\le N}
\left|\frac{J_n(kr)}{J_n(kR)}\right|^2.
\]
When $N=N_\tau$, the estimates in
Appendix~\ref{app:disk-regularized-estimates} show that this
hard-cutoff statistic and $Q_\tau^{\rm disk}$ have the same
limiting boundary-normalized profile $G(t)$, and hence the same
leading-order half-height width $t_{1/2}R/N_\tau$.
The following plots therefore illustrate the modal resolution
scale $R/N$; its dependence on regularization follows from
the asymptotic relation between $N_\tau$ and $\tau$.
	Figure~\ref{fig:diskwidth} visualizes the exterior transition for
	\(N=16,32,64\) on the same local window. The distance from the true
	boundary to the exterior half-height contour decreases as more angular
	modes are retained. 
	Figure~\ref{fig:disk-cutoff-halfheight} complements the visualization with
	radial profiles for \(N=24,48,96,192\) and directly computed first-crossing
	widths, approaching \(w_{1/2}\sim t_{1/2}R/N\).
	
	\begin{figure}[htbp]
		\centering
		\includegraphics[width=0.8\linewidth]{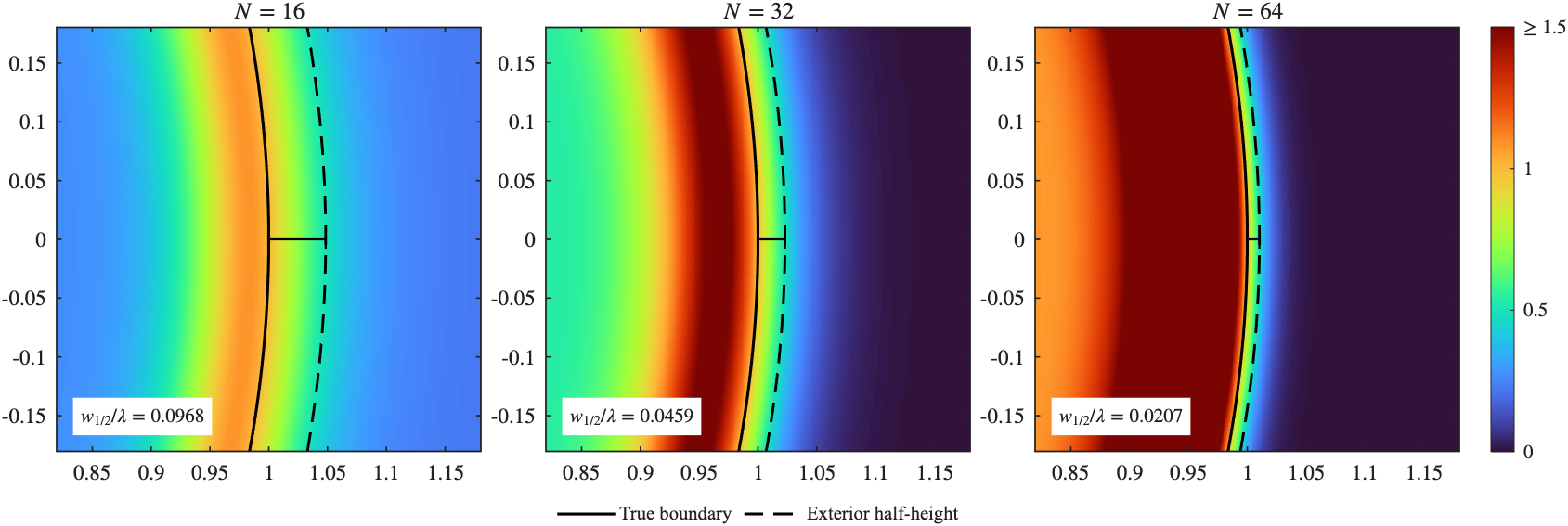}
		\caption{Visualization of the disk half-height resolution,
			with \(R=1\), \(\lambda=0.5\), and angular cutoffs
			\(N=16,32,64\) from left to right. Colors show
			\(Q_N(R)/Q_N(r)\) on a common scale, saturated above \(1.5\).
			Solid curves mark the true boundary, dashed curves mark the first
			exterior level \(Q_N(r)=2Q_N(R)\). The horizontal segments measure
			\(w_{1/2}\). 
            }
		\label{fig:diskwidth}
		\end{figure}
			
	\begin{figure}[htbp]
		\centering
		\subfloat[Radial transition at the boundary]{\includegraphics[width=0.35\linewidth]{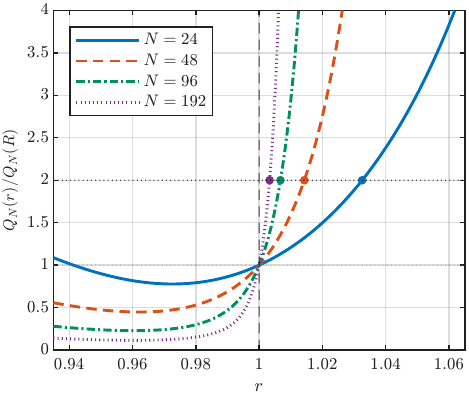}}\qquad
		\subfloat[Half-maximum width versus cutoff]{\includegraphics[width=0.35\linewidth]{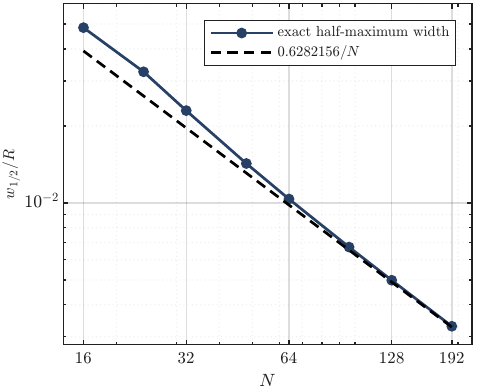}}
		\caption{Quantitative disk resolution analysis with \(R=1\) and \(\lambda=0.5\).
			(a) Boundary-normalized \(Q_N\) profiles; the vertical dashed line
			marks \(r=1\), and dots mark \(Q_N(r)/Q_N(1)=2\), equivalently
			the half-height level of the boundary-normalized   indicator.
			(b) Exact widths compared with \(t_{1/2}R/N\).}
		\label{fig:disk-cutoff-halfheight}
	\end{figure}

	\subsubsection{Sharp boundary-layer asymptotics for analytic boundaries}
\label{confo}
In this subsection, we extend the disk profile to the actual Tikhonov-regularized statistic
on a fixed bounded simply connected domain $D$ with regular analytic
Jordan boundary $\Gamma$. Let $\Phi$ map $\mathbb C\setminus\overline D$ conformally onto
$\{|w|>1\}$, with $\Phi(z)=z/\operatorname{cap}(\overline{D})+\mathcal O(1)$ at infinity, where 
$\operatorname{cap}(\overline{D})$ is the logarithmic capacity of $\overline{D}$ \cite{Saff2010}.
Analyticity permits conformal extension across $\Gamma$ \cite{Pommerenke1992}. In a fixed
sufficiently small collar, define
\begin{equation}
 g(z):=\log|\Phi(z)|,
 \qquad h_\Gamma(x):=\frac{1}{|\Phi'(x)|},
 \qquad x\in\Gamma.
 \label{eq:analytic-conformal-collar}
\end{equation}
Then $g=0$ on $\Gamma$, $g>0$ outside, and
\begin{equation}
 g(x+r\nu(x))=\frac{r}{h_\Gamma(x)}+\mathcal O(r^2)
 \label{eq:conformal-normal-scale}
\end{equation}
uniformly along $\Gamma$. The positive conformal scale $h_\Gamma$
converts conformal depth to physical normal distance and equals $R$
on a disk.

\begin{theorem}[Analytic-boundary profile and resolution]
\label{thm:analytic-boundary-resolution}
For a fixed bounded simply connected domain $D$ with regular analytic
Jordan boundary $\Gamma$, under the nonresonance condition, retain $G$ and $t_{1/2}$ from
Theorem~\ref{thm:disk-halfheight-resolution}, and set
$J_\tau=\#\{j:\sigma_j^2\ge\tau\}$,
$\rho_\tau=\tau/\|A\|^2$, and $L_\tau=\log(1/\rho_\tau)$.
As $\tau\downarrow0$,
\begin{equation}
 J_\tau\sim\frac{L_\tau}{\log L_\tau},
 \qquad Q_\tau(x)\sim\frac{J_\tau}{2\pi h_\Gamma(x)}.
 \label{eq:analytic-boundary-count}
\end{equation}
For every fixed $T>0$,
\begin{equation}
 \sup_{\substack{x\in\Gamma\\0\le t\le T}}
 \left|
 \frac{Q_\tau(x+2h_\Gamma(x)t\nu(x)/J_\tau)}{Q_\tau(x)}
 -G(t)
 \right|\longrightarrow0.
 \label{eq:analytic-boundary-profile}
\end{equation}
Consequently, the actual half-height width in
Definition~\ref{def:analytic-transition-length} satisfies, uniformly
in $x\in\Gamma$,
\begin{equation}
 \begin{aligned}
 w_{1/2}(x;\tau)
 =(2t_{1/2}+o(1))\,\frac{h_\Gamma(x)}{J_\tau}
 \sim 2t_{1/2}h_\Gamma(x)
       \frac{\log\log(1/\rho_\tau)}{\log(1/\rho_\tau)}.
 \end{aligned}
 \label{eq:analytic-halfheight-noise-law}
\end{equation}
\end{theorem}

The complete proof is given in
Appendix~\ref{app:analytic-radiation-estimates}.
The boundary-normalized profile $G$ is universal, while the local
conformal factor \(h_\Gamma(x)\) controls both the boundary value of the indicator \(I_{Q,\tau}(x)\sim\frac{2\pi h_\Gamma(x)}{J_\tau}\) from \eqref{eq:analytic-boundary-count} and the physical
transition width from \eqref{eq:analytic-halfheight-noise-law}. Thus, a larger conformal scale corresponds asymptotically to a brighter indicator and a wider physical boundary transition. 

For a disk, \(h_\Gamma=R\) and \(J_\tau=2N_\tau+1\) for sufficiently
small \(\tau\), so both the boundary value of \(I_{Q,\tau}\) and the
transition width are uniform along \(\Gamma\), recovering the preceding
disk law. For an ellipse with semiaxes \(a\ge b>0\) and parametrization
\(x(\theta)=(a\cos\theta,b\sin\theta)\),
\[
h_\Gamma(x(\theta))
=
\sqrt{a^2\sin^2\theta+b^2\cos^2\theta}.
\]
For a general analytic boundary, \(h_\Gamma\) is usually unavailable in closed form. We compute it from the harmonic-measure density \(\omega\) via
\(
h_\Gamma(\gamma(t))
=
|\gamma'(t)|/(2\pi\omega(t)) 
\)
where \(\omega\) is obtained by a Nystr\"om discretization of the logarithmic
equilibrium-potential equation \cite{Symm1967}.  With $\rho_\tau=10^{-3}$, the images in Figure~\ref{fig:conformal-modulation-numerics}(b)(c) are qualitatively consistent with the spatial dependence in the correlation-free model: regions with larger \(h_\Gamma\) tend to exhibit stronger boundary responses and broader transitions.

Thus conformal geometry alone can produce nonuniform contrast and
resolution under uniform, uncorrelated excitation. Interestingly, the enhanced brightness observed in concave regions in Figure \ref{fig:conformal-modulation-numerics}(b) differs from their weak
appearance in the scattering-induced image in Figure \ref{fig:conformal-modulation-numerics}(d), indicating that
this degradation cannot be explained by geometry alone and highlighting the role of nonuniform, correlated induced boundary sources discussed below.
\begin{figure}[htbp]
		\centering
         \subfloat[]{\includegraphics[width=0.20\linewidth]{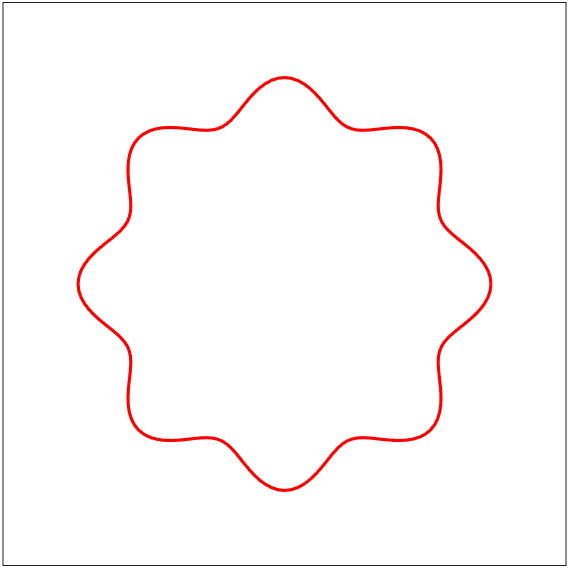}}\quad
         \subfloat[]{\includegraphics[width=0.23\linewidth]{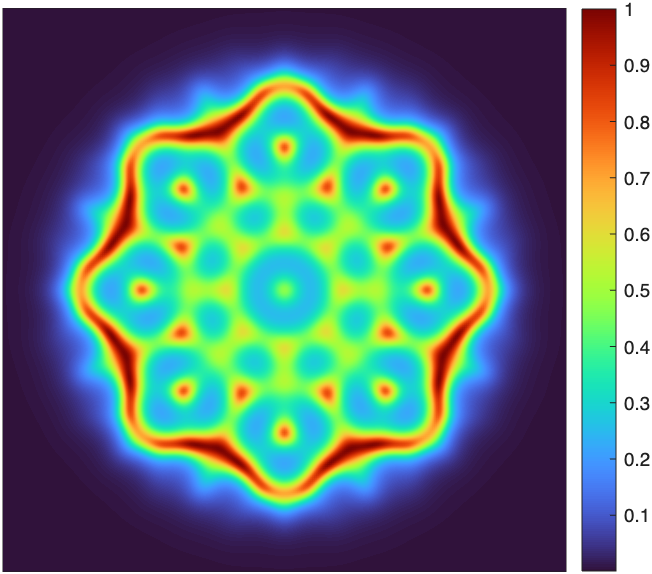}}\quad
         \subfloat[]{\includegraphics[width=0.23\linewidth]{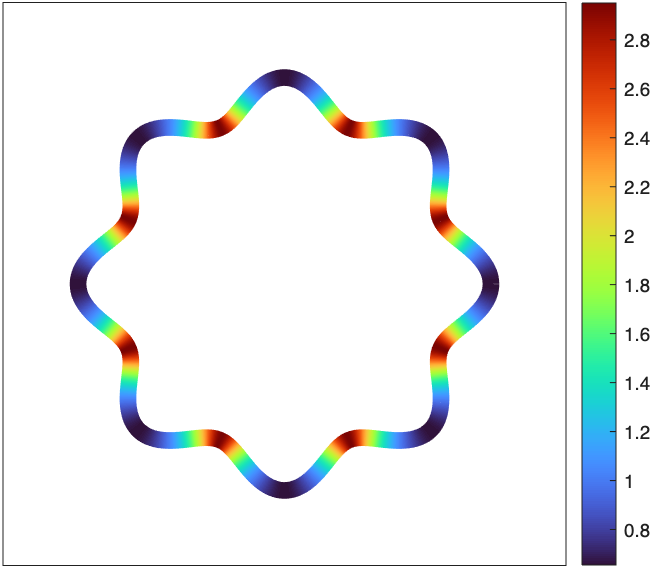}}\quad
         \subfloat[]{\includegraphics[width=0.22\linewidth]{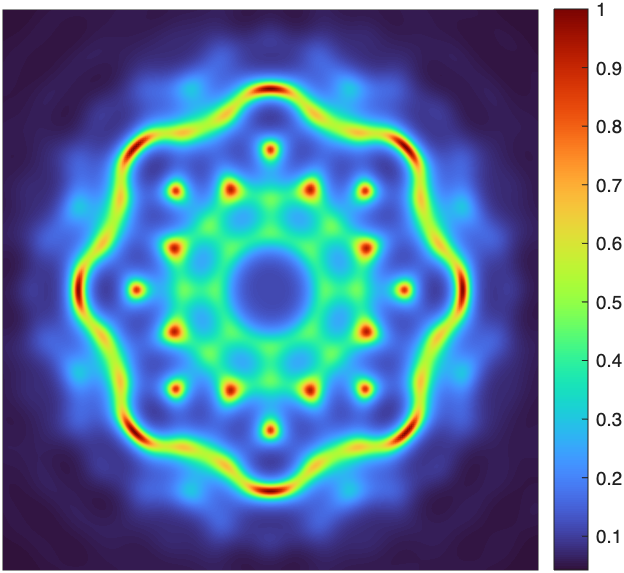}}
		\caption{
(a) The boundary geometry. (b) The correlation-free indicator
    \(I_{Q,\tau}\) with
    \(\tau=10^{-3}\|A\|^2\).
    (c) The conformal scale \(h_\Gamma\) along the boundary.  (d) The indicator
    \(I_{Q,\tau}\) constructed from the full-aperture
plane-wave far-field scattering data.
    }
		\label{fig:conformal-modulation-numerics}
		\end{figure}

        \subsubsection{Stability of the resolution laws}
\label{nnoise}
The regularization parameter $\tau$ is calibrated to the
second-moment error level and determines the effective spectral
bandwidth underlying the preceding resolution laws.
Using the exact correlation-free reference $\Sigma=AA^*$,
we have characterized the resulting boundary amplitude,
limiting profile, and sharp half-height width.
The following proposition shows that these same asymptotics
hold for the measured noisy statistic, provided that the second-moment error
remains below the regularization level by a fixed margin, namely, 
$\|\widehat\Sigma-\Sigma\|\le\delta_0\tau$
for a fixed $0\le\delta_0<1$.

\begin{proposition}[Noisy-data correlation-free boundary profile and resolution]
\label{prop:noisy-boundary-resolution}
Under the assumptions of
Theorem~\ref{thm:analytic-boundary-resolution}, let \(\Sigma=AA^*\) and define
\[
    \widehat Q_\tau(z)
    :=
    \left\langle
        \phi_z,(\widehat\Sigma+\tau I)^{-1}\phi_z
    \right\rangle_H,
\]
with \(
    \widehat\Sigma=\widehat\Sigma^*\)
    satisfying 
    \(\|\widehat\Sigma-\Sigma\|\le\epsilon_\Sigma\).
Suppose that $\tau\downarrow0$ and
\[
    \frac{\epsilon_\Sigma}{\tau}\le\delta_0<1
\]
for a fixed constant $\delta_0$.
Then, for every fixed $T>0$ and uniformly for $x\in\Gamma$,
\begin{equation}
\widehat Q_\tau(x)
    \sim\frac{J_\tau}{2\pi h_\Gamma(x)},\quad
    \sup_{\substack{x\in\Gamma\\0\le t\le T}}
    \left|
        \frac{\widehat Q_\tau(x+2h_\Gamma(x)t\nu(x)/J_\tau)}{\widehat Q_\tau(x)}
        -G(t)
    \right|
    \longrightarrow0.
    \label{eq:noisy-boundary-profile}
\end{equation}
Let $\widehat w_{1/2}(x;\tau)$ be defined as in
Definition~\ref{def:analytic-transition-length}, with $Q_\tau$
replaced by $\widehat Q_\tau$. Then, uniformly for $x\in\Gamma$,
\begin{equation}
    \widehat w_{1/2}(x;\tau)
    =
    (2t_{1/2}+o(1))
    \frac{h_\Gamma(x)}{J_\tau}
    \sim
    2t_{1/2}h_\Gamma(x)
    \frac{\log\log(1/\rho_\tau)}
         {\log(1/\rho_\tau)}.
    \label{eq:noisy-halfheight-width}
\end{equation}
\end{proposition}

\begin{proof}
Self-adjointness and the perturbation bound imply
\[
    \Sigma+(\tau-\epsilon_\Sigma)I
    \le
    \widehat\Sigma+\tau I
    \le
    \Sigma+(\tau+\epsilon_\Sigma)I.
\]
Since $\tau-\epsilon_\Sigma\ge(1-\delta_0)\tau>0$,
inversion reverses the operator order and gives
\begin{align}\label{compp}
    Q_{\tau+\epsilon_\Sigma}(z)
    \le
    \widehat Q_\tau(z)
    \le
    Q_{\tau-\epsilon_\Sigma}(z).
\end{align}
With $s_\pm=\tau\pm\epsilon_\Sigma$, since
\[
    L_{s_\pm}
    =
    L_\tau
    -
    \log\left(1\pm\frac{\epsilon_\Sigma}{\tau}\right)
    =
    L_\tau+O(1),
\]
the asymptotic $J_s\sim L_s/\log L_s$ yields
\(
    a_\pm:=J_{s_\pm}/J_\tau\longrightarrow1.
\)
Moreover,
\[
    z_\tau(x,t):=x+\frac{2h_\Gamma(x) t}{J_{\tau}}\nu(x)
    =
    x+\frac{2h_\Gamma(x)(a_\pm t)}{J_{s_\pm}}\nu(x).
\]
Applying the ideal boundary profile asymptotics \eqref{eq:analytic-boundary-profile} at
regularization level $s_\pm$, we obtain
\[
    \frac{2\pi h_\Gamma(x)}{J_\tau}
    Q_{s_\pm}(z_\tau(x,t))
    =
    a_\pm\bigl[G(a_\pm t)+o(1)\bigr]
    \longrightarrow G(t),
\]
uniformly for $x\in\Gamma$ and $0\le t\le T$.
The preceding comparison \eqref{compp} gives 
\begin{equation}
    \sup_{\substack{x\in\Gamma\\0\le t\le T}}
    \left|
        \frac{2\pi h_\Gamma(x)}{J_\tau}
        \widehat Q_\tau(z_\tau(x,t))
        -G(t)
    \right|
    \longrightarrow0.
\end{equation}
Taking $t=0$ gives the boundary amplitude, and division by this
positive boundary limit proves \eqref{eq:noisy-boundary-profile}.
Fix $0<a<t_{1/2}$. Since $G$ is strictly increasing and
$G(t_{1/2})=2$, for all sufficiently small $\tau$ the measured
ratio is below $2$ throughout $0\le t\le t_{1/2}-a$ and above
$2$ at $t=t_{1/2}+a$, uniformly in $x$. Hence
\[
    2(t_{1/2}-a)\frac{h_\Gamma(x)}{J_\tau}
    \le
    \widehat w_{1/2}(x;\tau)
    \le
    2(t_{1/2}+a)\frac{h_\Gamma(x)}{J_\tau}.
\]
Letting $a\downarrow0$ and using
$J_\tau\sim L_\tau/\log L_\tau$ proves
\eqref{eq:noisy-halfheight-width}.
\end{proof}

\begin{remark}[Boundary asymptotics in terms of the noise level]
Choose $\tau=\tau(\epsilon_\Sigma)=C\epsilon_\Sigma$
with a fixed $C>1$, 
applying Proposition~\ref{prop:noisy-boundary-resolution} therefore
gives the half-height width, as the noise level $\epsilon_\Sigma\downarrow0$ and uniformly for $x\in\Gamma$,
\[
    \widehat w_{1/2}(x;\tau(\epsilon_\Sigma))
    \sim
    2t_{1/2}h_\Gamma(x)
    \frac{
        \log\log(\|A\|^2/\epsilon_\Sigma)
    }{
        \log(\|A\|^2/\epsilon_\Sigma)
    }.
\]
The fixed calibration factor $C$ does not affect these leading-order asymptotics.
\end{remark}

\begin{remark}[Stability under geometric approximation]
Let $\Gamma$ be a regular $C^1$ Jordan boundary and $\Gamma_n$ be analytic Jordan boundaries converging to $\Gamma$ in $C^1$-norm through regular parametrizations. Let $A_n$ and $A$ denote the
corresponding normalized radiation operators in
\eqref{eq:normalized-ideal-radiation}, and set
$\Sigma_n=A_{n}A_{n}^*$ and
$\Sigma=AA^*$. Since the associated integral kernels converge uniformly, $\varepsilon_n:=\|\Sigma_n-\Sigma\|\to0$.
For every fixed $\tau>0$, an operator-order comparison gives
\[
\sup_{z\in\mathbb R^2}
\left|
\frac{I_{Q,\tau}(z,\Sigma_n)}
     {I_{Q,\tau}(z,\Sigma)}-1
\right|
\le \frac{\varepsilon_n}{\tau}
\longrightarrow0.
\]
Thus, the regularized image of a $C^1$ boundary is the
uniform relative limit of the images generated
by its analytic approximations. In particular, the loss of boundary
analyticity alone causes no discontinuity in the regularized
image under such geometric approximation.
\end{remark}
This section has developed the \(Q\)-statistic and statistical counterparts of classical sampling methods. Using the correlation-free reference to isolate the resolution restrictions imposed by
radiation and regularization, we have established sharp
boundary-transition laws for disks
and simply connected analytic boundaries, and proved their stability under second-moment errors controlled by the regularization level. The next section examines how coherent mean, nonuniform variance,
and correlation within general boundary-source statistics modify image contrast and resolution.

	\section{Boundary-Source Statistics and Resolution}
	\label{sec:source-second-moment-imaging}

Section~\ref{sec:mathematical-statistical-formulation} established the
factorization $\Sigma=A_\Gamma\mathcal K_{s,\Gamma}A_\Gamma^*$.
Section~\ref{sec:q-distance-resolution} then showed that $Q_\tau$ is the
minimal regularized statistical cost induced by this factorization, and
derived a noise-dependent resolution limit for correlation-free induced
boundary sources. In an actual scattering experiment, however, the
induced boundary sources are generally nonuniform and correlated. A
radiation mode that is observable in the correlation-free model may then
contribute little to imaging if the corresponding source mode is only
weakly excited, and is therefore costly to resolve from finite, noisy
data. The purpose of this section is to explain how coherent mean power, fluctuation variance, and correlation modify the contrast and modal visibility of \(Q_\tau\) relative to the correlation-free model. To separate these
three effects, we work with general boundary source statistics and do
not require them to come from a particular scattering configuration.

	\subsection{Mean--covariance decomposition}
	\label{subsec:mean-variance-diversity}
	We separate the contributions of the coherent mean and centered fluctuations associated with  boundary sources to identify their distinct roles in \(Q\)-imaging. In the \(L^2(\Gamma)\) realization, let
	\[
	m=\mathbb E\varphi,\qquad
	C=\mathbb E[(\varphi-m)\otimes(\varphi-m)^*],
	\qquad
	\mathcal K_{s,\Gamma}=C+m\otimes m^* .
	\]
	With \(b=A_\Gamma m\), this decomposition propagates to the data second moment,
	\begin{equation}
		\Sigma=A_\Gamma C A_\Gamma^*+b\otimes b^* .
		\label{eq:mean-fluctuation-data-decomposition}
	\end{equation}
	Applying the source--signal identity of
	Proposition~\ref{thm:regularized-mimic-variational} to
	\(B(r,\alpha)=A_\Gamma C^{1/2}r+\alpha b\) gives
	\begin{equation}
		\langle f,(\Sigma+\tau I)^{-1}f\rangle_H
		=
		\min_{\substack{g\in\mathcal H_C\\ \alpha\in\mathbb C}}
		\left\{
		\|g\|_C^2+|\alpha|^2
		+\frac1\tau\|A_\Gamma g+\alpha b-f\|_H^2
		\right\},
		\label{eq:section4-general-variational}
	\end{equation}
	where $\mathcal H_C$ and the corresponding covariance-weighted norm are defined analogously to \eqref{eq:source-covariance-norm}. The scalar \(\alpha\) controls the coherent contribution along \(b\), whereas \(g\) ranges over the fluctuation source space. Consequently, the mean modifies the synthesis cost through a single coherent field, while the covariance determines the available fluctuation directions and their relative costs.

	\subsection{Effects of   boundary source mean and variance}
	\label{subsec:zero-correlation-mean-variance}
We now suppress spatial correlation in the centered fluctuations of the   boundary sources to isolate the effects of the coherent mean and local variance.  Let \(v\in L^\infty(\Gamma)\), \(v\ge0\), denote the fluctuation
variance density with respect to arclength, and define the
multiplication operator
	\((M_vg)(y)=v(y)g(y)\).  With \(b=A_\Gamma m\), set
	\begin{equation}
		\Sigma_v:=A_\Gamma M_vA_\Gamma^*,
		\quad
		\Sigma_{v,m}:=\Sigma_v+b\otimes b^*,\quad Q_{\tau,v,m}(z):=Q_\tau(z,\Sigma_{v,m}).
		\label{eq:zero-correlation-data-model}
	\end{equation}
    Here \(M_v\) describes correlation-free fluctuations with possibly
nonuniform variance, while \(b\otimes b^*\) retains the coherent
mean contribution.
Since \(A_\Gamma\) is Hilbert--Schmidt and \(M_v\) is bounded and
nonnegative, \(\Sigma_{v,m}\) is a positive trace-class operator.
For \(m=0\) and \(v\equiv v_0>0\), this model reduces to the normalized correlation-free reference of Section~\ref{subsec:correlation-free-resolution}.
	
	\begin{remark}[Short-correlation interpretation]
		\label{rem:short-correlation-justification}
		The multiplication operator \(M_v\) in
		\eqref{eq:zero-correlation-data-model} can be interpreted as a
		zero-correlation-length limit. Specifically, set
		\(C_{\ell,v}=M_{\sqrt v}e^{\ell^2\Delta_\Gamma} M_{\sqrt v}/Z_\ell\) where $\Delta_\Gamma$ is the nonpositive arclength Laplacian and $Z_\ell$ is the power normalization constant in \eqref{zl},
see Appendix~\ref{app:heat-kernel-model} for the diffusion-model
construction.  Then
		\[\bigl\|Z_\ell\Sigma_{\ell,v,0}+b\otimes b^*
			-\Sigma_{v,m}\bigr\|=\bigl\|Z_\ell A_\Gamma C_{\ell,v}A_\Gamma^*
			-A_\Gamma M_vA_\Gamma^*\bigr\|\longrightarrow0,
			\qquad \ell\downarrow0.\]
		Thus \eqref{eq:zero-correlation-data-model} is the far-field limit of
		short-correlated source covariances.
	\end{remark}
	
	The variational identity
\eqref{eq:section4-general-variational} specializes to
	\begin{equation}
		\label{eq:independent-variational}
		Q_{\tau,v,m}(z)
		=
		\min_{\substack{g\in\mathcal X_v\\ \alpha\in\mathbb C}}
		\left\{
		\int_{\{v>0\}}\frac{|g(y)|^2}{v(y)}\,ds(y)
		+|\alpha|^2
		+\frac1\tau\|A_\Gamma g+\alpha b-\phi_z\|_H^2
		\right\},
	\end{equation}
	where
	\[
	\mathcal X_v
	=
	\left\{
	g:\ g=0\ \text{a.e. on }\{v=0\},\
	\int_{\{v>0\}}\frac{|g|^2}{v}\,ds<\infty
	\right\}.
	\]
	Larger $v$ lowers the local fluctuation cost, whereas the mean
contributes through the single direction $b$.
We first quantify the effect of this coherent contribution
on the $Q$-statistic.
\begin{proposition}
    \label{thm:mean-rank-one-effect}
    Let $R=(\Sigma_v+\tau I)^{-1}$ and assume $b\ne0$.
    Write
    $\psi_z=R^{1/2}\phi_z=\psi_z^\parallel+\psi_z^\perp$
    for its orthogonal decomposition along
    $\operatorname{span}\{R^{1/2}b\}$ and its orthogonal complement.
    Then
    \begin{equation}
 Q_{\tau,v,m}(z)
    =
    \|\psi_z^\perp\|_H^2
    +
    \frac{\|\psi_z^\parallel\|_H^2}
         {1+\|R^{1/2}b\|_H^2}.
        \label{eq:mean-rank-one-factor}
    \end{equation}
    Since $Q_{\tau,v,0}(z)
    =
    \|\psi_z^\perp\|_H^2
    +
    \|\psi_z^\parallel\|_H^2$, the coherent mean reduces the cost along a single
    weighted direction, leaving the orthogonal component unchanged.
\end{proposition}
\begin{proof}
    By the Sherman--Morrison identity \cite{Hager1989}
    \[(\Sigma_v+b\otimes b^*+\tau I)^{-1}=R-\frac{(R b)\otimes (Rb)^*}
		{1+\|R^{1/2}b\|_H^2},\]
    and the orthogonal decomposition of $\psi_z$,
    \[
    \begin{aligned}
    Q_{\tau,v,m}(z)
    =
    \|\psi_z\|_H^2
    -
    \frac{|\langle R^{1/2}b,\psi_z\rangle_H|^2}
         {1+\|R^{1/2}b\|_H^2}
    =
    \|\psi_z^\perp\|_H^2
    +
    \frac{\|\psi_z^\parallel\|_H^2}
         {1+\|R^{1/2}b\|_H^2}.
    \end{aligned}
    \]
\end{proof}
    We next examine how the variance affects the $Q$-statistic,
with the mean and regularization parameter fixed.
	\begin{proposition}[Monotonicity with respect to variance]
    \label{prop:raw-variance-sensitivity}
    Let \(w\in L^\infty(\Gamma)\), \(w\ge0\), and set
    \(v_\varepsilon=v+\varepsilon w\) for \(\varepsilon\ge0\).
    At fixed \(m\) and \(\tau>0\), with
    \(R_{\tau,v,m}=(\Sigma_{v,m}+\tau I)^{-1}\),
    \begin{equation}
        \left.\frac{d}{d\varepsilon}
        Q_{\tau,v_\varepsilon,m}(z)\right|_{\varepsilon=0^+}
        =
        -\int_\Gamma w(y)
        |(A_\Gamma^*R_{\tau,v,m}\phi_z)(y)|^2\,ds(y)
        \le0.
        \label{eq:variance-derivative}
    \end{equation}
\end{proposition}

\begin{proof}
    Since \(\partial_\varepsilon\Sigma_{v_\varepsilon,m}
    =A_\Gamma M_wA_\Gamma^*\), the inverse derivative at zero is
    \(-R_{\tau,v,m}A_\Gamma M_wA_\Gamma^*R_{\tau,v,m}\).
    Pairing with \(\phi_z\) proves \eqref{eq:variance-derivative}.
\end{proof}
    \begin{remark}
    At a fixed absolute second-moment noise level, increasing the variance
raises the effective data SNR. Specifically, for $m=0$, uniform variance scaling gives
    \[
    Q_{\tau,\rho v,0}(z)
    =\rho^{-1}Q_{\tau/\rho,v,0}(z),
    \qquad \rho>0.
    \]
    With $\tau$ fixed at the corresponding regularization level, the identity expresses this gain as a reduction
of the effective regularization level from $\tau$ to $\tau/\rho$,
which increases the regularization weights of the data modes, linking source power to the resolution analysis of
    Section~\ref{subsec:correlation-free-resolution}.
\end{remark}
        
	On the other hand,  a spatially nonuniform change of \(v\) reweights different boundary regions and may alter the normalized image profile. The following estimate gives a local geometric interpretation. 
    \begin{figure}[htbp]
    \centering
    \includegraphics[width=0.4\linewidth]{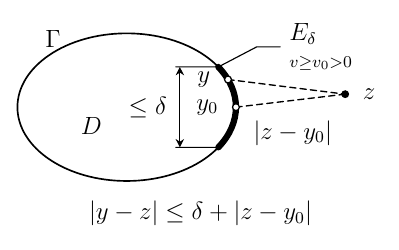}
    \caption{Geometry of the local-variance estimate.}
    \label{fig:local-variance-schematic}
\end{figure}

Suppose \(v\ge v_0>0\) on a boundary segment \(E_\delta\) containing
\(y_0\), with \(|E_\delta|>0\) and
\(\operatorname{diam}(E_\delta)\le\delta\).
In \eqref{eq:independent-variational}, choose \(\alpha=0\) and
\(g=\mathbf1_{E_\delta}/(\gamma_k|E_\delta|)\).
Using \(\|\phi_y-\phi_z\|_H\le\sqrt{2\pi}\,k|y-z|\) gives
\begin{equation}
    Q_{\tau,v,m}(z)
    \le
    \frac{1}{|\gamma_k|^2v_0|E_\delta|}
    +\frac{2\pi k^2(\delta+|z-y_0|)^2}{\tau}.
    \label{eq:local-high-variance-bound}
\end{equation}
A strong independent source segment can therefore reproduce nearby test
fields at low cost, providing a local mechanism for increased brightness.

	\subsection{  Boundary-source correlation and modal diversity}
	\label{sec:abstract-global-correlation}
	
In this subsection, we measure the modal diversity of the   boundary sources by regularized effective
dimensions and relate it to the spatial information carried by the
$Q$-statistic. We want to separate two effects: how the fluctuation
power is spread over the boundary source modes, and how large that power is in
total. To do this, we take mean-zero sources, $m=0$, and compare
nonnegative trace-class source covariances $C$ with the same total power
\[
  P:=\operatorname{tr}C=\mathbb E\|\varphi\|_{L^2(\Gamma)}^2 .
\]
Changing $C$ at fixed $P$ then only redistributes the spectrum of $C$.
For each such $C$, the data covariance and the corresponding statistic
are
\begin{equation}
  \Sigma_C:=A_\Gamma C A_\Gamma^*,
  \qquad
  Q_{\tau,C}(z):=Q_\tau(z,\Sigma_C).
  \label{eq:Sigma-C}
\end{equation}
	
\subsubsection{Fixed-power source modal diversity and effective ranks}
	
\begin{definition}[Modal diversity order]
  \label{def:modal-diversity-order}
  Let $C_1,C_2$ be nonzero nonnegative trace-class operators on
  $L^2(\Gamma)$ with $\operatorname{tr}C_1=\operatorname{tr}C_2$.
  For $i=1,2$, let $\kappa^{(i)}_1\ge\kappa^{(i)}_2\ge\cdots\ge0$ be the
  eigenvalues of $C_i$, repeated according to multiplicity and extended
  by zeros if $C_i$ has finite rank. We say that $C_1$ has
  \emph{at least as much modal diversity} as $C_2$, and write
  $C_1\succeq_{\rm div}C_2$, if
  \begin{equation}
    \sum_{j=1}^{M}\kappa_j^{(1)}
    \;\le\;
    \sum_{j=1}^{M}\kappa_j^{(2)}
    \qquad\text{for every } M\ge1 .
    \label{eq:fixed-power-majorization}
  \end{equation}
\end{definition}
	
Condition~\eqref{eq:fixed-power-majorization} says that the spectrum of
$C_2$ majorizes that of $C_1$. For every $M$, the $M$ leading modes of
$C_2$ carry at least as much power as those of $C_1$. Because the total
power is the same, $C_2$ has correspondingly less power in its trailing
modes, so its fluctuations are spread over fewer effective modes. To
measure this loss of modal diversity, we introduce regularized effective
ranks for both the source covariance and the data covariance.
	
	\begin{definition}[Regularized effective ranks]
    \label{def:global-effective-dimensions}
    For source and data regularization levels \(\beta,\tau>0\), define
    \begin{equation}
        \begin{aligned}
        \mathcal R_{\rm src}(C,\beta)
        &:=\operatorname{tr}[C(C+\beta I)^{-1}]
        =\sum_j\frac{\kappa_j}{\kappa_j+\beta},\\
        \mathcal R_{\rm far}(\Sigma_C,\tau)
        &:=\operatorname{tr}[\Sigma_C(\Sigma_C+\tau I)^{-1}]
        =\sum_j\frac{\lambda_j}{\lambda_j+\tau},
        \end{aligned}
        \label{eq:regularized-effective-dimensions}
    \end{equation}
    where \(\kappa_j\) and \(\lambda_j\) are the eigenvalues of
    \(C\) and \(\Sigma_C\).
\end{definition}
	Each weight is close to one for eigenvalues well above the
regularization level and small for those well below it,
giving a soft count of active modes.
The following proposition relates this count to source-side modal diversity order. 
	
	\begin{proposition}
		\label{thm:correlation-effective-dimension}
		 If \(C_2\succeq_{\rm div}C_1\), then, for every \(\beta>0\),
		\begin{equation}
			\mathcal R_{\rm src}(C_2,\beta)
			\ge
			\mathcal R_{\rm src}(C_1,\beta).
			\label{eq:source-correlation-effective-dimension-order}
		\end{equation}
		Moreover, for every $P>0$ and $\beta>0$, 
		\begin{equation}
			\sup_{{C\ge0, \operatorname{tr}C=P}}
			\mathcal R_{\rm src}(C,\beta)
			=
			\frac{P}{\beta}.
			\label{eq:correlation-free-effective-dimension-supremum}
		\end{equation}    
	\end{proposition}
	
	\begin{proof}
		For \(\beta>0\), the function
		\[
		f_\beta(x)=\frac{x}{x+\beta},
		\qquad x\ge0,
		\]
		is increasing, concave, and satisfies \(f_\beta(0)=0\).
		The trace-class form of Karamata's inequality \cite{MarshallOlkinArnold2011} applied to
		\eqref{eq:fixed-power-majorization} therefore gives
		\(
		\sum_j f_\beta(\kappa_j^{(2)})
		\ge
		\sum_j f_\beta(\kappa_j^{(1)}).
		\)
		
		On the other hand, if \(\{\kappa_j\}_{j\ge1}\) are the eigenvalues of \(C\), then
		\[
		\mathcal R_{\rm src}(C,\beta)
		=
		\sum_{j\ge1}\frac{\kappa_j}{\kappa_j+\beta}
		<
		\frac{1}{\beta}\sum_{j\ge1}\kappa_j
		=
		\frac{P}{\beta}.
		\]
		For \(C_d=(P/d)P_d\) where \(P_d\) is any rank-\(d\)
		orthogonal projection, its \(d\) nonzero eigenvalues are all \(P/d\), therefore
		\(
		\mathcal R_{\rm src}(C_d,\beta)
		=
		d\,\frac{P/d}{P/d+\beta}
		=
		\frac{P}{\beta+P/d}
		\longrightarrow
		\frac{P}{\beta}, d\rightarrow \infty.
		\)
	\end{proof}
    \begin{remark}[Recovery of the correlation-free reference]
\label{rem:fixed-power-white-reference}
Let $V_1\subset V_2\subset\cdots\subset L^2(\Gamma)$ be nested spaces
  for the   boundary sources, with $\dim V_n=n$, orthogonal
  projectors $P_n$, and union dense in $L^2(\Gamma)$. Among all
  nonnegative covariances supported on $V_n$ with trace $P$, the
  equal-power covariance
  \[
    C_n:=\frac{P}{n}\,P_n
  \]
  maximizes $\mathcal R_{\rm src}$. At fixed power,
  the data covariance $\Sigma_{C_n}=(P/n)\,A_\Gamma P_nA_\Gamma^*$ tends
  to zero as $n\to\infty$, so we rescale by $n/P$. Since $P_n\to I$
and $A_\Gamma$ is compact,
  \[
    \frac{n}{P}\,\Sigma_{C_n}
    =A_\Gamma P_nA_\Gamma^*
    \longrightarrow A_\Gamma A_\Gamma^*
    \qquad\text{in operator norm as } n\to\infty .
  \]
  Thus, after power rescaling, the most diverse fixed-power sources
  recover the correlation-free reference of
  Section~\ref{subsec:correlation-free-resolution}.
\end{remark}
		We note that the quotient
		\( |\Gamma|/\mathcal R_{\rm src}(C,\beta) \) can be interpreted as the average boundary length per effective mode remaining active above the floor \(\beta\), which carries information on the correlation length of the boundary sources. 
		The fixed-power diffusion model in
Appendix~\ref{app:heat-kernel-model}
provides a concrete realization of this ordering: increasing the spatial correlation length
concentrates the covariance spectrum and reduces the
source-side effective rank.

    \subsubsection{Radiation and spatial contrast}
	Proposition~\ref{thm:correlation-effective-dimension} concerns
source-side modal diversity. The far-field covariance depends
on both the source spectrum and the alignment of its eigenvectors
with \(A_\Gamma\). Nevertheless, the source spectrum and the radiation singular
values jointly constrain the far-field spectrum, yielding the
following effective-rank bound.

\begin{proposition}[Source-to-far-field spectral bounds]
\label{prop:source-to-far-effective-dimension}
Let \(C\) be a nonnegative trace-class operator and
\(A_\Gamma\) the radiation operator. Denote by \(\kappa_j\)
and \(\lambda_j\) the eigenvalues of \(C\) and
\(\Sigma_C=A_\Gamma C A_\Gamma^*\), respectively, and by
\(s_j\) the singular values of \(A_\Gamma\), all in
nonincreasing order. Then
\begin{equation}
    \lambda_j
    \le
    \min\{s_1^2\kappa_j,\;s_j^2\kappa_1\},
    \qquad j\ge1.
    \label{eq:spectral-bottleneck-upper}
\end{equation}
Consequently, for every \(\tau>0\),
\begin{equation}
    \mathcal R_{\rm far}(\Sigma_C,\tau)
    \le
    \mathcal R_{\rm src}
    \left(C,\frac{\tau}{\|A_\Gamma\|^2}\right).
    \label{eq:source-to-far-effective-dimension-bound}
\end{equation}
\end{proposition}

\begin{proof}
Set \(B=A_\Gamma C^{1/2}\). Since \(\Sigma_C=BB^*\),
its eigenvalues satisfy \(\lambda_j=s_j(B)^2\).
The standard singular-value inequalities \cite{Simon2005} give
\[
    s_j(B)
    \le
    \min\!\left\{
        \|A_\Gamma\|\,s_j(C^{1/2}),
        \;s_j(A_\Gamma)\,\|C^{1/2}\|
    \right\}
    =
    \min\{s_1\sqrt{\kappa_j},\;s_j\sqrt{\kappa_1}\}.
\]
Squaring yields the spectral bounds.
Since \(x\mapsto x/(x+\tau)\) is increasing,
\[
\begin{aligned}
    \mathcal R_{\rm far}(\Sigma_C,\tau)
    =\sum_j\frac{\lambda_j}{\lambda_j+\tau}\le
    \sum_j\frac{s_1^2\kappa_j}{s_1^2\kappa_j+\tau}
    =
    \mathcal R_{\rm src}
    \left(C,\frac{\tau}{\|A_\Gamma\|^2}\right).
\end{aligned}
\]
\end{proof}
	
Proposition~\ref{prop:source-to-far-effective-dimension} shows that two
factors limit which far-field modes are available: the statistics of the   boundary sources
and the radiation operator. The far-field effective rank at floor $\tau$
is at most the source effective rank at the matched floor
$\beta=\tau/\|A_\Gamma\|^2$. Radiation can lower it further. $A_\Gamma$
strongly attenuates some source directions, so their far-field
contributions fall below the floor $\tau$ even when they are well
excited.
    
    To relate these spectral constraints to spatial contrast of \(Q\), let \((\lambda_j,u_j)\) denote the positive
eigenpairs of \(\Sigma_C\). Rewriting the spectral representation by 	the resolvent identity
		\(
		(\Sigma+\tau I)^{-1}
		=
		\frac1\tau I
		-
		\frac1\tau
		\Sigma(\Sigma+\tau I)^{-1}
		\)
gives
\begin{equation}
    Q_{\tau,C}(z)
    =
    \frac{2\pi}{\tau}
    -
    \frac1\tau
    \sum_{\lambda_j>0}
    \frac{\lambda_j}{\lambda_j+\tau}
    |\langle\phi_z,u_j\rangle_H|^2.
    \label{eq:Q-baseline-modal-correction}
\end{equation}

The first term is independent of \(z\).
The spectral weights in the second term are precisely the summands
defined in \(\mathcal R_{\rm far}(\Sigma_C,\tau)\). A mode with
	\(\lambda_{j}\gg\tau\) can lower \(Q\) substantially and differently at
	different sampling points while the modes with \(\lambda_{j}\ll\tau\) can
	barely move \(Q\) away from the common baseline.

Consequently, when increasing correlation concentrates the covariance spectrum
	so that more data eigenvalues fall below \(\tau\), fewer terms remain capable
	of distinguishing sampling points.  Nearby tests whose projections differ
	mainly in the suppressed modes then yield similar values of the \(Q\)-statistic, and the
	spatial variation of both \(Q_\tau\) and \(I_\tau=1/Q_\tau\) may be weakened. 

	\subsubsection{Local correlation and spatial contrast}
\label{sec:abstract-local-correlation}

To examine local correlation effects, let \(E\subset\Gamma\)
be a boundary patch of positive arclength and \(R_E\) the
restriction to \(L^2(E)\). Set
\[
C_E=R_ECR_E^*,\qquad
A_E=A_\Gamma R_E^*,\qquad
\Sigma_E=A_EC_EA_E^*,
\]
and define
\begin{equation}
    Q_{\tau,E}(z):=Q_\tau(z,\Sigma_E).
    \label{eq:local-correlation-Q}
\end{equation}
This statistic isolates   boundary sources on \(E\), omitting sources
on its complement and cross-correlations between the two
regions. For sampling points close to a regular boundary
segment, \(E\) may be chosen around the nearest boundary
point. In a locally dominated regime, where the omitted
contributions have little effect on the statistic at these
points, \(Q_{\tau,E}\) provides a local approximation to
\(Q_{\tau,C}\). 

The modal representation~\eqref{eq:Q-baseline-modal-correction}
applies with \(\Sigma_C\) replaced by \(\Sigma_E\).
For example, spatially uniform, perfectly correlated
fluctuations on \(E\) are modeled by
\[
C_E=\gamma_E\psi_E\otimes\psi_E^*,
\qquad
\psi_E=|E|^{-1/2}\mathbf1_E,
\qquad \gamma_E>0.
\]
The pointwise variance is \(\gamma_E/|E|>0\) while
\(C_E\) has rank one. With \(a_E=A_E\psi_E\), the
rank-one identity gives
\begin{equation}
    Q_{\tau,E}(z)
    =
    \frac{2\pi}{\tau}
    -
    \frac{\gamma_E|\langle\phi_z,a_E\rangle_H|^2}
    {\tau\bigl(\tau+\gamma_E\|a_E\|_H^2\bigr)}.
    \label{eq:rank-one-local-Q}
\end{equation}
Thus when the patch dominates the response near \(E\),
distinguishing nearby sampling points relies on a single
far-field mode $a_E$, which can lead to weaker local imaging contrast, depending only on \(|\langle\phi_z,a_E\rangle_H|^2\).

	\section{Scattering-induced source statistics and imaging}
	\label{sec:physical-statistics-consequences}

    Section~\ref{sec:source-second-moment-imaging} examined how
boundary-source statistics affect imaging. In physical scattering,
these statistics are induced by the illumination and the obstacle
through the second-moment factorization $\Sigma=A_\Gamma S_\Gamma^{-1}\mathcal{K}_{i,\Gamma}
		S_\Gamma^{-*}A_\Gamma^*$ derived in \eqref{eq:second-moment-factorization}, where \(\mathcal K_{i,\Gamma}\) describes the incident-trace
second moment, \(S_\Gamma^{-1}\) encodes the effect of multiple scattering and represents the boundary response
that generates the induced source, and \(A_\Gamma\) maps this
source to the measured far field. The induced boundary-source statistics therefore
reflect both excitation and scatterer interaction, while their
contribution to the data is further filtered by radiation.

Since the induced statistics depend nonlocally on geometry and frequency, we focus on qualitative mechanisms, deriving simple
incident-side or asymptotic predictions where possible.
These mechanisms concern several representative scattering configurations, including limited-aperture illumination, scatterers with geometric
features, closely spaced obstacles, and point-source excitation, which are illustrated in
Figure~\ref{fig:scattering-configurations}, and their effects
on source statistics and imaging are examined numerically in
Section~\ref{subsec:part3-physical-statistics} and Section~\ref{subsec:general-incidence-Q} for the point-source excitation.
	\begin{figure}[htbp]
    \centering
    \includegraphics[width=0.7\linewidth]
    {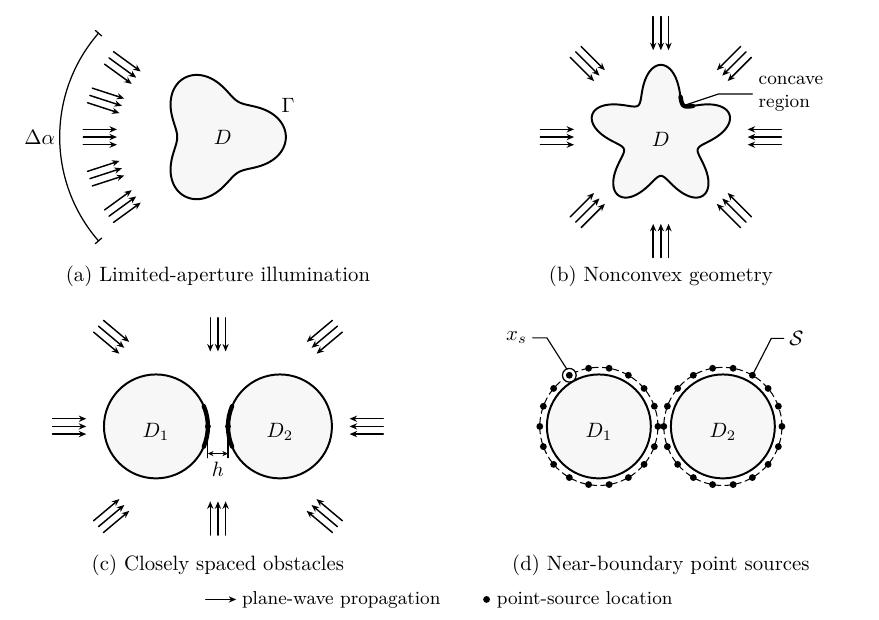}
    \caption{
        Representative configurations for scattering-induced
        source statistics.
    }
    \label{fig:scattering-configurations}
\end{figure}
	\subsection{Limited-aperture illumination}

Limited-aperture illumination restricts angular diversity without
reducing the incident mean-square amplitude, since
\(K_{\mu,\Gamma}(y,y)=1\) under the normalization of \(\mu\).
Its effects can be examined through incident spatial coherence and
induced source power.

Let \(y_0\in\Gamma\) be a smooth boundary point with outward unit
normal \(\nu_0\), and let \(y_\delta\) denote the boundary point at
signed arclength distance \(\delta\) from \(y_0\). For a uniform
angular window of width \(\Delta\alpha\ll1\) centered at \(d_0\),
linearization in arclength and incident angle gives an approximation to the incident second moment and hence the local coherence scale,
\begin{equation}
\left|K_{\mu,\Gamma}(y_0,y_\delta)\right|
\approx
\left|
\operatorname{sinc}\!\left(
\frac{k\delta\Delta\alpha}{2}\,d_0\cdot\nu_0
\right)
\right|,\qquad
\ell_{\rm inc}(y_0)
\asymp
\frac{1}{k\Delta\alpha\,|d_0\cdot\nu_0|}.
\label{eq:incident-aperture-correlation-scale}
\end{equation}
Thus a narrower aperture produces longer incident coherence.
As the angular window shrinks to \(d_0\),
\(
K_{\mu,\Gamma}(y,y')
\longrightarrow
e^{ikd_0\cdot y}e^{-ikd_0\cdot y'}.
\)
The limiting incident second moment has rank one, yielding source-side and data-side second moments of rank at most one by linearity.

A complementary high-frequency approximation describes the spatial
distribution of source power. For a smooth convex obstacle, the physical-optics approximation gives
\cite{ChandlerWildeGrahamLangdonSpence2012}
\[
\varphi_d(y)
\approx
2ik(-d\cdot\nu(y))_+e^{ikd\cdot y},
\qquad
a_+:=\max\{a,0\}.
\]
Consequently,
\begin{equation}
V_\varphi(y):=\mathbb E_\mu|\varphi_d(y)|^2
\approx
4k^2\int_{\mathbb S^1}
(-d\cdot\nu(y))_+^2\mu(d)\,d\sigma(d).
\label{eq:section5-physical-optics-variance}
\end{equation}
 The physical-optics approximation predicts weak source excitation
on the shadowed side, making this part of the boundary potentially
poorly visible in the image.

Together, these incident-side and high-frequency estimates provide a qualitative explanation of
limited-aperture imaging through incident correlation and source-power
distribution. Boundary regions facing the aperture center combine
strong excitation with a short incident correlation length, favoring
clearer imaging. Along the illuminated boundary, more oblique incidence
weakens excitation and increases the predicted local correlation length,
suggesting progressively weaker imaging contrast.
	
	\subsection{Nonconvex geometry}
	\label{subsec:geometry-concavity-rank}
	At fixed illumination, geometry also enters the induced second moment through
	\(S_\Gamma^{-1}\). Because this
	map is nonlocal, a boundary patch is driven by both the incident field and
	the scattered field returned from the rest of the boundary. Interference and multiple
	scattering may therefore change its local variance and correlation. Concave regions are natural locations
	for such effects. Accordingly,
	Section~\ref{sec:numerical-experiments} examines the boundary source statistics around these recessed regions, where
relatively low source variance and strong local correlation are
observed.

    \subsection{Closely spaced obstacles}
\label{subsec:close-obstacles-gap-blurring}

For closely spaced obstacles, each component is driven by the incident
field together with the field scattered by the other component.
This mutual interaction can modify both the strength and correlation
of the induced sources on the facing arcs. Full-aperture illumination
therefore does not ensure uniform excitation of the boundaries across the gap.

In the configurations considered in Section~\ref{sec:numerical-experiments}, relatively low source
variance and strong cross-correlation are observed on parts of the
facing arcs, where the reconstructed boundaries are less clearly
resolved. The former indicates weak fluctuation power, while the
latter indicates closely linked fluctuations across the gap.
Both features may contribute to reduced contrast between the
facing boundaries.

For the same geometry and far-field observation setup, near-boundary
point-source illumination and prescribed correlation-free sources
improve the visibility of the facing arcs
(Figure~\ref{fig:point-source-improvement}).
These comparisons suggest that the observed gap blurring depends
not only on separation, but also on how the facing boundaries
are excited.
	
	\subsection{Near-boundary point-source illumination}
	\label{subsec:point-source-resolution-enhancement}
	Let point-source locations $x_s$ be distributed according to a probability
	measure $\pi_{\rm s}$ on
	\(\mathcal S\subset\mathbb R^2\setminus\overline D\), yielding the incident-trace second-moment kernel
\begin{equation}
		K_{\pi_{\rm s},\Gamma}^{\rm pt}(y,y')
		=
		\int_{\mathcal S}
		\Phi_k(y,x_s)\,
		\overline{\Phi_k(y',x_s)}
		\,d\pi_{\rm s}(x_s).
		\label{eq:point-source-physical-sandwich}
	\end{equation}
	A nearby point source produces strong spatial variation in the
incident trace near the closest boundary region. Sampling different
source locations changes these local excitation patterns and can
strengthen boundary variations that are weakly excited by plane waves.
This modifies the source statistics while leaving the radiation
operator and far-field observation geometry unchanged.
Figure~\ref{fig:point-source-improvement} illustrates improved
visibility of the facing arcs as the source circles approach
the boundaries.

\section{Numerical Experiments}
\label{sec:numerical-experiments}

We first prescribe boundary-source statistics to isolate the mechanisms of
Section~\ref{sec:source-second-moment-imaging}, and then examine the statistics
generated by physical scattering. Throughout this section, the wavelength is
\(\lambda=0.5\), the far field is observed in \(100\) uniformly spaced directions,
and the perturbed far-field second-moment
matrix \(\widehat\Sigma\) is Hermitian positive semidefinite and satisfies
\(\|\widehat\Sigma-\Sigma\|_2/\|\Sigma\|_2=10^{-2}\), with \(\tau=\|\widehat\Sigma-\Sigma\|_2\).
All images display
\(I_{Q,\tau}=1/Q_\tau\), normalized by their individual maxima. 

\subsection{Prescribed boundary-source statistics}
\label{subsec:part2-artificial-statistics}

For each fixed geometry, we prescribe the boundary-source
statistics directly, without solving a scattering equation.
We adopt a Gaussian kernel-convolution construction
\cite{PaciorekSchervish2006}, adapted to boundary arclength
coordinates. In quadrature-normalized source coordinates,
let \(m_h\) denote the mean, \(v_j\ge0\) the variance, and
\(L_j=\lambda l_j>0\) the correlation-length parameter at node
\(s_j\). When all \(L_j\) coincide, we denote their common value by \(L\). To prescribe variance independently of correlation,
we normalize the Gaussian windows and define the covariance
through their overlaps,
\begin{equation}\label{pf}
    C_{ij}=\sqrt{v_i v_j}\,\rho_{ij},
    \qquad
    \rho_{ij}
    =
    \frac{\langle g_i,g_j\rangle_h}
         {\|g_i\|_h\|g_j\|_h},
    \qquad
    g_j(s)=
    \exp\!\left[-\frac{2d_\Gamma(s_j,s)^2}{L_j^2}\right],
\end{equation}
where \(d_\Gamma\) is periodic arclength distance and
\(
    \langle f,g\rangle_h
    :=
    \sum_p w_p f(s_p)\overline{g(s_p)}
\)
with \(C_{jj}=v_j\). Different boundary components have
independent fluctuations. For well-resolved windows, neglecting
the periodic wraparound effects gives
the following explicit local covariance approximation, 
\[
    C_{ij}\approx
    \sqrt{v_i v_j}
    \sqrt{\frac{2l_i l_j}{l_i^2+l_j^2}}
    \exp\!\left[
        -\frac{2d_\Gamma(s_i,s_j)^2}
        {\lambda^2(l_i^2+l_j^2)}
    \right],
\]
which agrees with the normalized diffusion model in Appendix~\ref{app:heat-kernel-model} for uniform correlation length. The data second moment is then
\(
    \Sigma_h=A_h(C+m_hm_h^*)A_h^*,
\)
where \(A_h\) is the quadrature-weighted radiation matrix.

\subsubsection{Coherent mean and local variance}
\label{subsubsec:mean-test}
\label{subsubsec:variance-test}

Using the boundary parameter \(\theta\), define
\[
a(\theta)=\frac12+\frac12\cos(\theta-\pi/2),
\]
shown in Figure~\ref{fig:artificial-mean}(a).
Figure~\ref{fig:artificial-mean} compares
\((m,v)=(ae^{i\theta},1)\) in panel~(b) with
\((m,v)=(0,a^2)\) in panel~(c), where the same profile \(a\) specifies the mean magnitude in~(b)
and the fluctuation standard deviation in~(c).
Both cases use \(L/\lambda=10^{-4}\), giving effectively uncorrelated
fluctuations. The boundary remains broadly visible in panel~(b), whereas panel~(c) emphasizes the high-variance portion. This is consistent with the different roles of mean and variance:
the mean contributes through a single coherent data direction,
while the variance profile controls local fluctuation strengths
along the boundary.

\begin{figure}[htbp]
 \centering
 \subfloat[]{\includegraphics[width=0.25\linewidth]{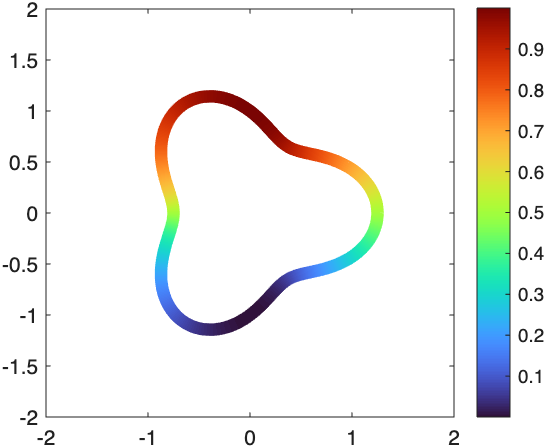}}\qquad
 \subfloat[]{\includegraphics[width=0.25\linewidth]{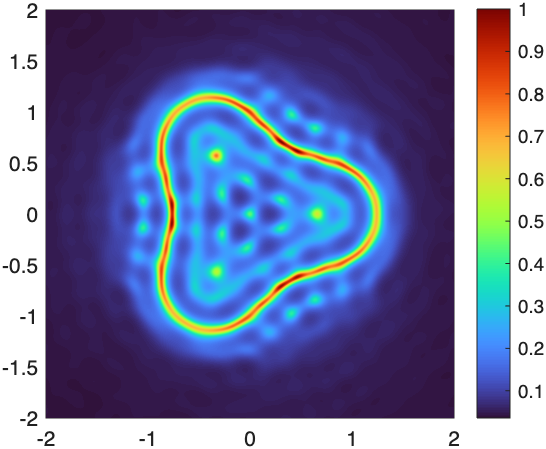}}\qquad
 \subfloat[]{\includegraphics[width=0.25\linewidth]{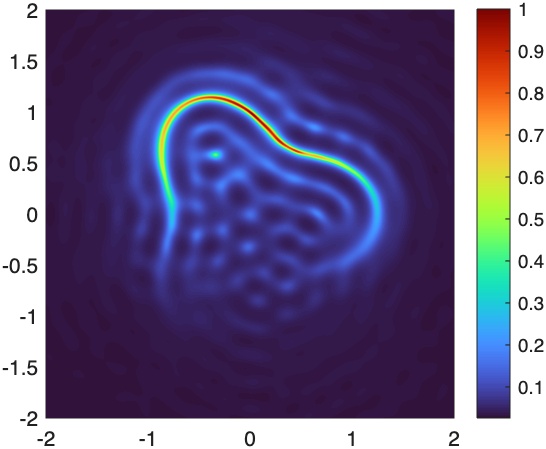}}
 \caption{Coherent mean and local variance with \(L/\lambda=10^{-4}\).
 (a) The common profile \(a\).
 (b) \(Q\)-imaging with \((m,v)=(ae^{i\theta},1)\).
 (c) \(Q\)-imaging with \((m,v)=(0,a^2)\).}
 \label{fig:artificial-mean}
\end{figure}

\subsubsection{Uniform correlation length at fixed variance}
\label{subsubsec:global-correlation-test}

With \(m=0\) and \(v=1\), we vary the uniform correlation length
over \(L/\lambda=1,0.6,0.1\).
Figure~\ref{fig:artificial-global-correlation}(a) shows the
eigenvalues of the noisy far-field covariance matrices, each
normalized by its largest eigenvalue.

As the correlation length decreases, more eigenvalues remain above
the nominal relative noise level of \(10^{-2}\), consistent with
more modes contributing to spatial discrimination at the chosen
regularization level. Correspondingly, panels~(b)--(d) show a
transition from isolated interior peaks to a more continuous
and clearly defined boundary.

These comparisons illustrate improved modal visibility with weak boundary source correlation while the observable modes remain constrained by radiation
and noise, as in the correlation-free reference of
Section~\ref{subsec:correlation-free-resolution}.
\begin{figure}[!htbp]
 \centering
 \subfloat[]{\includegraphics[width=0.24\linewidth]{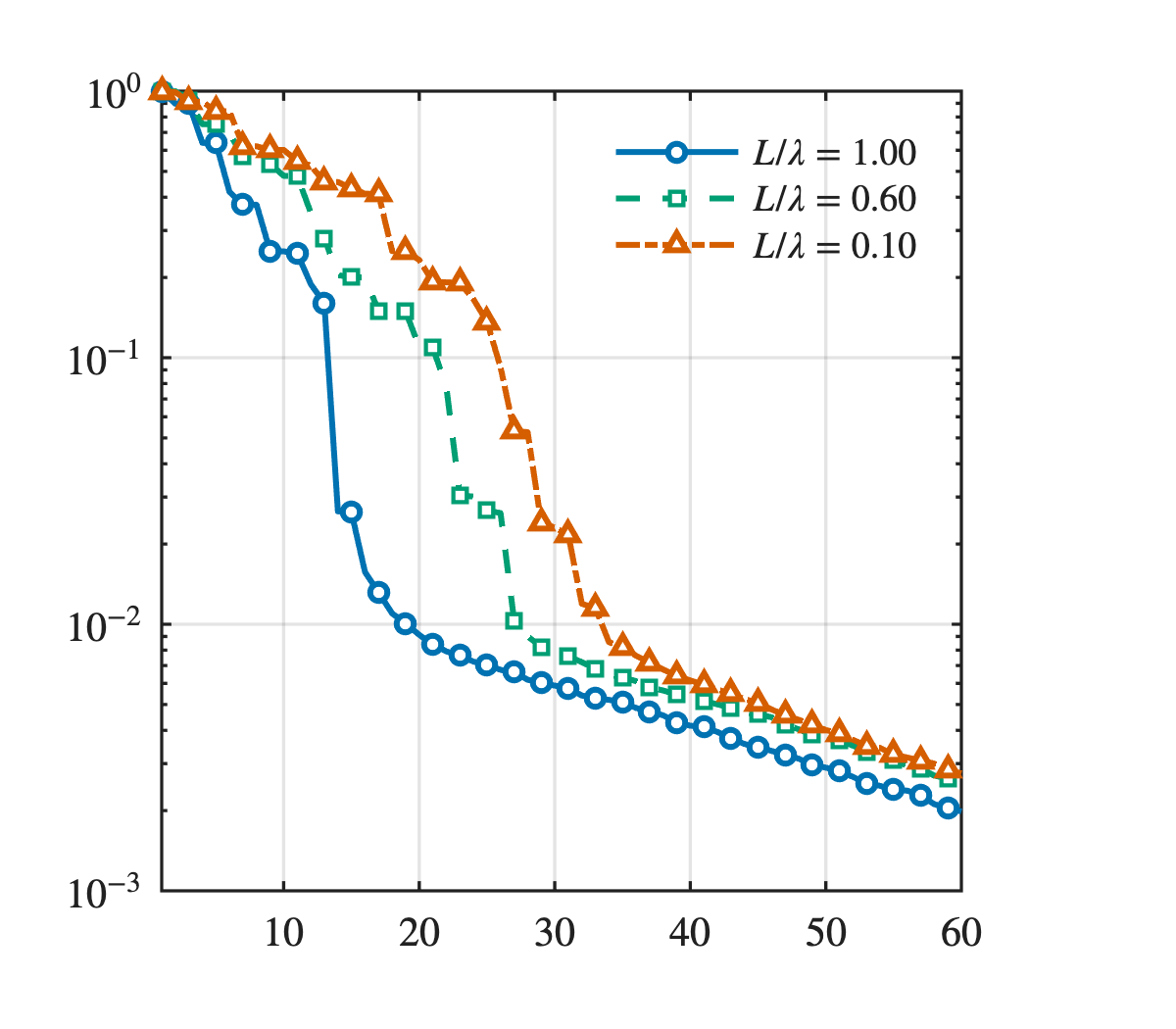}}\hfill
 \subfloat[]{\includegraphics[width=0.24\linewidth]{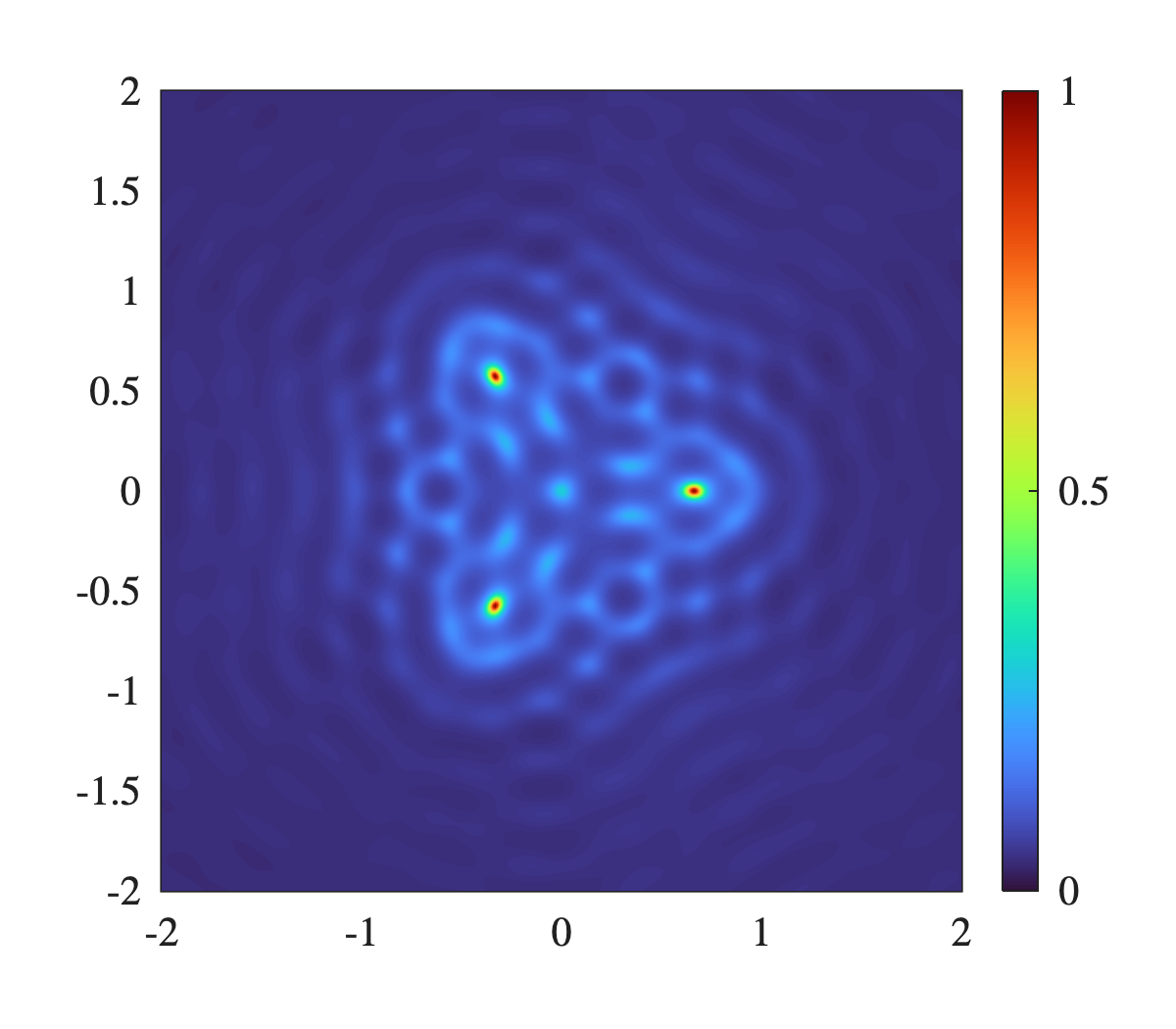}}\hfill
 \subfloat[]{\includegraphics[width=0.24\linewidth]{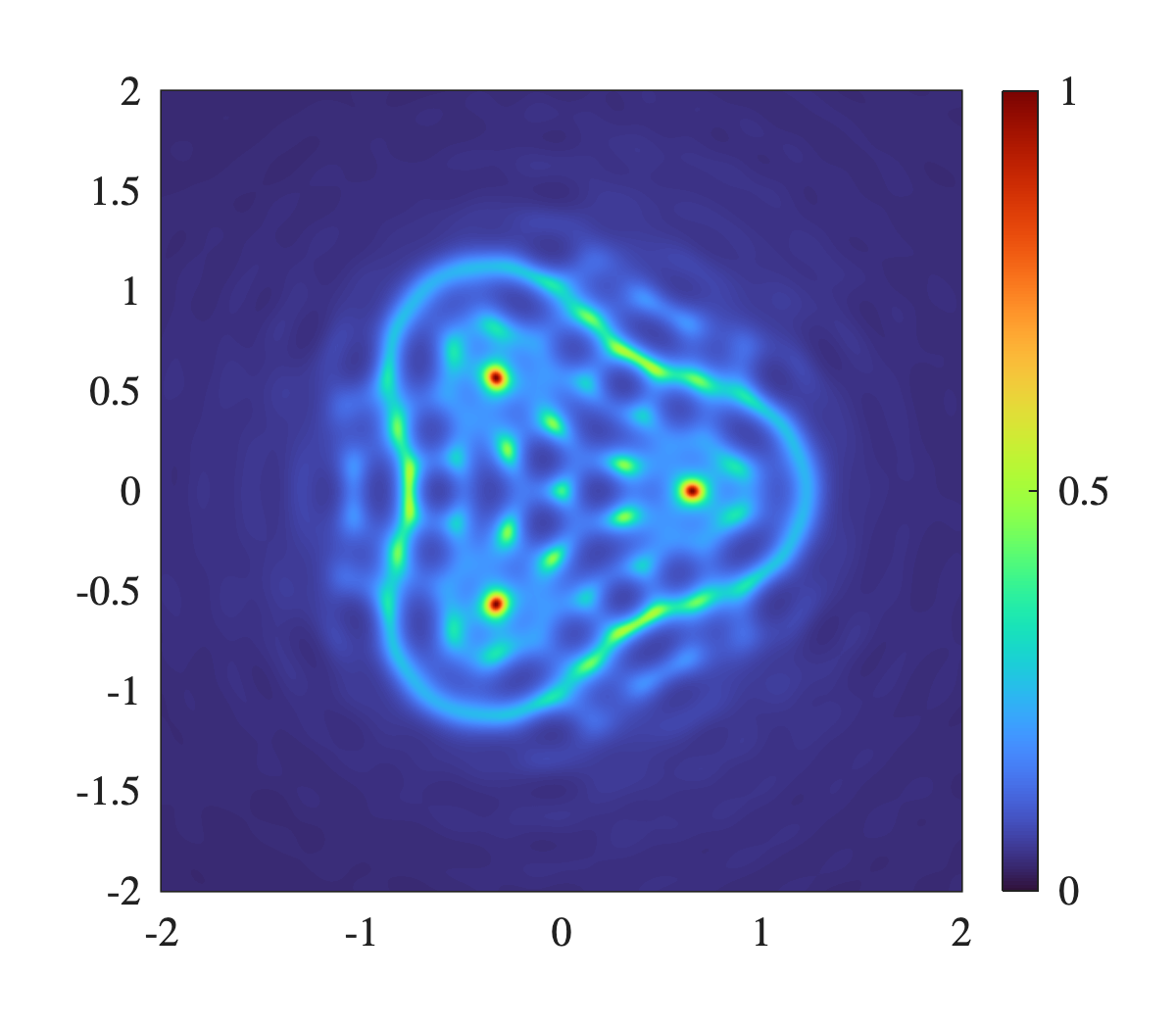}}\hfill
 \subfloat[]{\includegraphics[width=0.24\linewidth]{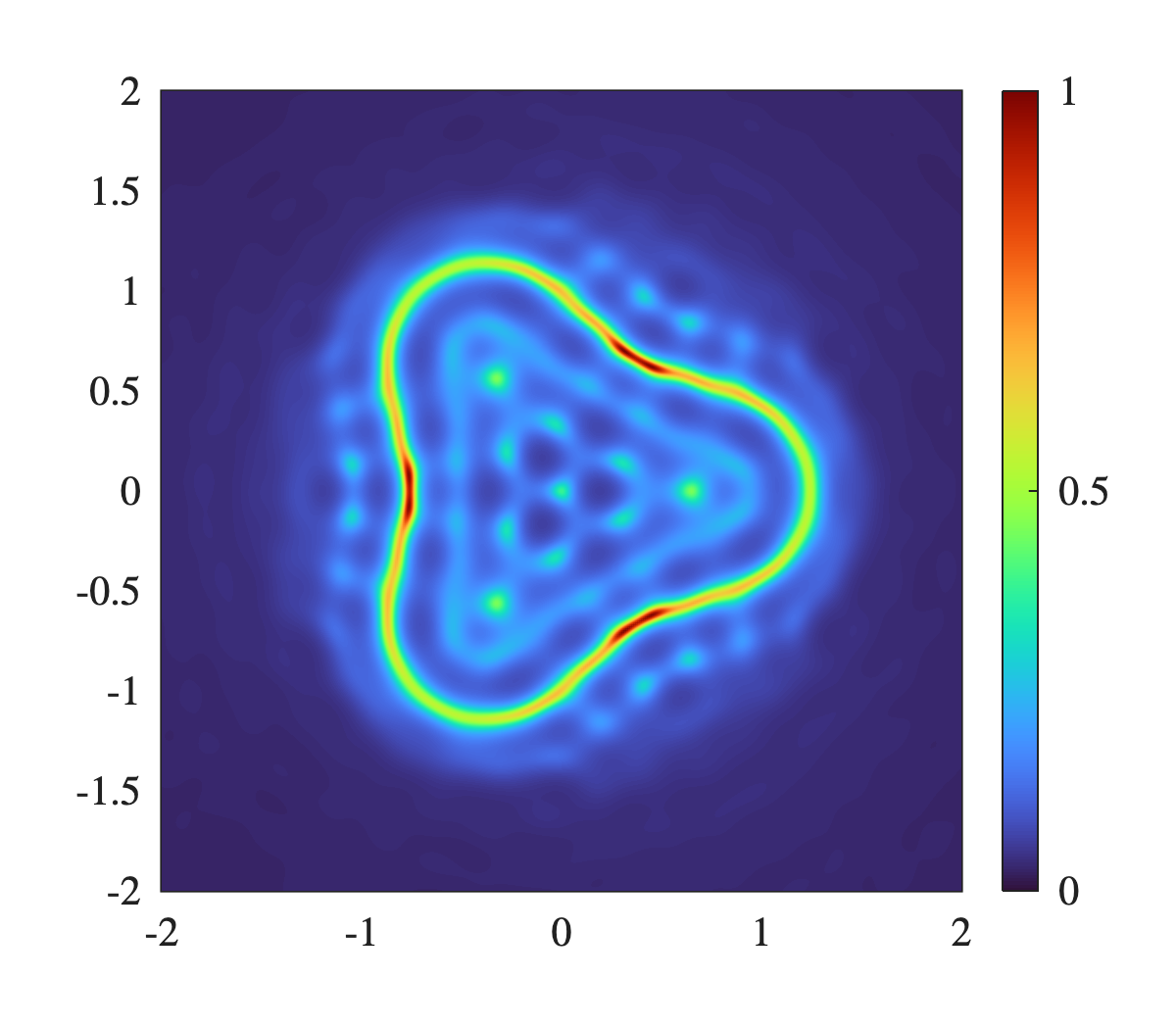}}
 \caption{Uniform source correlation with zero mean and unit variance.
(a) Eigenvalues of the noisy far-field covariance matrices,
normalized by their respective largest eigenvalues.
(b)--(d) Individually normalized \(1/Q_\tau\) images for
\(L/\lambda=1,0.6,0.1\), respectively.
The relative second-moment perturbation level is \(10^{-2}\).}
 \label{fig:artificial-global-correlation}
\end{figure}
\subsubsection{Spatially varying correlation}
\label{subsubsec:local-correlation-test}
With \(m=0\) and \(v=1\),
Figure~\ref{fig:artificial-local-correlation}(a)(c) shows the
source correlation magnitudes for
\[
\frac{L(\theta)}{\lambda}
=
\frac34\bigl[1+\cos(\theta-\pi/2)\bigr]+\epsilon_L,
\qquad
\frac34\bigl[1+\cos(\theta+2\pi/3)\bigr]+\epsilon_L,
\]
respectively, with
broader diagonal bands indicating longer-range correlation and $\epsilon_L=10^{-4}$ to ensure strict positivity of the correlation length. In the corresponding images in Figure~\ref{fig:artificial-local-correlation}(b)(d), the clearer
boundary contours shift from the lower to the upper portion
of the obstacle, consistent with the shift of the
short-correlation regions.

\begin{figure}[htbp]
 \centering
 \subfloat[]{\includegraphics[width=0.23\linewidth]{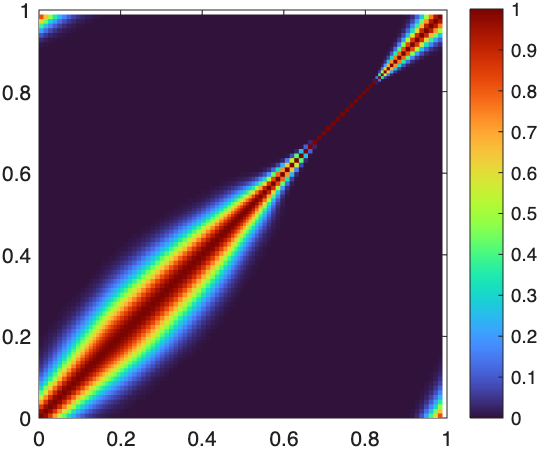}}\hfill
 \subfloat[]{\includegraphics[width=0.23\linewidth]{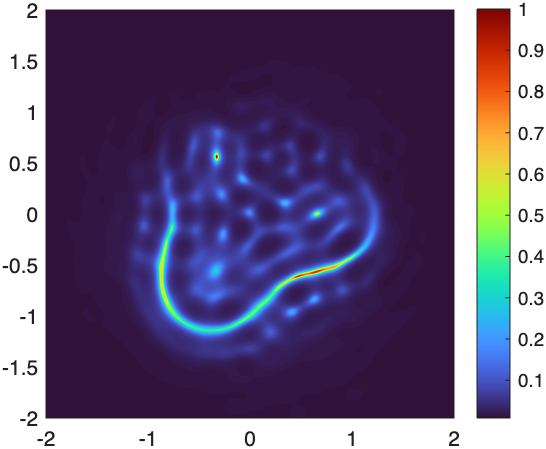}}\hfill
 \subfloat[]{\includegraphics[width=0.23\linewidth]{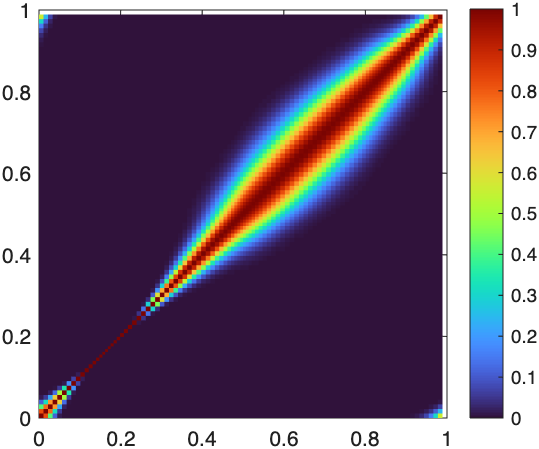}}\hfill
 \subfloat[]{\includegraphics[width=0.23\linewidth]{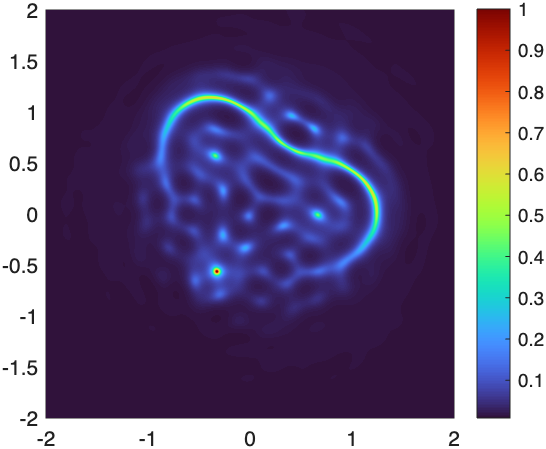}}
 \caption{Spatially varying correlation with zero mean and unit variance along the boundary of the pear-shaped obstacle.
(a)(c) Source correlation magnitudes for the two correlation-length
profiles, respectively, plotted against normalized arclength.
(b)(d) Corresponding individually normalized \(1/Q_\tau\) images.
}
 \label{fig:artificial-local-correlation}
\end{figure}

\subsubsection{Source statistics and facing-boundary visibility}
\label{subsubsec:facing-source-control}

We next examine how local variance and correlation affect the visibility of two facing boundaries at a fixed separation. Figure~\ref{fig:varcor} compares the effects of local variance
and correlation on two radius-\(0.7\) disks centered at
\((\pm0.8,0)\), with a gap of \(0.4\lambda\).
The mean is zero throughout.
The complementary profiles in panels~(a)(d) prescribe
the standard deviation in panels~(b)(e) with nearly independent
fluctuations, and the correlation length \(L/\lambda\)
in panels~(c)(f) with unit variance.

In panels~(b)(e), larger variance on the facing arcs makes
these arcs more prominent in the normalized images.
In panels~(c)(f), the facing arcs are more clearly delineated
when their local correlation length is shorter, despite
unchanged variance.
Thus, even at the same separation, local source statistics
can substantially affect facing-boundary visibility.

The near-boundary point-source example in
Figure~\ref{fig:point-source-improvement} complements these
prescribed-source experiments, illustrating how a change in illumination
can improve obstacle separation for the configuration considered.

\begin{figure}[htbp]
 \centering
 \subfloat[]{\includegraphics[width=0.25\linewidth]{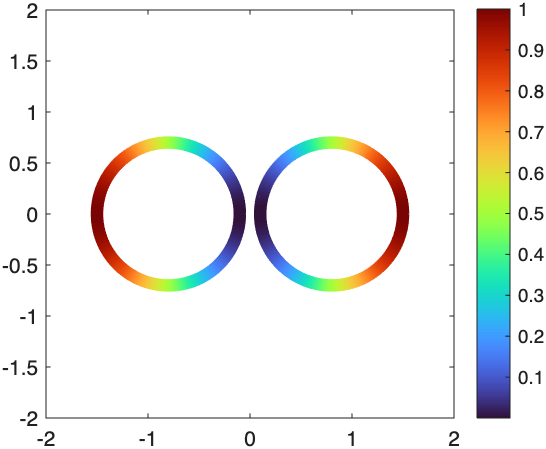}}\qquad
 \subfloat[]{\includegraphics[width=0.25\linewidth]{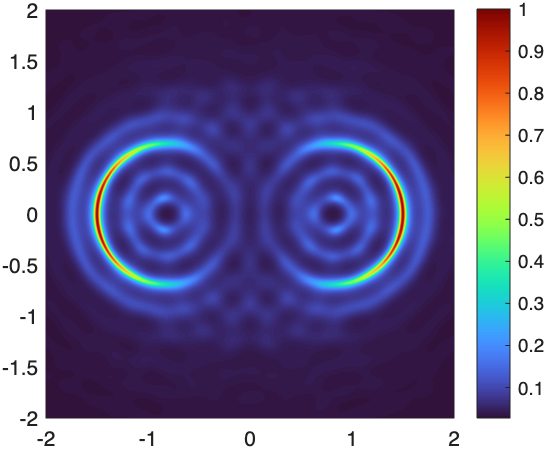}}\qquad
 \subfloat[]{\includegraphics[width=0.25\linewidth]{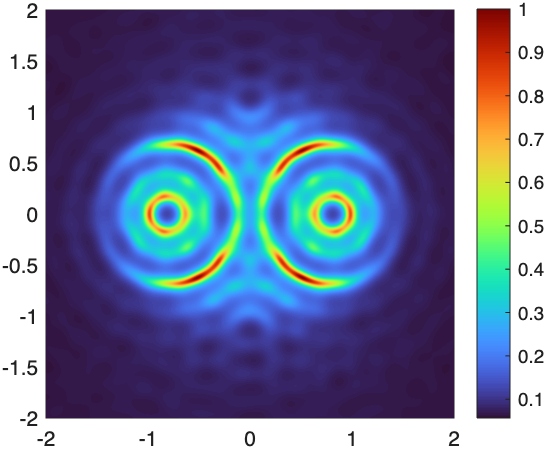}}\\[3pt]
 \subfloat[]{\includegraphics[width=0.25\linewidth]{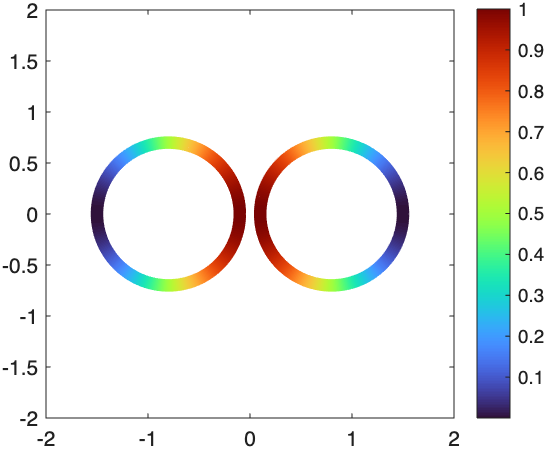}}\qquad
 \subfloat[]{\includegraphics[width=0.25\linewidth]{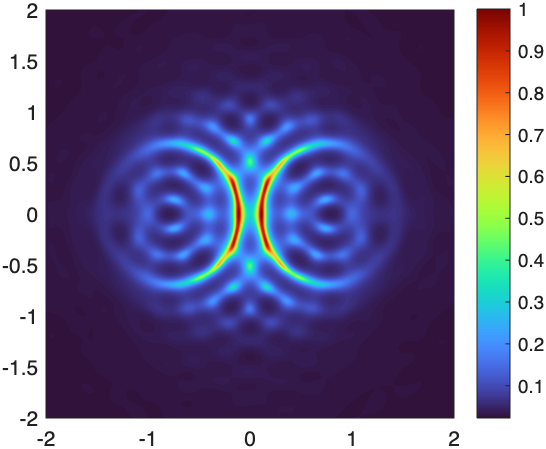}}\qquad
 \subfloat[]{\includegraphics[width=0.25\linewidth]{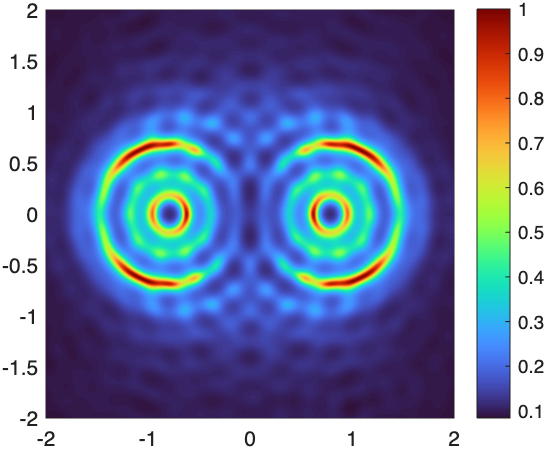}}
 \caption{Prescribed-source control for two disks with gap \(0.4\lambda\).
				Left: complementary profiles used as the fluctuation standard deviation in the middle
				column and as \(L/\lambda\) in the right column. Middle: images with
				standard deviation in (a)(d) and \(L/\lambda=10^{-4}\).
				Right: images with \(L/\lambda\) in (a)(d) and unit standard deviation.}
			\label{fig:varcor}
\end{figure}

\subsection{Scattering-induced boundary-source statistics}
\label{subsec:part3-physical-statistics}

In contrast to Section~\ref{subsec:part2-artificial-statistics}, the boundary sources are now
obtained from the scattering equation for prescribed illumination.
We compare their centered variance and correlation with
the resulting images to examine the mechanisms discussed
in Section~\ref{sec:physical-statistics-consequences}. Imaging uses the raw data second moment, including the coherent mean. The displayed source statistics are the centered variance $v_\varphi(s)=\mathbb E|\varphi(s)-m(s)|^2$ and the correlation magnitude $|C(s,s')|/\sqrt{v_\varphi(s)v_\varphi(s')}$. Boundary locations are shown in normalized arclength.

\subsubsection{Limited-aperture illumination}
\label{subsubsec:peanut-limited-aperture}

The incident plane-wave directions are uniformly distributed
over \((0,\pi)\), while far-field observations cover all directions.
In Figure~\ref{fig:peanut-limited-aperture}, the directly
illuminated boundary portions generally exhibit higher relative
source variance and narrower local correlation bands than
the weakly illuminated portions.
These portions are also more clearly delineated in the images, showing a qualitative correspondence between the computed source statistics and local boundary visibility.

\begin{figure}[htbp]
 \centering
 \subfloat[]{\includegraphics[width=0.75\linewidth]{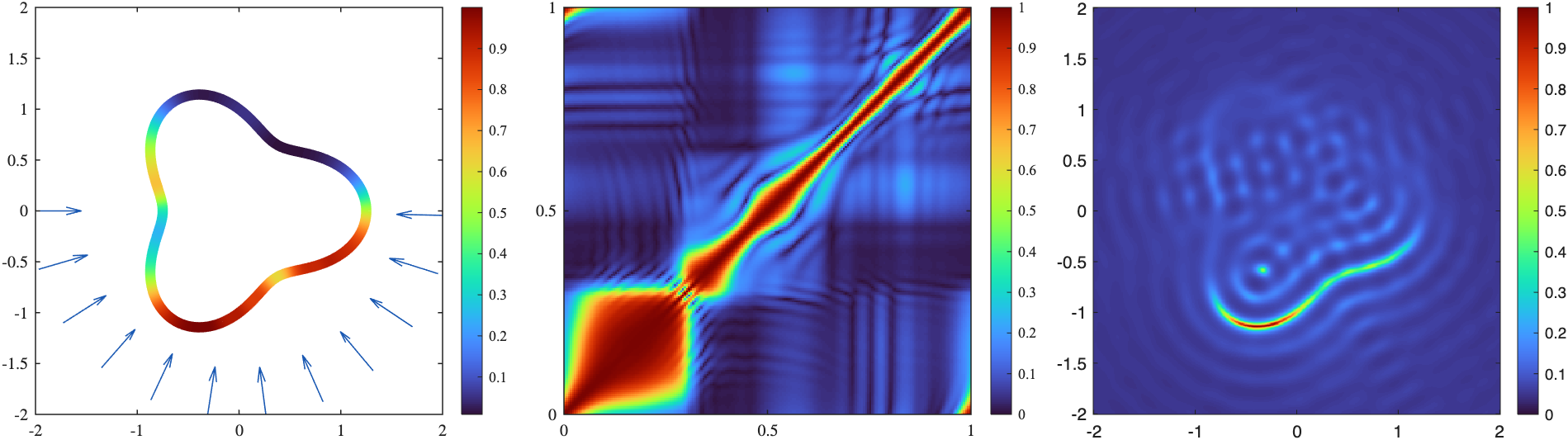}}\\[3pt]
 \subfloat[]{\includegraphics[width=0.75\linewidth]{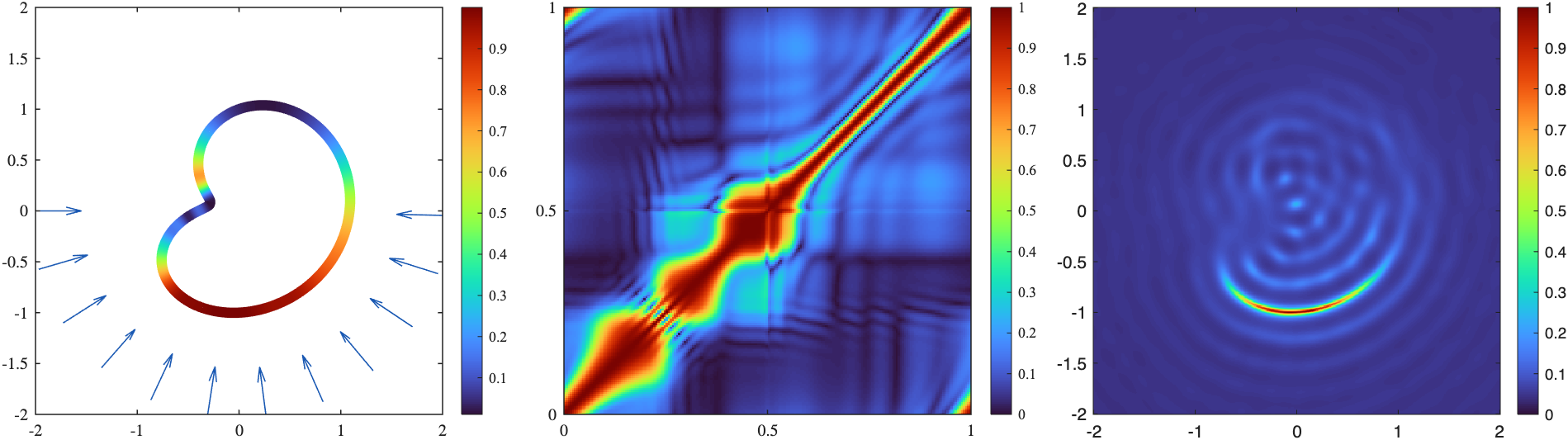}}
 \caption{Limited-aperture plane-wave illumination, with incident angles
 uniformly distributed over \((0,\pi)\), for (a) a pear and (b) a
 smooth nonconvex obstacle. Each row shows, from left to right: the incident
 aperture and normalized boundary-source variance; centered source
 correlation magnitude; and normalized \(1/Q_\tau\) image.}
 \label{fig:peanut-limited-aperture}
\end{figure}

\subsubsection{Nonconvex obstacles under full-aperture illumination}

Figure~\ref{fig:starfish-concavity} shows the computed source
statistics and images for two nonconvex obstacles under
full-aperture plane-wave illumination.
In these examples, the recessed boundary portions exhibit
lower relative source variance and stronger local correlation,
and are less clearly delineated in the images.
\begin{figure}[htbp]
 \centering
 \subfloat[]{\includegraphics[width=0.75\linewidth]{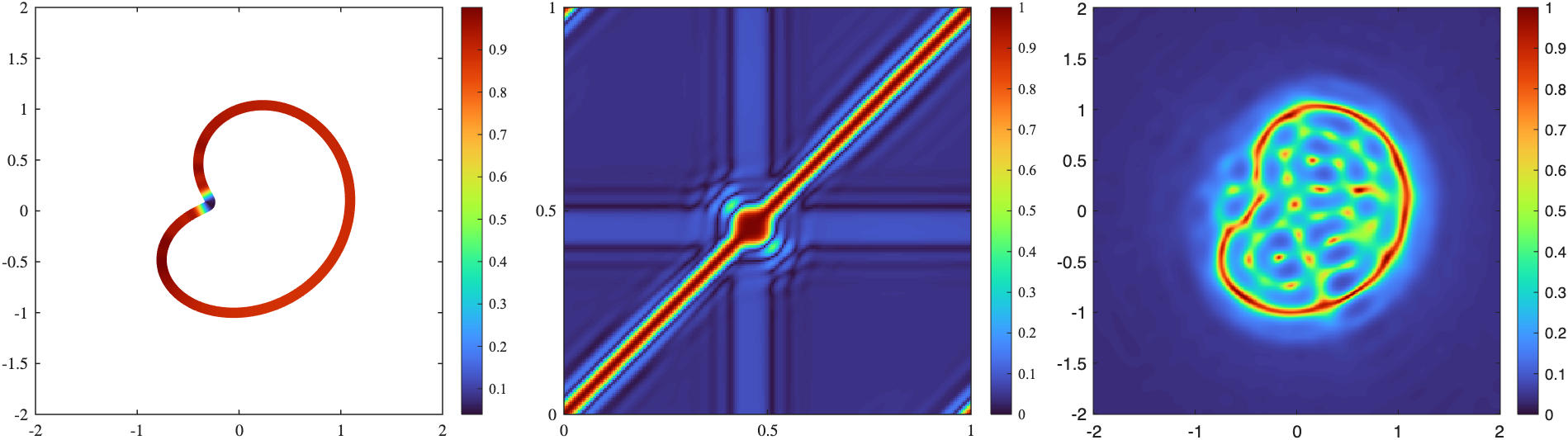}}\\[3pt]
 \subfloat[]{\includegraphics[width=0.75\linewidth]{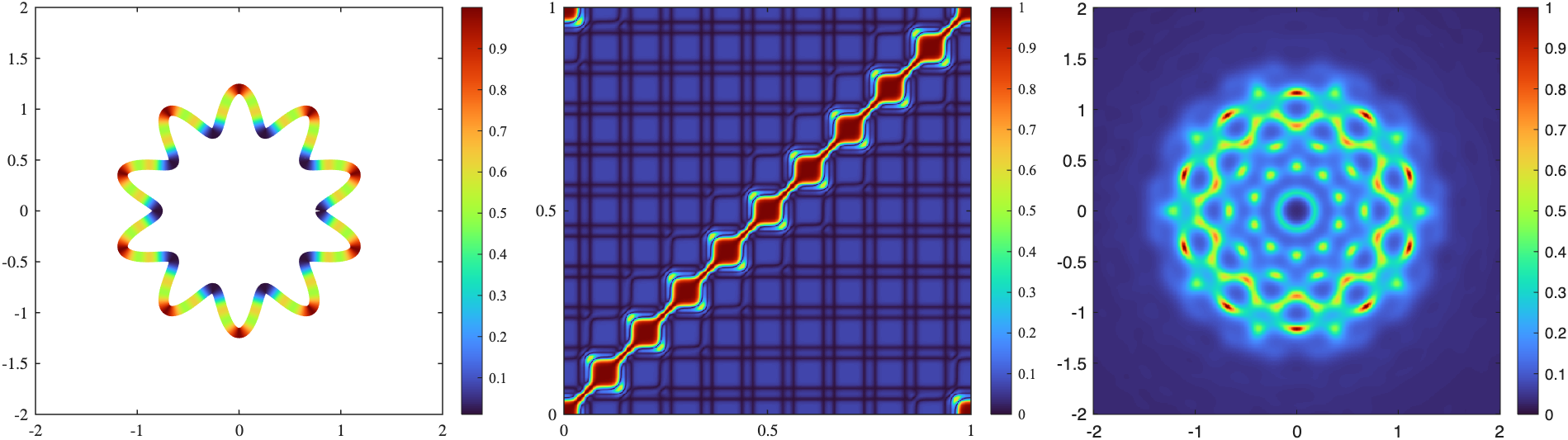}}
 \caption{Source statistics and imaging under full-aperture
plane-wave illumination for (a) a smooth nonconvex obstacle
and (b) a star-shaped obstacle.
Columns show normalized boundary-source variance,
centered source correlation magnitude, and individually
normalized \(1/Q_\tau\) images.}
 \label{fig:starfish-concavity}
\end{figure}

For the same boundaries,
Figure~\ref{fig:correlation-free-geometry} uses prescribed
equal-strength, correlation-free sources, with the observation
aperture and relative noise level unchanged.
The recess and star-shaped valleys become more clearly visible, consistent with the conformal scale $h_\Gamma(x)$ in Section \ref{confo}.
This comparison indicates that source statistics contribute
to the poor visibility of these features under plane-wave
illumination.
\begin{figure}[htbp]
 \centering
 \subfloat[]{\includegraphics[width=0.35\linewidth]{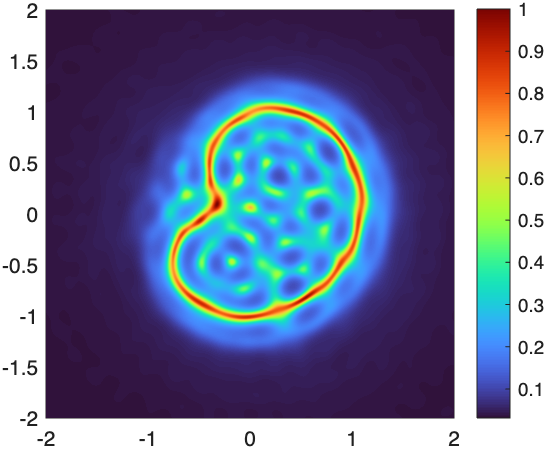}}\qquad
 \subfloat[]{\includegraphics[width=0.35\linewidth]{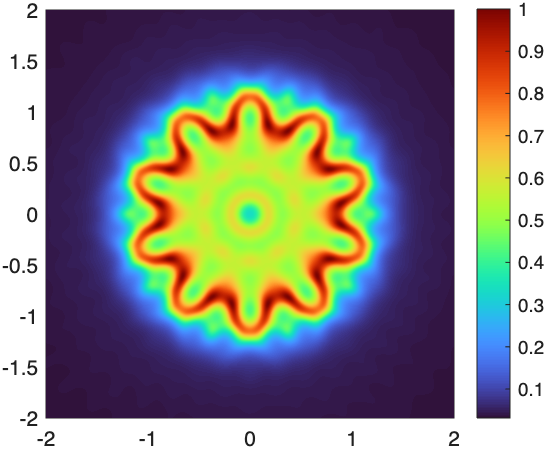}}
 \caption{Individually normalized \(1/Q_\tau\) images for the
same obstacles as in Figure~\ref{fig:starfish-concavity},
using prescribed equal-variance, correlation-free sources.
}
 \label{fig:correlation-free-geometry}
\end{figure}
\subsubsection{Closely spaced obstacles}
\label{subsubsec:two-ellipses-facing}
Figure~\ref{fig:two-ellipses-facing} compares two disk
configurations with gaps of \(2\lambda\) and \(0.4\lambda\)
under full-aperture plane-wave illumination.
At the smaller gap, the facing boundary contours are
less clearly separated.
The computed source statistics show low relative variance
on parts of the facing arcs and correlated fluctuations
between the two disks.

For the same close-spacing geometry, the prescribed-source
controls in Figure~\ref{fig:varcor} and the near-boundary
point-source and correlation-free examples in
Figure~\ref{fig:point-source-improvement} show that
facing-boundary visibility can be improved without changing
the far-field observation aperture.
These comparisons indicate that the observed gap blurring
depends on source statistics as well as separation.

\begin{figure}[!htbp]
 \centering
 \subfloat[]{\includegraphics[width=0.75\linewidth]{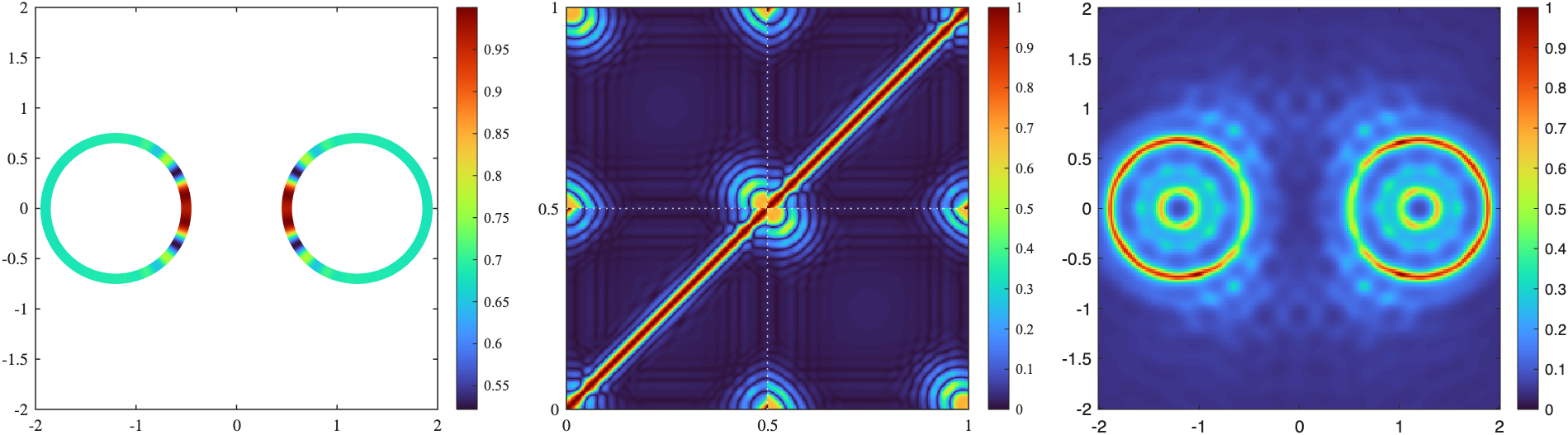}}\\[3pt]
 \subfloat[]{\includegraphics[width=0.75\linewidth]{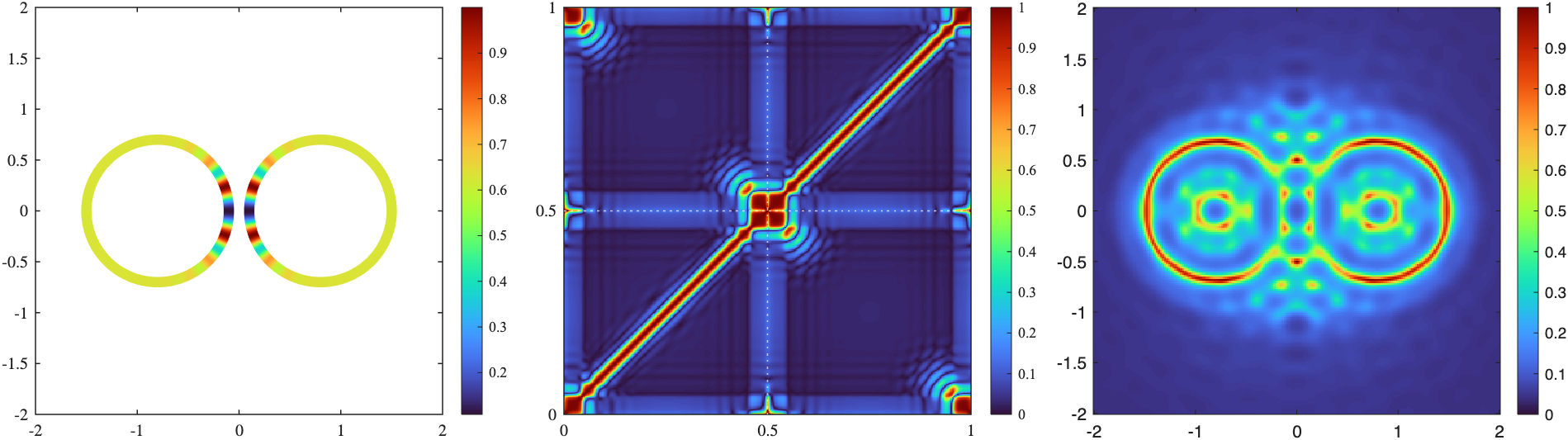}}
 \caption{Two radius-\(0.7\) disks under full-aperture plane-wave
 illumination, with centers at (a) \((\pm1.2,0)\) and
 (b) \((\pm0.8,0)\). Columns show boundary-source variance,
 centered source correlation magnitude, and normalized \(1/Q_\tau\).
 The dotted lines in the correlation plots separate the arclength
intervals of the two boundary components.}
 \label{fig:two-ellipses-facing}
\end{figure}
	
	\section{Conclusion}
\label{connnn}
We have developed a statistical framework for obstacle imaging
and resolution analysis from far-field second moments under
random illumination. The \(Q\)-statistic admits an interpretation as the minimal statistically weighted
source-synthesis cost of reproducing a test far field,
and the reciprocal of its regularized form  provides an imaging indicator.
Under uniform full-aperture plane-wave illumination, the identity
\(\Sigma=(2\pi)^{-1}FF^*\) yields statistical counterparts of
classical sampling methods, distinguished by their spectral
weighting.

The source-to-data factorization separates the effects of
boundary-source statistics from those of radiation and noise.
Within this framework, we define a correlation-free resolution
limit through the one-sided half-height width of the regularized
boundary transition. At fixed nonresonant wavenumber, we prove a universal boundary-layer profile and a sharp half-height asymptotic for the correlation-free regularized statistic on simply connected analytic boundaries, where the local conformal scale determines both the leading boundary intensity and the physical transition width.

For general boundary-source statistics, we characterize how
coherent mean, fluctuation variance, and correlation affect
imaging contrast through their distinct contributions to source
energy and modal diversity. Combined with the induced-source
factorization, this analysis provides qualitative explanations
of resolution mechanisms associated with limited illumination
aperture, nonconvex geometry, closely spaced obstacles, and
near-boundary point-source excitation. The numerical examples
illustrate these mechanisms and show that practical resolution
depends jointly on source statistics, radiation, and noise,
rather than on wavelength alone.

This statistical framework can be readily extended to other inverse scattering settings, including multiple frequencies, inverse medium scattering, and inverse electromagnetic and elastic scattering, which we plan to address in future work.  

\section*{Use of generative AI}
We used GPT-6 Astra (OpenAI) to improve the language and
readability of the manuscript and to suggest proof strategies
for Appendices A and C. We reviewed the AI-assisted content,
verified all mathematical arguments, and take full
responsibility for the manuscript.
	\appendix

\section{Random near-boundary point-source illumination}
\label{appendix-point source}

  We show that near-boundary random monopoles generate the
			correlation-free far-field reference of
			Section~\ref{subsec:correlation-free-resolution} in the limit.
			Let \(\Gamma=\bigcup_{j=1}^{m}\partial D_j\), where the bounded
			obstacles have pairwise disjoint closures and \(C^2\) simple closed
			boundaries, and let \(\nu\) be the outward unit normal.
			Fix \(k>0\) such that \(k^2\) is not a Dirichlet eigenvalue in any
			\(D_j\), and write \(S=S_\Gamma\).
			For sufficiently small \(\delta>0\), place a monopole of fixed
			amplitude \(a\ne0\) uniformly with respect to arclength on
			\(\Gamma_\delta=\{z_\delta(y)=y+\delta\nu(y):y\in\Gamma\}\).
			Writing its random location as \(z_\delta(Y)\), the density of
			\(Y\) on \(\Gamma\) is
			\begin{equation}
				\sigma_\delta(y)
				=\frac{1+\delta\kappa(y)}{|\Gamma|+2\pi m\delta}
				=\frac1{|\Gamma|}+\mathcal O(\delta)
				\label{eq:density}
			\end{equation}
			uniformly, where \(\kappa\) is curvature with
			\(\partial_s\nu=\kappa\,\partial_s y\).
			Set
			\[
			g_\delta(x)=a\Phi_k(x,z_\delta(Y)),\qquad
			\varphi_\delta=-S^{-1}g_\delta,\qquad
			y_\delta=A_\Gamma\varphi_\delta .
			\]
			Here \(g_\delta\) is the incident trace and \(y_\delta\) the
			scattered far field. Their raw second moments are
			\(\mathcal K_{i,\delta}=\mathbb E[g_\delta\otimes g_\delta^*]\)
			on \(L^2(\Gamma)\) and
			\(\Sigma_\delta=\mathbb E[y_\delta\otimes y_\delta^*]\)
			on \(H=L^2(\mathbb S^1)\).

			\begin{theorem}[Monopole far-field limit]
				\label{thm:monopole}
				\label{thm:monopole-farfield}
				With \(\kappa_0=|a|^2/|\Gamma|\), as \(\delta\downarrow0\),
				\begin{align}
					\|\mathcal K_{i,\delta}-\kappa_0 SS^*\|_{\mathrm{HS}}
					&\longrightarrow0,
					\label{eq:monopole-incident-limit}\\
					\|\Sigma_\delta-\kappa_0 A_\Gamma A_\Gamma^*\|_{H\to H}
					&\longrightarrow0.
					\label{eq:monopole-farfield-limit}
				\end{align}
			\end{theorem}

			\begin{proof}
				\emph{Incident traces.}
				Define the multiplication
				operator
                \(M_b:L^2(\Gamma)\to L^2(\Gamma)\) by
				\(
					(M_bf)(y)=b(y)f(y),
				\) and 
                the source-location-to-trace operator
				\(\mathcal T_\delta:L^2(\Gamma)\to L^2(\Gamma)\)
				by  
                \[
					(\mathcal T_\delta f)(x)
					=\int_\Gamma\Phi_k(x,z_\delta(y))f(y)\,ds(y).
				\]
				Averaging the outer product of \(g_\delta\) over its
				random location gives
				\[
				\mathcal K_{i,\delta}
				=|a|^2\mathcal T_\delta
				M_{\sigma_\delta}\mathcal T_\delta^*.
				\]
				The \(C^2\) geometry and separation of the components give
				\(|\nu(y)\cdot(x-y)|\le C|x-y|^2\), and hence
				\(|x-z_\delta(y)|^2\ge
				\tfrac12(|x-y|^2+\delta^2)\) for small \(\delta\).
				The logarithmic singularity of \(\Phi_k\) therefore implies
				\[
				|\Phi_k(x,z_\delta(y))|+|\Phi_k(x,y)|
				\le C(1+|\log|x-y||)
				\in L^2(\Gamma\times\Gamma).
				\]
				Since the kernels converge off the diagonal, dominated
				convergence gives \[
					\|\mathcal T_\delta-S\|_{\mathrm{HS}}^2
					=\int_{\Gamma\times\Gamma}
					|\Phi_k(x,z_\delta(y))-\Phi_k(x,y)|^2\,ds(x)\,ds(y)
					\longrightarrow0.
				\]
				Together with \eqref{eq:density} and the factorization \[\mathcal T_\delta M_{\sigma_\delta}\mathcal T_\delta^*
						-|\Gamma|^{-1}SS^*
						=(\mathcal T_\delta-S)M_{\sigma_\delta}\mathcal T_\delta^*
						+SM_{\sigma_\delta}(\mathcal T_\delta^*-S^*)+S(M_{\sigma_\delta}-|\Gamma|^{-1}I)S^*,\] the Hilbert--Schmidt product
				estimates yield
				\[
				\|\mathcal K_{i,\delta}-\kappa_0 SS^*\|_{\mathrm{HS}}
				\le C\left(\|\mathcal T_\delta-S\|_{\mathrm{HS}}
				+\|\sigma_\delta-|\Gamma|^{-1}\|_\infty\right)\to0.
				\]

				\emph{Scattered far field.}
				Standard single-layer regularity and the nonresonance condition give
				bounded inverses
				\(S^{-1},(S^*)^{-1}:H^1(\Gamma)\to L^2(\Gamma)\)
				\cite{coltonInverseAcousticElectromagnetic2019,KirschGrinberg2008},
				while \(A_\Gamma^*:H\to H^1(\Gamma)\) is bounded.
				Here \(S^*\) is the \(L^2\) adjoint.
				Thus \(W=(S^*)^{-1}A_\Gamma^*\) is bounded, and
				\(\mathcal B_\Gamma=-W^*:L^2(\Gamma)\to H\)
				extends \(-A_\Gamma S^{-1}\), since for \(g\in H^1(\Gamma)\),
				\(f\in H\),
				\[
				\langle-A_\Gamma S^{-1}g,f\rangle_H
				=-\langle g,Wf\rangle_{L^2(\Gamma)}
				=\langle\mathcal B_\Gamma g,f\rangle_H.
				\]
				In particular, \(\mathcal B_\Gamma S=-A_\Gamma\).
				For fixed \(\delta>0\), \(g_\delta\in H^1(\Gamma)\), so
				\(\Sigma_\delta=
				\mathcal B_\Gamma\mathcal K_{i,\delta}\mathcal B_\Gamma^*\).
				Consequently,
				\[
				\|\Sigma_\delta-\kappa_0 A_\Gamma A_\Gamma^*\|
				\le \|\mathcal B_\Gamma\|^2
				\|\mathcal K_{i,\delta}-\kappa_0 SS^*\|_{\mathrm{HS}}
				\longrightarrow0.
				\]
			\end{proof}
    
    \section{Regularized Disk Estimates}
\label{app:disk-regularized-estimates}

This appendix supplies the estimates used in
Theorem~\ref{thm:disk-halfheight-resolution}. Throughout, $R,k>0$ are fixed,
$J_n(kR)\ne0$ for all $n\in\mathbb Z$, and $N=N_\tau$ is defined by
\eqref{eq:disk-effective-order}. Recall that
$L_\tau=\log(\|A\|^2/\tau)$. All estimates concern the full regularized
series \eqref{eq:disk-exact-Q}.

\begin{lemma}[Regularized series and effective angular order]
	\label{lem:disk-regularized-estimates}
	As $\tau\downarrow0$, $N\to\infty$ and, for every fixed $T>0$,
	\begin{align}
		2\pi R Q_\tau^{\rm disk}(R)&=2N+\mathcal O(1),
		\label{eq:appendix-disk-boundary-value}\\
		2\pi R Q_\tau^{\rm disk}(R(1+t/N))
		&=1+2\sum_{n=1}^{N}\left(1+\frac{t}{N}\right)^{2n}
		+O_T(\log N),
		\label{eq:appendix-disk-regularized-sum}
	\end{align}
	uniformly for $0\le t\le T$. Moreover,
	\begin{equation}
		L_\tau=2N\log\frac{2N}{\mathrm e kR}+\mathcal O(\log N),
		\qquad N\sim\frac{L_\tau}{2\log L_\tau}.
		\label{eq:appendix-disk-effective-order}
	\end{equation}
\end{lemma}

\begin{proof}
	\emph{Bessel estimates.}
	Set $x=kR$. The power series for $J_n$ gives, uniformly for $y$ in a
	fixed compact subset of $(0,\infty)$,
	\[
	J_n(y)=\frac{(y/2)^n}{n!}\bigl(1+\mathcal O(n^{-1})\bigr),
	\]
    thus $\lambda_{n+1}/\lambda_n=\mathcal O(n^{-2})$.
	Choose $n_0$ so that this ratio is at most $1/2$ for $n\ge n_0$.
	Set
\[
R_{n,N}(t):=
\left|
\frac{J_n(x(1+t/N))}{J_n(x)}
\right|^2,
\qquad
a_{n,N}(t):=
\left(1+\frac{t}{N}\right)^{2n}.
\]
Consequently, for sufficiently large $n$, uniformly for  \(N\ge1\) and $ 0\le t\le T$, 
    \begin{align}
        \label{conne}
        R_{n,N}(t)
	=a_{n,N}(t)\bigl(1+O_T(n^{-1})\bigr).
    \end{align}

    For $n_0\le n\le N$ and $m\ge1$, since $\lambda_n\ge\tau$ 
	whereas $\lambda_{N+1}<\tau$, the eigenvalue ratios give
	\[
	1-p_{n,\tau}\le\frac{\tau}{\lambda_n}
	\le\frac{\lambda_N}{\lambda_n}\le2^{-(N-n)},
	\qquad
	p_{N+m,\tau}\le\frac{\lambda_{N+m}}{\tau}\le2^{-(m-1)}.
	\]
	Thus $\sum_{n=0}^N(1-p_{n,\tau})=\mathcal O(1)$ and
	$\sum_{n>N}p_{n,\tau}=\mathcal O(1)$. Since $p_{-n,\tau}=p_{n,\tau}$,
    \[2\pi R Q_\tau^{\rm disk}(R)=\sum_{n\in\mathbb{Z}}p_{n,\tau}=2N+\mathcal O(1),\]
which proves \eqref{eq:appendix-disk-boundary-value}. By symmetry,
\begin{align*}
    2\pi R Q_\tau^{\rm disk}(R(1+t/N))
=&p_{0,\tau}R_{0,N}(t)
+2\sum_{n=1}^{N}p_{n,\tau}R_{n,N}(t)
+2\sum_{n>N}p_{n,\tau}R_{n,N}(t)\\
=&p_{0,\tau}R_{0,N}(t)
-2\sum_{n=1}^{N}(1-p_{n,\tau})R_{n,N}(t)+2\sum_{n=1}^{N}
(R_{n,N}(t)-a_{n,N}(t))\\&+2\sum_{n=1}^{N}
a_{n,N}(t)
+2\sum_{n>N}p_{n,\tau}R_{n,N}(t)
\end{align*}
Since
\(a_{n,N}(t)\le e^{2T}\), it follows from \eqref{conne} that
\(R_{n,N}(t)\le C_T\) for \(n\le N\). Hence
\[
\sum_{n=1}^{N}
(1-p_{n,\tau})R_{n,N}(t)
\le
C_T\sum_{n=1}^{N}(1-p_{n,\tau})
=
O_T(1),
\]
whereas
\[
\sum_{n=n_0}^{N}
|R_{n,N}(t)-a_{n,N}(t)|
\le
C_T\sum_{n=n_0}^{N}\frac1n
=
O_T(\log N).
\]
The finitely many modes \(0\le n<n_0\) contribute \(O_T(1)\).
	For sufficiently large \(N\), the tail is bounded uniformly
for \(0\le t\le T\) by
	\[
	\sum_{n>N}p_{n,\tau}
	R_{n,N}(t)
	\le C_T\sum_{m=1}^{\infty}2^{-(m-1)}e^{2Tm/N}
	=O_T(1).
	\]
	Combining these estimates gives
\[
2\pi R Q_\tau^{\rm disk}(R(1+t/N))
=
1+2\sum_{n=1}^{N}
\left(1+\frac{t}{N}\right)^{2n}
+O_T(\log N),
\]
uniformly for \(0\le t\le T\), which proves \eqref{eq:appendix-disk-regularized-sum}.

\emph{Effective angular order.} The large-order Bessel expansion and
Stirling's formula yield
\[
b_n:=\log\frac{\|A\|^2}{\lambda_n}
=
2n\log\frac{2n}{\mathrm e kR}+\mathcal O(\log n).
\]
The inequalities \(\lambda_N\ge\tau>\lambda_{N+1}\) give
\(b_N\le L_\tau<b_{N+1}\), while
\[
\begin{aligned}
b_{N+1}-b_N
&=
2\log\frac{2(N+1)}{\mathrm e kR}
+
2N\log\left(1+\frac1N\right)
+\mathcal O(\log N)
=\mathcal O(\log N),
\end{aligned}
\]
since \(0\le N\log(1+1/N)\le1\).
Consequently,
\[
\begin{aligned}
L_\tau
&=b_N+\mathcal O(\log N)
=2N\log\frac{2N}{\mathrm e kR}+\mathcal O(\log N)
=2N\log N\bigl(1+o(1)\bigr).
\end{aligned}
\]
Taking logarithms gives
\(
\log L_\tau
=
\log N+\log\log N+\log 2+o(1),
\)
so \(\log L_\tau/\log N\to1\). Therefore,
\[
\frac{N}{L_\tau/(2\log L_\tau)}
=
\frac{2N\log N}{L_\tau}
\frac{\log L_\tau}{\log N}
\longrightarrow1,
\]
proving \(N_\tau\sim L_\tau/(2\log L_\tau)\).
\end{proof}

	\section{Sharp Analytic-Boundary Asymptotics}
\label{app:analytic-radiation-estimates}

This appendix proves Theorem~\ref{thm:analytic-boundary-resolution}
for the full regularized statistic.
We work with the exact normalized reference operator
\eqref{eq:normalized-ideal-radiation}, at fixed $k>0$ satisfying the
nonresonance condition, on a fixed bounded simply connected domain with regular
analytic Jordan boundary. 

Put 
$g=\log|\Phi|$ and $h_\Gamma=|\Phi'|^{-1}$.
Retain $J_\tau$, $\rho_\tau=\tau/\|A\|^2$,
$L_\tau=\log(1/\rho_\tau)$, $G$, and $t_{1/2}$ from the body.
In this appendix
\[
 (\cH f)(z)=\int_{\mathbb S^1}e^{ik\hat x\cdot z}f(\hat x)\,ds(\hat x),
 \qquad A^*f=(\cH f)|_\Gamma,\qquad
 \cH f(z)=\langle f,\phi_z\rangle_H.
\]
Constants may depend on the fixed geometry, $k$, compact sets,
collars, and fixed comparison parameters, but never on $M,j,J_\tau$
or $\tau$.
The proof is organized as follows. Section \ref{sec:ana-spectrum} establishes the effective spectral count and Fourier truncation estimates, and introduces the finite-dimensional Fourier--Bessel spaces. Section \ref{sxeg} develops the Szeg\H{o} polynomial asymptotics and Vekua estimates needed to construct a Helmholtz basis. Section \ref{gram} proves uniform stability of this basis through its Gram operator. Section \ref{aker} derives the scaled boundary-layer profile of the associated reproducing kernel. Finally, Section \ref{sec:ana-tikhonov} uses a variational comparison with slightly smaller and larger Fourier cutoffs to transfer this profile to the full Tikhonov statistic, thus yielding the half-height width of the boundary transition in Section \ref{theoan}.

\subsection{Singular values and Fourier truncation}
	\label{sec:ana-spectrum}
	
	Translate an interior point of $D$ to the origin before defining the angular
projection. This multiplies $A^*$ by the
	unitary multiplication operator $f\mapsto e^{ik\hat x\cdot z_0}f$ on
	$L^2(\mathbb S^1)$. It therefore changes neither the singular values
	nor $Q_\tau$ (in translated coordinates). From now on
	$0\in D$. Choose radii $r_1<r_2$ with
	\[
	\overline{B_{r_1}(0)}\subset D,\qquad \overline D\subset B_{r_2}(0),
	\qquad J_n(kr_1)J_n(kr_2)\ne0\quad(n\ge0);
	\]
	this is possible because the excluded radii form a countable set. For
	$\varrho>0$ let $A_\varrho$ be the operator \eqref{eq:normalized-ideal-radiation} with
	$\Gamma$ replaced by the circle $|y|=\varrho$. Every Herglotz wave
	function has the Fourier--Bessel expansion
	\begin{equation}
		\HH f(|z|e^{i\theta})
		=\sum_{n\in\mathbb Z}d_nJ_{|n|}(k|z|)e^{in\theta},
		\qquad
		\sum_{n\in\mathbb Z}|d_n|^2=2\pi\norm f_{L^2(\mathbb S^1)}^2,
		\label{eq:fb-expansion}
	\end{equation}
	where $d_n$ is $2\pi$ times the $n$th Fourier coefficient of $f$, up to
	a unimodular factor.
	
	\begin{lemma}[Superexponential decay of singular values]
		\label{lem:spectrum}
		The operator $A$ has infinitely many positive singular values, and
		as $j\to\infty$,
		\begin{equation}
			\log\sigma_j=-\frac j2\log j+\mathcal O(j).
			\label{eq:spectrum}
		\end{equation}
		Consequently
		\begin{equation}
			L_\tau=J_\tau\log J_\tau+\mathcal O(J_\tau),\qquad
			J_\tau\sim L_\tau/\log L_\tau\qquad(\tau\downarrow0).
			\label{eq:count-law}
		\end{equation}
	\end{lemma}
	
	\begin{proof}
		Let $Y_1:L^2(\Gamma)\to L^2(\{|y|=r_1\})$ be the solution operator of
		the interior Dirichlet problem in $D$ followed by restriction to the
		inner circle. Let $Y_2:L^2(\{|y|=r_2\})\to L^2(\Gamma)$ be the
		analogous operator for $B_{r_2}(0)$ followed by restriction to
		$\Gamma$. Both are bounded, since each target curve lies strictly
		inside the corresponding domain and the Poisson kernel is smooth on the
		product of the two curves. Every Herglotz wave function solves both
		Dirichlet problems, so
		\[
		A_{r_1}^*=Y_1A^*,\qquad A^*=Y_2A_{r_2}^*,
		\]
		and the inequality $\sigma_j(ST)\le\norm S\sigma_j(T)$ gives
		\[
		\norm{Y_1}^{-1}\sigma_j(A_{r_1})\le\sigma_j(A)\le\norm{Y_2}\sigma_j(A_{r_2}).
		\]
		By \eqref{eq:fb-expansion}, $A_\varrho^*$ is diagonal in the Fourier
		basis, and its singular values are $2\pi\sqrt\varrho\,|J_n(k\varrho)|$,
		$n\in\mathbb Z$. Our choice of $r_1$ makes all these values positive
		for $\varrho=r_1$, so the lower bound above proves that $A$ has
		infinitely many positive singular values. The power series
		\cite[Eq.~10.2.2]{DLMF} and Stirling's formula give
		\[
		J_n(k\varrho)=\frac{(k\varrho/2)^n}{n!}(1+\mathcal O(n^{-1})),\qquad
		\log\bigl(2\pi\sqrt\varrho\,|J_n(k\varrho)|\bigr)=-n\log n+\mathcal O(n).
		\]
		These values are eventually strictly decreasing in $n$. Each order
		$n\ge1$ occurs twice ($\pm n$), so the $j$th largest corresponds to
		$n=j/2+\mathcal O(1)$. This yields \eqref{eq:spectrum}.
		
		Put $\ell_j=-\log(\sigma_j^2/\norm A^2)$, a nondecreasing sequence with
		$\ell_j=j\log j+\mathcal O(j)$. Every fixed $\sigma_j$ eventually exceeds
		$\sqrt\tau$, so $J_\tau\to\infty$. By definition of $J_\tau$,
		$\ell_{J_\tau}\le L_\tau<\ell_{J_\tau+1}$, which gives the first
		relation in \eqref{eq:count-law}. Taking logarithms gives
		$\log L_\tau=\log J_\tau+\mathcal O(\log\log J_\tau)$, hence the second.
	\end{proof}
	
	Let $E_M$ be the orthogonal projection in $L^2(\mathbb S^1)$ onto the
	Fourier modes $e^{in\theta}$, $|n|\le M$. Denote
	\[
	\BH_M=\{\HH f:f\in\operatorname{Ran}E_M\}
	=\operatorname{span}\{J_{|n|}(k|z|)e^{in\arg z}:|n|\le M\},
	\]
	the space of generalized harmonic polynomials of degree $M$, regarded
	as entire solutions of $(\Delta+k^2)u=0$. Since $k^2$ is not a
	Dirichlet eigenvalue, $\norm\cdot_{L^2(\Gamma)}$ is a norm on $\BH_M$.
	
	\begin{lemma}[Fourier truncation estimates]
		\label{lem:fourier-estimates}
		For every fixed compact set $\mathcal K\subset\mathbb R^2$ and $M\ge2$, put
		$\delta_M=Ce^{-M\log M+CM}$ and $\Theta_M=Ce^{M\log M+CM}$. Then
		\begin{align}
			\norm{\HH(\Id-E_M)f}_{C(\mathcal K)}
			+\norm{\HH(\Id-E_M)f}_{L^2(\Gamma)}
			&\le \delta_M\norm f,
			\label{eq:angular-tail}\\
			\norm f&\le \Theta_M\norm{\HH f}_{L^2(\Gamma)}
			\quad(f\in\operatorname{Ran}E_M).
			\label{eq:density-bound}
		\end{align}
	\end{lemma}
	
	\begin{proof}
		The Bessel series gives
		$\sup_{z\in\mathcal K\cup\Gamma}|J_m(k|z|)|\le C R^m/m!$ for a fixed $R$.
		Since $|d_n|\le\sqrt{2\pi}\norm f$ in \eqref{eq:fb-expansion},
		summing the factorial tail proves \eqref{eq:angular-tail}.
		
		For $f\in\operatorname{Ran}E_M$, restriction to the inner circle and
		Parseval's identity give
		\[
		\norm{\HH f}_{L^2(\Gamma)}^2
		\ge c\sum_{|n|\le M}|d_n|^2|J_{|n|}(kr_1)|^2
		\ge c\min_{0\le n\le M}|J_n(kr_1)|^2\norm f^2.
		\]
		The nonzero initial Bessel values and their factorial asymptotics
		imply $\min_{0\le n\le M}|J_n(kr_1)|\ge c e^{-M\log M-CM}$.
		This proves \eqref{eq:density-bound}.
	\end{proof}

	\subsection{{Szeg\H{o}} polynomials and the Vekua operator}
	\label{sxeg}
	Let $p_m$ ($m\ge0$) be the Szeg\H{o} polynomials of $\Gamma$, i.e. the
	polynomials orthonormal with respect to $ds$ on $\Gamma$, of degree
	$m$ and with positive leading coefficient; thus $p_0=|\Gamma|^{-1/2}$.
	Since $\Gamma$ is analytic, Szeg\H{o}'s strong asymptotics (normalized
	for $ds$ rather than $ds/2\pi$) read
	\begin{equation}
		p_m(z)=\sqrt{\frac{\Phi'(z)}{2\pi}}\;\Phi(z)^m(1+o(1)),
		\label{eq:szego}
	\end{equation}
	uniformly on compact subsets of an annular neighbourhood of $\Gamma$
	(see \cite{MinaDiaz2009}). The branch of the square root is positive at infinity.
	Fix $\kappa>0$ so small that the level curve
	$\Gamma_\kappa=\{g=-\kappa\}$ lies in this neighbourhood and encloses
	$0$, and let $D_\kappa$ be its interior. By \eqref{eq:szego} and the
	maximum modulus principle,
	\begin{equation}
		\sup_{\overline{D_\kappa}}|p_m|\le Ce^{-\kappa m},
		\qquad |p_m(z)|\le C e^{mg(z)}\quad\text{for } -\kappa\le g(z)\le g_0,
		\label{eq:polynomial-bounds}
	\end{equation}
	for a fixed $g_0>0$.
	
	Let $E^2(D)$ denote the Smirnov space, i.e. the closure of the
	polynomials in $L^2(\Gamma)$. Its elements are identified with their
	holomorphic extensions to $D$. Define the Vekua operator $\VV_k^+$ on
	polynomials by
	\begin{equation}
		\VV_k^+(z^m)
		=m!\,(2/k)^m J_m(k|z|)e^{im\arg z},\qquad m\ge0,
		\label{eq:monomial-transform}
	\end{equation}
	extended linearly (the value at $z=0$ is taken by continuity), and set
	$\VV_k^-q=\overline{\VV_k^+(\overline q)}$ for antiholomorphic $q$. In particular $\VV_k^+(1) =\VV_k^-(1)= J_0(k|z|)$.
	On polynomials this is the holomorphic part of Vekua's operator for
	the Helmholtz equation centred at $0$ \cite[Sec.~4]{MoiolaHiptmairPerugia2011}.
	
	\begin{lemma}[Properties of the Vekua operator]
		\label{lem:transform-properties}
		\leavevmode
		\begin{enumerate}
			\item[(i)] For $f\in E^2(D)$ and $z\in D$,
			\begin{equation}
				\VV_k^+f(z)-f(z)
				=-\lambda\overline z\int_0^z f(\zeta)
				\Lambda(\lambda\overline z(z-\zeta))\,d\zeta,
				\qquad
				\lambda=\frac{k^2}{4},\quad
				\Lambda(u)=\sum_{\ell\ge0}\frac{(-u)^\ell}{\ell!(\ell+1)!},
				\label{eq:transform-integral}
			\end{equation}
			with the integral independent of the path in $D$. Both $\VV_k^+f$ and
			$\VV_k^-q$, for $q\in\overline{E^2(D)}$, solve $(\Delta+k^2)u=0$. For
			polynomials the formula holds in all of $\mathbb C$, and
			\[
			\BH_M=\{\VV_k^+F+\VV_k^-q:\ F,\overline q
			\text{ polynomials of degree}\le M\}.
			\]
			
			\item[(ii)] For $f\in E^2(D)$,
			\begin{equation}
				\norm{(\VV_k^+-\Id)f}_{H^1(\Gamma)}
				+\norm{(\VV_k^+-\Id)f}_{H^1(D)}
				\le C\norm f_{L^2(\Gamma)}.
				\label{eq:transform-regularity}
			\end{equation}
			The same bound holds for $\VV_k^-$ on $\overline{E^2(D)}$.
			
			\item[(iii)] If $F\in E^2(D)$, $q\in\overline{E^2(D)}$ and
			$\VV_k^+F+\VV_k^-q=0$ in $D$, then $F\equiv c$ and $q\equiv-c$ for a
			constant $c$.
			
			\item[(iv)] For $m\ge1$ and $0\le g(z)\le g_0$,
			\begin{equation}
				|\VV_k^+p_m(z)-p_m(z)|\le \frac{C}{m}e^{mg(z)},
				\label{eq:transform-error}
			\end{equation}
			and likewise for $\VV_k^-\overline{p_m}-\overline{p_m}$.
		\end{enumerate}
	\end{lemma}
	
	\begin{proof}
		(i) For $f(\zeta)=\zeta^m$, expanding $\Lambda$ and using
		$\int_0^z\zeta^m(z-\zeta)^\ell\,d\zeta=\frac{m!\,\ell!}{(m+\ell+1)!}z^{m+\ell+1}$
		gives
		\[
		z^m-\lambda\overline z\int_0^z\zeta^m
		\Lambda(\lambda\overline z(z-\zeta))\,d\zeta
		=z^m\sum_{n\ge0}\frac{m!\,(-\lambda|z|^2)^n}{n!\,(m+n)!},
		\]
		which equals \eqref{eq:monomial-transform} by \cite[Eq.~10.2.2]{DLMF}.
		By linearity the formula holds for polynomials. For fixed $z$ the
		integrand is holomorphic in $\zeta$, so simple connectivity gives path
		independence. If polynomials converge in $L^2(\Gamma)$, Cauchy's
		formula gives locally uniform convergence in $D$ of the polynomials and
		their derivatives, hence of the integrals; this extends
		\eqref{eq:transform-integral} to $E^2(D)$. The Helmholtz equation holds
		for monomial images by Bessel's equation and passes to the limit; the
		statement for $\VV_k^-$ follows by conjugation. Finally,
		\eqref{eq:monomial-transform} and its conjugate produce all
		$J_m(k|z|)e^{\pm im\arg z}$, $0\le m\le M$, with
		$\VV_k^+1=\VV_k^-1=J_0(k|z|)$; these span $\BH_M$.
		
		(ii) It suffices to prove the estimate for polynomials. Put
		$N=\norm f_{L^2(\Gamma)}$. By the Riemann mapping theorem, there exists a conformal map $\chi:\mathbb D\to D$ from the unit disk to $D$ with $\chi(0)=0$. Since $\Gamma$ is analytic, 
		$|\chi'|$ is bounded above and below on $\overline{\mathbb D}$.
		For $F=f\circ\chi=\sum_{n\ge0}a_nw^n$, Parseval's identity gives
		$\sum_n|a_n|^2\le CN^2$. Cauchy--Schwarz therefore gives
		\[
		|F(re^{i\theta})|
		\le\Bigl(\sum_n|a_n|^2\Bigr)^{1/2}
		\Bigl(\sum_n r^{2n}\Bigr)^{1/2}
		\le\frac{CN}{\sqrt{1-r^2}},\qquad 0\le r<1.
		\]
		For $z=\chi(re^{i\theta})$, integrate along the path
		$\gamma_z=\{\chi(se^{i\theta}):0\le s\le r\}$. The preceding bound
		and boundedness of $\chi'$ imply
		\begin{equation}
			\int_{\gamma_z}|f(\zeta)|\,|d\zeta|
			\le CN\int_0^r\frac{ds}{\sqrt{1-s^2}}
			\le CN,\qquad 0\le r\le1.
			\label{eq:path-bound}
		\end{equation}
		
		Write \eqref{eq:transform-integral} as
		\[
		Wf(z):=(\VV_k^+-\Id)f(z)=\int_0^z f(\zeta)K(z,\zeta)\,d\zeta,
		\qquad
		K(z,\zeta)=-\lambda\overline z\,
		\Lambda(\lambda\overline z(z-\zeta)).
		\]
		The kernel $K$ and its first derivatives in $z=x+iy$ are bounded on
		$\overline D\times\overline D$. Differentiating the path-independent
		integral gives, for $\alpha=x,y$,
		\[
		\partial_\alpha Wf(z)
		=f(z)K(z,z)\,\partial_\alpha z
		+\int_0^z f(\zeta)\partial_\alpha K(z,\zeta)\,d\zeta,
		\qquad \partial_xz=1,\quad\partial_yz=i.
		\]
		Thus \eqref{eq:path-bound} yields
		\[
		|Wf(z)|\le CN,\qquad |\nabla Wf(z)|\le C\bigl(|f(z)|+N\bigr).
		\]
		On $\Gamma$, $|\partial_sWf|\le|\nabla Wf|$, so squaring and
		integrating proves the $H^1(\Gamma)$ bound. For the interior bound,
		Parseval also gives
		\[
		\norm f_{L^2(D)}^2\le C\int_{\mathbb D}|F(w)|^2\,dA(w)
		=C\pi\sum_{n\ge0}\frac{|a_n|^2}{n+1}\le CN^2.
		\]
		Integrating the same pointwise estimates over $D$ proves the
		$H^1(D)$ bound. Density of polynomials now gives
		\eqref{eq:transform-regularity} for every $f\in E^2(D)$.
		By the locally uniform convergence established in (i), this extension
		agrees with \eqref{eq:transform-integral} in $D$. Conjugation gives
		the assertion for $\VV_k^-$.
		
		(iii) Write $q(z)=B(\overline z)$ with $B$ holomorphic in
		$\{\overline z:z\in D\}$, and expand $F=\sum_m F_mz^m$,
		$B=\sum_m q_mw^m$ near $0$. On a sufficiently small closed disk about
		$0$, these Taylor series converge uniformly. Applying
		\eqref{eq:transform-integral} to their partial sums and using
		\eqref{eq:monomial-transform} gives, for small $|z|$,
		\[
		(\VV_k^+F+\VV_k^-q)(|z|e^{i\theta})
		=(F_0+q_0)J_0(k|z|)
		+\sum_{m\ge1}m!\,(2/k)^mJ_m(k|z|)
		\bigl(F_me^{im\theta}+q_me^{-im\theta}\bigr).
		\]
		If this vanishes, uniqueness of Fourier coefficients gives
		$F_m=q_m=0$ for $m\ge1$ and $F_0+q_0=0$, since no $J_m(k\cdot)$
		vanishes identically. Analytic continuation gives the claim on $D$.
		
		(iv) Write $z=\Phi^{-1}(e^{g+i\theta})$. In
		\eqref{eq:transform-integral}, integrate first from $0$ to
		$\Phi^{-1}(e^{-\kappa+i\theta})$ inside $\overline{D_\kappa}$ (paths
		of bounded length), then along
		$\zeta(s)=\Phi^{-1}(e^{s+i\theta})$, $-\kappa\le s\le g$ (bounded
		speed). The kernel is bounded, so by \eqref{eq:polynomial-bounds}
		\[
		|\VV_k^+p_m(z)-p_m(z)|
		\le Ce^{-\kappa m}+C\int_{-\kappa}^g e^{ms}\,ds
		\le \frac{C}{m}e^{mg},
		\]
		using $\sup_m me^{-\kappa m}<\infty$ and $g\ge0$. Conjugation gives the
		second statement.
	\end{proof}
	
	\subsection{The Gram operator}
	\label{gram}
	For $m\ge1$ put
	\begin{equation}
		\psi_0=\VV_k^+p_0,\qquad
		\psi_m^+=\VV_k^+p_m,\qquad
		\psi_m^-=\VV_k^-\overline{p_m}=\overline{\psi_m^+}.
		\label{eq:helmholtz-basis}
	\end{equation}
	By Lemma~\ref{lem:transform-properties}(i), the $2M+1$ functions
	$\psi_0,\psi_1^\pm,\dots,\psi_M^\pm$ form a basis of $\BH_M$. Let
	\[
	\mathcal L=\mathbb C\oplus\ell^2(\mathbb N)\oplus\ell^2(\mathbb N),
	\qquad
	\|a\|_{\mathcal L}^2=|a_0|^2+\sum_{m\ge1}(|a_m^+|^2+|a_m^-|^2),
	\]
	with coordinates ordered $0,1+,1-,2+,2-,\ldots$. For finitely supported
	$a$ define the synthesis operators
	\begin{align}
		R_0a&=a_0p_0+\sum_{m\ge1}\bigl(a_m^+p_m+a_m^-\overline{p_m}\bigr),
		\label{eq:R0}\\
		T_0a&=a_0\psi_0+\sum_{m\ge1}\bigl(a_m^+\psi_m^++a_m^-\psi_m^-\bigr),
		\label{eq:T0}
	\end{align}
	with values (restricted to $\Gamma$) in $L^2(\Gamma)$. They give the
	boundary traces of the harmonic and the Helmholtz functions with
	coefficients $a$, respectively.
	
	\begin{lemma}[Uniform lower bound for the Gram operator]
		\label{lem:gram}
		$R_0$ and $T_0$ extend to bounded operators $\mathcal L\to L^2(\Gamma)$,
		and the Gram operator $\GG=T_0^*T_0$ satisfies
		\begin{equation}
			\GG=\Id+\CC,\qquad \CC\text{ compact},\qquad \GG\ge c_G\Id
			\label{eq:gram}
		\end{equation}
		for some fixed $c_G>0$. Hence the finite Gram matrices $\GG_M=\mathsf P_M\GG\mathsf
		P_M|_{\operatorname{Ran}\mathsf P_M}$ satisfy $\norm{\GG_M^{-1}}\le c_G^{-1}$,
		where $\mathsf P_M$ is the coordinate projection onto
		$0,1\pm,\dots,M\pm$.
	\end{lemma}
	
	\begin{proof}
		\emph{Step 1: $R_0^*R_0-\Id$ is Hilbert--Schmidt.} 
		We compute the entries of $R_0^*R_0$, that is, the inner products of
		the functions $p_0,p_m,\overline{p_m}$ ($m\ge1$). Within each family $\{p_m\}_{m\ge0}$ and $\{\overline{p_m}\}_{m\ge1}$
		the entries are Kronecker deltas, by orthonormality. Since $p_0$ is
		real, $\langle\overline{p_m},p_0\rangle_{L^2(\Gamma)}=\overline{\langle p_m,p_0\rangle_{L^2(\Gamma)}}=0$
		for $m\ge1$. The only remaining entries are
		$\langle p_n,\overline{p_m}\rangle_{L^2(\Gamma)}=\int_\Gamma p_mp_n\,ds$ for
		$m,n\ge1$, together with their complex conjugates.
		
		To estimate these, we use that $\Gamma$ is analytic. 
        Let \(z=z(s)\), \(s\in\mathbb R\), be the arclength parametrization of
\(\Gamma\subset\mathbb R^2\cong\mathbb C\), traversed counterclockwise.
Then \(z(s)\) is \(|\Gamma|\)-periodic and real-analytic with \(|z'(s)|=1\),
so it has the Fourier expansion
\[
  z(s)=\sum_{n\in\mathbb Z}\hat z_n\,e^{2\pi i n s/|\Gamma|}
\]
whose coefficients decay exponentially:
\(|\hat z_n|\le C e^{-2\pi\delta|n|/|\Gamma|}\) for some \(C,\delta>0\).
For \(|\operatorname{Im}s|\le\delta'<\delta\), each term is bounded by
\(C e^{-2\pi(\delta-\delta')|n|/|\Gamma|}\). So the series converges
uniformly on every such closed substrip, and it extends \(z\) holomorphically
to the strip \(\{|\operatorname{Im}s|<\delta\}\), still \(|\Gamma|\)-periodic.
		For small $\delta$ the extension is injective modulo the period with $z'\neq0$, and maps
		the strip onto a ring domain $U\supset\Gamma$. With $s(\cdot)$ any
		local holomorphic inverse, define
		\[
		\omega(z)=\frac{1}{z'(s(z))},\qquad z\in U .
		\]
		Since \(z'\) is \(|\Gamma|\)-periodic, the definition is independent of the choice of local inverse. This $\omega$ is holomorphic in $U$, and on $\Gamma$ we have
		$dz=z'(s)\,ds$. Hence
		$ds=\omega(z)\,dz$ on $\Gamma$.
		Shrinking $\kappa$ if necessary, the closed ring between $\Gamma_\kappa$
		and $\Gamma$ lies in $U$. Since $p_mp_n\omega$ is holomorphic there,
		Cauchy's theorem and \eqref{eq:polynomial-bounds} give
		\[
		\Bigl|\int_\Gamma p_mp_n\,ds\Bigr|
		=\Bigl|\int_{\Gamma_\kappa}p_mp_n\,\omega\,dz\Bigr|
		\le|\Gamma_\kappa|\,\max_{\Gamma_\kappa}|\omega|\,
		\max_{\Gamma_\kappa}|p_m|\,\max_{\Gamma_\kappa}|p_n|
		\le C_1e^{-\kappa(m+n)},
		\]
		with $C_1$ independent of $m,n$.
		
		Hence, on finitely supported vectors, $R_0^*R_0=\Id+\mathsf K$, where
		$\mathsf K$ has nonzero entries only in the cross block between the
		$+$ and $-$ coordinates, and
		\[
		\norm{\mathsf K}_{\mathrm{HS}}^2
		\le 2C_1^2\sum_{m,n\ge1}e^{-2\kappa(m+n)}
		=\frac{2C_1^2}{(e^{2\kappa}-1)^2}<\infty .
		\]
		Therefore $\norm{R_0a}^2=\langle(\Id+\mathsf K)a,a\rangle_{\mathcal L}
		\le(1+\norm{\mathsf K})\norm a^2$. So $R_0$ extends to a bounded
		operator on $\mathcal L$, and $R_0^*R_0=\Id+\mathsf K$ with $\mathsf K$
		compact.
		
		\emph{Step 2: $T_0-R_0$ is compact.} We have $(T_0-R_0)a
		=(\VV_k^+-\Id)F+(\VV_k^--\Id)q$, where
		$F=a_0p_0+\sum a_m^+p_m$ and $q=\sum a_m^-\overline{p_m}$, with
		$\norm F_{L^2(\Gamma)}^2+\norm q_{L^2(\Gamma)}^2=\norm a^2$. By
		\eqref{eq:transform-regularity}, $T_0-R_0:\mathcal L\to H^1(\Gamma)$ is
		bounded, hence compact into $L^2(\Gamma)$. Therefore $T_0$ is bounded
		and
		\[
		T_0^*T_0=R_0^*R_0+R_0^*(T_0-R_0)+(T_0-R_0)^*T_0=\Id+\CC,
		\]
		with $\CC$ compact.
		
		\emph{Step 3: $\ker T_0=\{0\}$.} 
		Let $T_0a=0$, with $F,q$ as in
		Step~2. Here $F$ is holomorphic and $q$ is antiholomorphic in $D$. Put
		\[
		u=\VV_k^+F+\VV_k^-q .
		\]
		Then $u$ solves $\Delta u+k^2u=0$ in $D$, and its boundary
		trace is $T_0a=0$. Write $u=h+w$, where $h=F+q$ is harmonic and
		$w=(\VV_k^+-\Id)F+(\VV_k^--\Id)q$. By Step 2, $w\in H^1(D)$ and
		$w|_\Gamma\in H^1(\Gamma)$. Hence $h|_\Gamma=-w|_\Gamma\in H^1(\Gamma)$,
		so $h\in H^{3/2}(D)\subset H^1(D)$  by standard regularity
of the harmonic Poisson extension. Thus $u\in H^1_0(D)$ is a weak
		solution of the Helmholtz equation. Since $k^2$ is not a Dirichlet
		eigenvalue, $u\equiv0$ in $D$.
		Lemma~\ref{lem:transform-properties}(iii)
		then gives $F\equiv c$, $q\equiv-c$. But $q$ is orthogonal to
		constants by Step~1 and the definition of $q$, so $c=0$.
		Orthonormality then gives $a=0$.
		
		\emph{Conclusion.} $\GG$ is self-adjoint, nonnegative, of the form
		$\Id+$compact, and injective. By the Fredholm alternative it is
		invertible, hence $\GG\ge c_G\Id$ with $c_G>0$. Compressing to
		$\operatorname{Ran}\mathsf P_M$ gives $\GG_M\ge c_G\Id$.
	\end{proof}
	
	\subsection{Asymptotics of the kernel near the boundary}
	\label{aker}
	Let
	\[
	e_M(z)=\bigl(\overline{\psi_0(z)},
	\overline{\psi_1^+(z)},\overline{\psi_1^-(z)},\ldots,
	\overline{\psi_M^+(z)},\overline{\psi_M^-(z)}\bigr)^{\mathsf T},
	\]
	also regarded as an element of $\mathcal L$ (extended by zero).
	The diagonal reproducing kernel of \(\mathcal B_M\), equipped with the \(L^2(\Gamma)\) norm, is
	\begin{equation}
		K_M(z)=\sup\{|u(z)|^2:u\in\BH_M,\ \norm u_{L^2(\Gamma)}=1\}.
		\label{eq:kernel}
	\end{equation}
	\begin{lemma}[Kernel in terms of the Gram matrix]
		\label{lem:gram-kernel}
		For $M\ge1$ and $z\in\mathbb R^2$,
		\[
		K_M(z)=e_M(z)^*\GG_M^{-1}e_M(z).
		\]
	\end{lemma}
	
	\begin{proof}
		Each $u\in\BH_M$ has a unique coefficient vector $a$ in the basis
		\eqref{eq:helmholtz-basis}, with
		$\norm u_{L^2(\Gamma)}^2=a^*\GG_Ma$ and $u(z)=e_M(z)^*a$. Put
		$y=\GG_M^{-1}e_M(z)$. Then $u(z)=y^*\GG_Ma$, and the Cauchy--Schwarz
		inequality for the inner product $(a,b)\mapsto b^*\GG_Ma$ gives
		\[
		|u(z)|^2\le(y^*\GG_My)(a^*\GG_Ma)
		=\bigl(e_M(z)^*\GG_M^{-1}e_M(z)\bigr)\norm u_{L^2(\Gamma)}^2.
		\]
		If $e_M(z)\ne0$, then $a=y/\sqrt{e_M(z)^*y}$ gives equality with
		$a^*\GG_Ma=1$. If $e_M(z)=0$, both sides vanish.
	\end{proof}
	
	\begin{proposition}[Kernel asymptotics]
		\label{prop:bessel-kernel}
		For every $T>0$, uniformly in $x\in\Gamma$ and $0\le t\le T$,
		\begin{equation}
			\frac1M K_M\bigl(x+h_\Gamma(x)t\nu(x)/M\bigr)
			\longrightarrow \frac{G(t)}{\pi h_\Gamma(x)}\qquad(M\to\infty).
			\label{eq:bessel-kernel-limit}
		\end{equation}
	\end{proposition}
	
	\begin{proof}
		Let $z=x+h_\Gamma(x)t\nu(x)/M$. By \eqref{eq:conformal-normal-scale},
		$0\le g(z)=t/M+\mathcal O(M^{-2})$ uniformly. By \eqref{eq:polynomial-bounds}
		and \eqref{eq:transform-error}, every coordinate of $e_M(z)$ is bounded
		by $C_T$, so $\norm{e_M(z)}\le C_T\sqrt M$.
		
		\emph{Step 1: removing the Gram matrix.} Since $\CC$ is compact,
		$\norm{\CC(\Id-\mathsf P_N)}\to0$ as $N\to\infty$. Hence
		\[
		\norm{\CC e_M(z)}\le\norm{\CC(\Id-\mathsf P_N)}\,C_T\sqrt M
		+\norm\CC\,\norm{\mathsf P_Ne_M(z)}
		\le\norm{\CC(\Id-\mathsf P_N)}\,C_T\sqrt M+C_{T,N},
		\]
		so $\norm{\CC e_M(z)}=o(\sqrt M)$ uniformly. On
		$\operatorname{Ran}\mathsf P_M$ we have
		$\GG_M^{-1}-\Id=-\GG_M^{-1}\mathsf P_M\CC\mathsf P_M$. By
		Lemmas~\ref{lem:gram} and~\ref{lem:gram-kernel},
		\begin{equation}
			K_M(z)=\norm{e_M(z)}^2+O\bigl(\norm{e_M(z)}\,\norm{\CC e_M(z)}\bigr)
			=\norm{e_M(z)}^2+o(M).
			\label{eq:compact-kernel-error}
		\end{equation}
		
		\emph{Step 2: from Helmholtz to harmonic values.} Here
		$|p_m(z)|\le C_T$ and $|\psi_m^+(z)-p_m(z)|\le C_T/m$, so
		$\bigl||\psi_m^+(z)|^2-|p_m(z)|^2\bigr|\le C_T/m$.
		Since $|\psi_m^-|=|\psi_m^+|$, summing gives
		\[
		\norm{e_M(z)}^2=2\sum_{m=1}^{M}|p_m(z)|^2+\mathcal O(\log(M+1)).
		\]
		
		\emph{Step 3: Szeg\H{o} asymptotics.} By \eqref{eq:szego},
		$|p_m(z)|^2=\frac{|\Phi'(z)|}{2\pi}e^{2mg(z)}(1+o(1))$ as
		$m\to\infty$, uniformly. Moreover $|\Phi'(z)|=h_\Gamma(x)^{-1}+o(1)$, and
		$e^{2mg(z)}=e^{2tm/M}(1+\mathcal O(M^{-1}))$ for $m\le M$. Discarding a fixed
		initial segment of indices costs $\mathcal O(1)$, so
		\[
		\frac2M\sum_{m=1}^{M}|p_m(z)|^2
		=\frac1{\pi h_\Gamma(x)}\,\frac1M\sum_{m=1}^{M}e^{2tm/M}+o(1)
		=\frac{G(t)}{\pi h_\Gamma(x)}+o(1),
		\]
		the last step being a Riemann sum, uniform for $t\in[0,T]$. Together
		with Steps~1--2 this proves \eqref{eq:bessel-kernel-limit}.
	\end{proof}
	
	\subsection{Direct comparison with Tikhonov regularization}
	\label{sec:ana-tikhonov}
	
	\begin{lemma}[Variational formula for Tikhonov regularization]
		\label{lem:tikhonov-variational}
		For every $\tau>0$ and $z\in\mathbb R^2$,
		\begin{equation}
			Q_\tau(z)=\sup_{f\ne0}
			\frac{|\HH f(z)|^2}
			{\norm{\HH f}_{L^2(\Gamma)}^2+\tau\norm f^2},
			\qquad f\in L^2(\mathbb S^1).
			\label{eq:tikhonov-variational}
		\end{equation}
	\end{lemma}
	
	\begin{proof}
		Set $B_\tau=AA^*+\tau\Id$ and $f_{\tau,z}=B_\tau^{-1}\phi_z$.
		Since $\HH f(z)=\langle f,\phi_z\rangle_H=\langle B_\tau f,f_{\tau,z}\rangle_H$,
		Cauchy--Schwarz in the inner product
		$(f,g)\mapsto\langle B_\tau f,g\rangle_H$ gives
		\[
		|\HH f(z)|^2
		\le\langle B_\tau f,f\rangle_H\langle B_\tau f_{\tau,z},f_{\tau,z}\rangle_H
		=\bigl(\norm{\HH f}_{L^2(\Gamma)}^2+\tau\norm f^2\bigr)Q_\tau(z).
		\]
		Equality holds for $f=f_{\tau,z}\ne0$, which proves the formula.
	\end{proof}
	
	Each nonzero Fourier order occurs as a pair $\pm n$, so
	$\dim\BH_M=2M+1$; accordingly, the following comparison uses cutoffs
	near $M=J_\tau/2$.
	
	\begin{lemma}[Comparison with Fourier--Bessel kernels]
		\label{lem:direct-comparison}
		Fix $0<\eps<1$, put $J=J_\tau$ and
		$M_\pm=\lfloor(1\pm\eps)J/2\rfloor$. On every fixed compact set
		$\mathcal K\subset\mathbb R^2$, there are
		$\alpha_\tau,\beta_\tau\to0$ such that, for $z\in\mathcal K$ and all
		sufficiently small $\tau$,
		\begin{equation}
			\frac{K_{M_-}(z)}{1+\alpha_\tau}\le Q_\tau(z)
			\le\bigl((1+\beta_\tau)\sqrt{K_{M_+}(z)}+\beta_\tau\bigr)^2.
			\label{eq:direct-comparison}
		\end{equation}
	\end{lemma}
	
	\begin{proof}
		For the lower bound, take $u=\HH f\in\BH_{M_-}$ with
		$f\in\operatorname{Ran}E_{M_-}$ and $\norm u_{L^2(\Gamma)}=1$.
		Equation~\eqref{eq:density-bound} gives $\norm f\le \Theta_{M_-}$, so
		\eqref{eq:tikhonov-variational} implies
		\[
		Q_\tau(z)\ge\frac{|u(z)|^2}{1+\tau \Theta_{M_-}^2}.
		\]
		Taking the supremum over $u$ proves the lower bound with
		$\alpha_\tau=\tau \Theta_{M_-}^2$.
		
		For the upper bound, normalize $f$ so that
		$\norm{\HH f}_{L^2(\Gamma)}^2+\tau\norm f^2=1$.
		Then $\norm f\le\tau^{-1/2}$, and \eqref{eq:angular-tail} gives
		\[
		\norm{\HH(\Id-E_{M_+})f}_{C(\mathcal K)}
		+\norm{\HH(\Id-E_{M_+})f}_{L^2(\Gamma)}
		\le \delta_{M_+}/\sqrt\tau=:\beta_\tau.
		\]
		Thus $\HH E_{M_+}f\in\BH_{M_+}$ has boundary norm at most
		$1+\beta_\tau$, and
		$|\HH f(z)|\le(1+\beta_\tau)\sqrt{K_{M_+}(z)}+\beta_\tau$.
		Taking the supremum in \eqref{eq:tikhonov-variational} proves the
		upper bound.
		
		Finally, \eqref{eq:count-law} gives $-\log\tau=J\log J+\mathcal O(J)$.
		The definitions of $M_\pm,\Theta_M,\delta_M$ therefore yield
		\[
		\alpha_\tau\le e^{-\eps J\log J+C_\eps J},\qquad
		\beta_\tau\le e^{-(\eps/2)J\log J+C_\eps J}.
		\]
		Both tend to zero, uniformly on the fixed compact set.
	\end{proof}
	\subsection{Proof of Theorem~\ref{thm:analytic-boundary-resolution}}
    \label{theoan}
	\begin{proof}
		The count law is \eqref{eq:count-law}. Fix $T>0$, put $J=J_\tau$ and
		$z_\tau(t)=x+2h_\Gamma(x)t\nu(x)/J$. Since $h_\Gamma$ is bounded on $\Gamma$ and
		$J\to\infty$, all these points, with $x\in\Gamma$ and $0\le t\le T$,
		lie in the fixed compact set
		\[
		\mathcal K=\{x+r\nu(x):x\in\Gamma,\ 0\le r\le r_0/2\}
		\]
		for sufficiently small $\tau$. Thus Lemma~\ref{lem:direct-comparison}
		applies uniformly to them. For fixed $0<\eps<1$,
		Proposition~\ref{prop:bessel-kernel}, with the rescaled variable
		$2tM_\pm/J$, gives
		\[
		\frac{K_{M_\pm}(z_\tau(t))}{J}
		\longrightarrow\frac{1\pm\eps}{2\pi h_\Gamma(x)}G((1\pm\eps)t)
		\]
		uniformly in $x\in\Gamma$ and $0\le t\le T$. These kernel ratios are
		bounded, so dividing \eqref{eq:direct-comparison} by $J$ makes all
		its error terms vanish. Let $\tau\downarrow0$ and then
		$\eps\downarrow0$. Uniform continuity of $G$ on bounded intervals
		therefore gives
		\begin{equation}
			\frac{Q_\tau(z_\tau(t))}{J_\tau}
			\longrightarrow\frac{G(t)}{2\pi h_\Gamma(x)}
			\label{eq:tikhonov-limit}
		\end{equation}
		uniformly. At $t=0$ this proves the boundary amplitude in
		\eqref{eq:analytic-boundary-count}. Dividing by this positive boundary limit
		proves \eqref{eq:analytic-boundary-profile}.

	For the width, fix $0<\varepsilon<t_{1/2}$. Since $G$ is strictly increasing,
		the normalized profile is below $2$ throughout
		$0\le t\le t_{1/2}-\varepsilon$ and above $2$ at $t=t_{1/2}+\varepsilon$, for all
		small enough $\tau$, uniformly in $x$. These points lie in the fixed
		normal neighbourhood, so \eqref{eq:analytic-halfheight-definition} gives
		\[
		2(t_{1/2}-\varepsilon)\frac{h_\Gamma(x)}{J_\tau}
		\le w_{1/2}(x;\tau)
		\le2(t_{1/2}+\varepsilon)\frac{h_\Gamma(x)}{J_\tau}.
		\]
		Letting $\varepsilon\downarrow0$ and substituting
		$J_\tau\sim L_\tau/\log L_\tau$ proves \eqref{eq:analytic-halfheight-noise-law}.
	\end{proof}
	
	\section{A diffusion-based boundary correlation model}
\label{app:heat-kernel-model}

This appendix introduces a prescribed boundary-source correlation
model based on diffusion along the boundary \cite{NiuEtAl2019Intrinsic}, which provides a concrete realization of the source-side modal diversity considered in Section~\ref{sec:abstract-global-correlation}.
Let \(\Gamma\) be a smooth, connected, closed planar curve of
length \(L\), parameterized by arclength \(s\in[0,L)\).
The boundary Laplacian is
\(
\Delta_\Gamma=\partial_s^2
\)
with periodic boundary conditions.
The heat equation \(\partial_t f=\Delta_\Gamma f\) averages
neighboring values along the boundary and suppresses rapid
spatial variations. 

This motivates using the smoothing operator
\[
R_\ell=e^{\ell^2\Delta_\Gamma},\qquad \ell>0,
\]
to model boundary-source correlation, which can be interpreted as smoothing spatial white noise
along the boundary. Neighboring points then share random inputs
over an arclength scale of order \(\ell\), producing spatial
correlation.

The arclength Fourier modes satisfy
\[
e_n(s)=L^{-1/2}e^{2\pi ins/L},
\qquad
-\Delta_\Gamma e_n=\omega_n^2e_n,
\qquad
\omega_n=\frac{2\pi n}{L},
\quad n\in\mathbb Z.
\]
Consequently,
\[
R_\ell e_n=e^{-\ell^2\omega_n^2}e_n.
\]
These formulas hold for any smooth closed connected planar boundary since the diffusion is defined intrinsically in arclength. To vary correlation without changing source power, fix \(P=v_0L\)
and define
\begin{align}\label{zl}
    C_\ell=\frac{P}{Z_\ell}R_\ell,
\qquad
Z_\ell=\operatorname{tr}R_\ell
=\sum_{n\in\mathbb Z}e^{-\ell^2\omega_n^2}.
\end{align}
Its eigenvalues are  \(\kappa_n(\ell)
=
P\frac{e^{-\ell^2\omega_n^2}}{Z_\ell}\),
and the normalization preserves both total power and pointwise
variance,
\[
\operatorname{tr}C_\ell=P,
\qquad
\sum_{n\in\mathbb Z}\kappa_n(\ell)|e_n(s)|^2=v_0.
\]
Increasing \(\ell\) therefore redistributes a fixed source power
toward lower arclength frequencies. The following proposition
makes this precise.

\begin{proposition}
\label{prop:arclength-spectral-concentration}
If \(0<\ell_1<\ell_2\), then
\(
C_{\ell_1}\succeq_{\rm div}C_{\ell_2}.
\)
Consequently, for every \(\beta>0\),
\(
\mathcal R_{\rm src}(C_{\ell_1},\beta)
\ge
\mathcal R_{\rm src}(C_{\ell_2},\beta).
\)
\end{proposition}

\begin{proof}
Set \(t=\ell^2\) and order the squared frequencies as
\(a_1\le a_2\le\cdots\), including multiplicities.
The corresponding decreasing covariance eigenvalues are
\[
\kappa_j(t)
=
P\frac{e^{-ta_j}}{\sum_i e^{-ta_i}}.
\]
For every \(M\ge1\), differentiation gives
\[
\frac{d}{dt}\sum_{i=1}^{M}\kappa_i(t)
=
\frac1P
\sum_{i=1}^{M}\sum_{j>M}
\kappa_i(t)\kappa_j(t)(a_j-a_i)
\ge0.
\]
Thus every leading partial sum increases while the total sum
remains \(P\), proving spectral concentration.
The effective-rank inequality follows from
Proposition~\ref{thm:correlation-effective-dimension}.
\end{proof}

	\bibliography{bib.bib}
	\bibliographystyle{abbrv} 
\end{document}